\documentclass[12pt]{article}
  \usepackage{oldgerm,extarrows,amsthm,thmtools,remark,color,graphicx,hyperref} 
\usepackage{amssymb}
  \usepackage[all]{xy}
  \usepackage[shortlabels]{enumitem}
  \usepackage[twoside,top=1.5in,bottom=1.5in,left=1in,right=1in,headheight=16pt,headsep=30pt,footskip=40pt]{geometry}
  
\def\cx{{\textstyle{\cdot}}}
\newcommand{\ep}{\epsilon}
\newcommand{\ov}{\overline}
\newcommand{\id}{{\rm id}}
\newcommand{\und}{\underline}
\newcommand{\ot}{\otimes}
\def\mex#1{e^{#1}}
\newcommand{\tr}{\mathbf {tr}}
\newcommand{\ie}{\textit{i.e.}}
\newcommand{\sone}{\mathbf S^1}
\newcommand{\tp}{{\rm top}}
\newcommand{\lag}{{\rm log}}
\newcommand{\et}{{\rm\acute et}}
\newcommand{\ket}{{\rm k\et}}
\newcommand{\g}{{\rm gp}}
\newcommand{\an}{{\rm an}}
\newcommand{\ch}{\mathop{\rm ch}\nolimits}
\newcommand{\cLL}{\mathcal{L}}
\newcommand{\cA}{\mathcal{A}}
\newcommand{\cF}{\mathcal{F}}
\newcommand{\cM}{\mathcal{M}}
\newcommand{\Hh}{\mathcal H}
\newcommand{\cQ}{{\cal Q}}
\newcommand{\cK}{\mathcal{K}}

\newcommand{\pee}{\textfrak{p}}
\newcommand{\qu}{\textfrak{q}}

\newcommand{\bc}{{\bf C}}
\newcommand{\br}{{\bf R}}
\newcommand{\bq}{{\bf Q}}
\newcommand{\bn}{{\bf N}}
\newcommand{\bz}{{\bf Z}}
\newcommand{\I}{{\bf I}}
\newcommand{\R}{{\rm R}}

\def\lt#1{\mathop{\mathsf{T}_{#1}}}

\def\lv#1{\mathop{\mathsf{V}_{#1}}}
\def\lr#1{\mathop{\mathsf{R}_{#1}}}
\def\li#1{\mathop{\mathsf{I}_{#1}}}

\def\varprojlim{\mathop{\vtop{\ialign{$##$\cr
 \hfil{\fam0lim}\hfil\cr\noalign{\nointerlineskip}%
 {\leftarrow}\mkern-6mu\cleaders\hbox{$\mkern-2mu{-}\mkern-2mu$}\hfill
 \mkern-6mu{-}\cr\noalign{\nointerlineskip\kern-.2326ex}\cr}}}}
\def\dirlim{\mathop{\vtop{\ialign{$##$\cr
 \hfil{\fam0lim}\hfil\cr\noalign{\nointerlineskip}%
 {-}\mkern-6mu\cleaders\hbox{$\mkern-2mu{-}\mkern-2mu$}\hfill
 \mkern-6mu{\to}\cr\noalign{\nointerlineskip\kern-.2326ex}\cr}}}\nolimits}
\newcommand{\spec}{\mathop{\rm Spec}\nolimits}
\newcommand{\Cyl}{{\rm Cyl}}
\newcommand{\Ker}{\mathop{\rm Ker}\nolimits}
\newcommand{\cok}{\mathop{\rm Cok}\nolimits}
\newcommand{\Cok}{\mathop{\rm Cok}\nolimits}
\newcommand{\im}{\mathop{\rm Im}\nolimits}
\def\angles#1#2{\langle{#1},{#2}\rangle}
\newcommand{\gr}{\mathop{\rm Gr}\nolimits}
\newcommand{\Hom}{\mathop{\rm Hom}\nolimits}
\newcommand{\Hhom}{{\mathcal{H}{\rm om}}} 
\newcommand{\Ext}{\mathop{\rm Ext}\nolimits}

\renewcommand{\bigwedge}{\mbox{\large $\wedge$}}
\newcommand{\Div}{\mathop{\rm Div }\nolimits}
\newcommand{\Kos}{\mathop{\rm Kos}\nolimits}
\newcommand{\Con}{\mathop{\rm Cone}\nolimits}

\def\oh#1{{\cal O}_{#1}}
\def\toh#1{{\tilde{\mathcal{O}}_{#1}}}

\newcommand{\ocM}{{\ov \cM}}
\def\ms#1{\mathop{\und{\mathsf{A}}_{{#1}}}}  
\def\mss#1{\mathop{\und{\mathsf{A}}^*_{#1}}}
\def\ls#1{\mathop{\mathsf{A}_{#1}}}
\def\lslag#1{\mathop{\mathsf{A^\lag_{#1}}}}
\def\tlslag#1{\mathop{\mathsf{\tilde A^\lag_{#1}}}}
\def\lsan#1{\mathop{\mathsf{A}^\an_{#1}}}
\def\lL#1{\mathop{\mathsf{L}_{#1}}}
\def\b#1#2{({#1}|{#2})}
\newcommand{\ttau}{{\tilde \tau}}
\newcommand{\tX}{{\tilde X}}

\newcommand{\rTo}[1]{\xlongrightarrow{#1}}
\newcommand{\To}{\longrightarrow}
\newcommand{\tto}[1]{\xrightarrow{#1}}  
\newcommand*\isommap{\xrightarrow{\raisebox{-0.2 em}{\smash{\ensuremath{\sim}}}}}
\newcommand*\isomlong{\xlongrightarrow{\raisebox{-0.2 em}{\smash{\ensuremath{\sim}}}}}
\newcommand*\isomlongleft{\xlongleftarrow{\raisebox{-0.2 em}{\smash{\ensuremath{\sim}}}}}

\declaretheoremstyle[bodyfont=\normalfont\slshape]{slanted}
\theoremstyle{slanted}
\newtheorem{theorem}{Theorem}[subsection]
\newtheorem{definition}[theorem]{Definition}
\newtheorem{corollary}[theorem]{Corollary}
\newtheorem{proposition}[theorem]{Proposition}
\newtheorem{lemma}[theorem]{Lemma}
\newtheorem{claim}[theorem]{Claim}

\newtheorem{conjecture}[theorem]{Conjecture}
\theoremstyle{definition}

\newtheorem{remark}[theorem]{Remark}

\newcommand{\xx}{{\ov x}}

\newcommand{\cJ}{\mathcal{J}}

  \numberwithin{equation}{subsection}
\numberwithin{theorem}{subsection} 

  \author{
    Piotr Achinger
    \thanks{
      Institute of Mathematics, Polish Academy of Sciences, ul. \'Sniadeckich 8, 00-656 Warszawa, Poland.  \newline 
      The first author was supported by NCN OPUS grant number UMO-2015/17/B/ST1/02634. 
    } 
    \and 
    Arthur Ogus
    \thanks{
      Department of Mathematics,  University of California,  Berkeley,  CA 94720, USA
    }
  }
  
  \title{Monodromy and Log Geometry}
  
\begin{document}

\maketitle

\abstract{A now classical construction due to Kato and Nakayama attaches a
topological space (the  ``Betti realization'') to a log scheme over $\bc$.  We 
show that in the case of a log smooth degeneration  over the standard log disc, this  construction allows one to recover the topology of the germ of  the family  from the  log special fiber alone.   
We go on to  give combinatorial formulas for the monodromy and the $d_2$ differentials acting on the nearby cycle complex in terms of the log structures. We also provide variants of these results for the Kummer \'etale topology. In the case of curves, these data are essentially equivalent to those 
encoded by the dual graph of a semistable degeneration, including the
monodromy pairing and the  Picard-Lefschetz formula. 
\tableofcontents

\section{Introduction} 
\label{intro.s}

Log geometry was introduced with the purpose of studying compactification and degeneration in a wide context of geometric and arithmetic situations.  For example, moduli problems usually give rise to spaces $U$ which are not compact, and it is often desirable to construct an understandable  compactification $X$ of $U$. Typically  the points of $D:=X\setminus U$ correspond to ``degenerate but decorated'' versions of the objects classified by points of $U$. 
In classical language, one keeps track of the difference between $X$ and $U$ by remembering the 
sheaf of functions on $X$ which vanish on $D$, a sheaf of ideals  in
$\oh X$.  Log geometry takes the complementary point of view, encoding
instead the sheaf  functions on $X$ which become invertible on $U$, a
sheaf of multiplicative  monoids in $\oh X$.  In general, a  {\em log
  scheme} is a scheme  $X$ endowed with a homomorphism of  sheaves of
commutative monoids $\alpha_X \colon \cM_X \to \oh X$; it is
convenient to also require that the induced map $\alpha_X^{-1}(\oh
X^*) \to \oh X^*$ be an isomorphism. Thus there is a natural ``exact sequence'':
\[ 0 \To \oh X^* \To \cM_X \To \ocM_X \To 0,\]
where the quotient is a sheaf of monoids which is essentially
combinatorial in nature.
 The resulting formalism allows
one to study the properties of $U$ locally along the complement $D$,
and in a relative situation, provides a very appealing  picture of the
theory of nearby cycles.  Furthermore, log structures behave well under base change,
and the log structure induced on  $D$  can often be related to the ``decoration''  needed to define the compactified moduli problem represented by $X$. 

In the complex analytic context, a construction~\cite{kkcn.lblelsc} of Kato and Nakayama gives a key insight into the {working of} log geometry. Functorially associated to any fine log analytic space $X$ is a topological space $X_\lag$, together with a natural proper and surjective continuous map $\tau_X \colon X_\lag \to X_\tp$, where $X_\tp$ is the topological space underlying $X$.   For each point  $x$ of  $X_\tp$, the fiber $\tau_X^{-1}(x)$ is naturally a torsor under $\Hom(\ov \cM_{X,x}, \sone)$. 
The  morphism  $\tau_X$ fits into a commutative diagram:
\[ 
  \xymatrix{
    & X_\lag \ar[d]^{\tau_X} \\
    X^*_\tp  \ar[ur]^{j_\lag} \ar[r]_{j_\tp} & X_\tp, 
  } 
\]
where $X^*$ is the open set on which the log structure is
  trivial and $j  \colon X^* \to X$ is the inclusion.  If the log scheme $X$ is (logarithmically) smooth over $\bc$, then the morphism $j_\lag$ is aspheric~\cite[3.1.2]{o.lrhc}, and in particular it induces an equivalence between the categories of locally constant sheaves on $X^*_\tp$ and on $X_\lag$.  Thus $\tau_X$ can be viewed as a compactification of the open immersion $j$; it has the advantage of preserving the local homotopy theory of $X^*$.  In particular, the behavior of a locally constant sheaf $\cF$  on $X_\lag$ can be studied locally over points of $X \setminus X^*$,  a very agreeable way of investigating local monodromy. 

We shall apply the above philosophy to study the behavior of a morphism $f \colon X \to Y$ of fine saturated log analytic spaces. Our goal is to exploit the log structures of $X$ and $Y$ to describe the topological behaviour
 of $f$ locally in a neighborhood of a point $y$ of $Y$, especially
 when $y$ is a point over which $f$ is smooth in the logarithmic sense
 but singular in the classical sense. The philosophy of log geometry
 suggests that (a~large part of) this topology can be computed just
 from the log fiber $X_y \to y$. For example, we show  that if $Y$ is
 a standard log disc and $X \to Y$ is smooth, proper, and vertical,
 then the
 {germs of $X_\tp$ and $Y_\tp$ are homeomorphic to the (open)
 mapping cylinders of the maps $\tau_{X_y} \colon X_{y,\lag} \to
 X_{y,\tp}$ and $\tau_y \colon y_\lag \to y_\tp$ respectively,
compatibly with the map $f_\tp \colon X_\tp \to Y_\tp$.}
 (See Proposition~\ref{discr.t}
 for  a precise statement and Conjecture~\ref{noglob.c} for a hoped for generalization.) Furthermore, it is shown in \cite[8.5]{ikn.qlrhc} that, in the above context, the classical  complex of nearby cycles $\Psi_{X/Y}$ on $X_{y,\tp}$ can be computed directly from the morphism  of log spaces $X_y \to y$, and in fact can be identified with (a~relative version of) $R\tau_{X_y*}(\bz)$.  (See \S 4 for the precise statement.) In particular,
the map $X_{y,\lag} \to y_ \lag$ serves as an ``asymptotic approximation'' to 
 the map $X \to Y$  near~$y$.

With the above motivation in mind, we devote our main attention to the study of a  morphism $f  \colon X  \to S$, where $X$ is a fine saturated log analytic space and $S$ is the  fine saturated split log point  associated to a fine sharp monoid $P$.  To emphasize the geometric point of view, we work mainly in the context of complex analytic geometry, describing the \'etale analogs of our main results in Section~\ref{etale.s}.
We assume that $f$ is saturated; this implies that the homomorphism  $f^\flat \colon P^\g= \ocM_S^\g \to \ocM_X^\g$ is injective and has torsion-free cokernel.  The map $X_\lag \to S_\lag$ is a topological fibration, trivial over the universal cover $\tilde S_\lag$ of $S_\lag$, and the cohomology of $\tilde X_\lag:= X_\lag \times_{S_\lag}  \tilde S_\lag $ is isomorphic to the cohomology  of a fiber.  The fundamental group $\li P$ of $S_\lag$ is canonically isomorphic to $\Hom (P^\g,\bz(1))$ and acts naturally on this cohomology and on the ``nearby cycle complex'' $\Psi_{X/S} := \R\ttau_{X*}(\bz)$, where $\ttau_{X} \colon  \tilde X_\lag \to X_\tp$ is the natural map. This situation is illustrated by the diagram
\[
  \xymatrix{
    \tilde X_\lag\ar[r]\ar[d]\ar@/^2em/[rr]^{\ttau_X} & X_\lag \ar[d] \ar[r]_{\tau_X} & X_\tp \ar[d] \\
    \tilde S_\lag \ar[r] & S_\lag \ar[r] & S_\tp = {\rm pt.}
  }
\]

Our first observation is that if $X/S$ is (log) smooth,
then  $X/\bc$  becomes (log) smooth
when  $X$ is endowed with the idealized log structure induced from the maximal ideal of $P$.
Theorem~\ref{idealcut.t} 
 shows that the normalization of a smooth and reduced idealized log scheme
can be endowed with  a natural ``compactifying'' log structure  which makes it smooth (without idealized structure).
This construction gives  a canonical way of cutting our $X$ into pieces, each of whose
Betti realizations is a family of manifolds with boundary, canonically trivialized  over $S_\lag$.

We then turn to our main goal, which is 
 to describe the topology of $\tilde X_\lag$, together with the monodromy action, directly in terms of log geometry.  
We  use the exact sequences:
\begin{equation} \label{logks.e}
  0 \To \ocM_S^\g \To \ocM_X^\g \To \cM_{X/S}^\g\To 0
\end{equation}
(``log Kodaira--Spencer''), and
\begin{equation} \label{logch.e}
  0 \To \bz(1)  \To  \oh X \rTo{\exp}  \cM_{X/P}^\g \To  \cM_{X/S}^ \g \To 0,
\end{equation}
(``log Chern class''), where $\cM_{X/P} := \cM_X/P$,
(the quotient in the category of sheaves of monoids).
The sequence (\ref{logch.e}) is obtained by splicing together the two exact sequences:
  \begin{equation} \label{eq:6}
 0 \To \bz(1) \To \oh X \rTo{\exp} \oh X^* \To 0
  \end{equation}
and
\begin{equation} \label{xsmp.e}
0 \To \oh X^* \To \cM^\g_{X/P} \To \cM_{X/S}^\g \To 0. 
\end{equation}
If $\ell$ is a global section of $\cM_{X/S}^\g$, its inverse image in $\cM_{X/P}^\g$ is an $\oh X^*$-torsor, 
which defines an invertible sheaf $\cLL_\ell$ on $X$.  The Chern class $c_1(\cLL_\ell) \in H^2(X, \bz(1))$
is the image of $\ell$ under the morphism
$H^0(X, \cM_{X/S}) \to H^2(X,\bz(1))$ defined by (\ref{logch.e}).}

The spectral sequence of nearby cycles reads:
\[
  E_2^{p,q} = H^p(X_\tp, \Psi^q_{X/S}) \quad \Rightarrow \quad H^{p+q}(\tilde X_\lag, \bz),
\]
where $\Psi^q_{X/S}$ is the $q$-th cohomology sheaf of the nearby cycle complex $\Psi_{X/S}$.
By (a relative version of) a theorem of Kato and Nakayama~\cite[1.5]{ikn.qlrhc}, there are natural isomorphisms:
\begin{equation}\label{ikn5.e}
  \sigma^q_{X/S} \colon \bigwedge^q \cM^\g_{X/S}(-q)  \isomlong \Psi^q_{X/S}. 
\end{equation}
It follows that  the action of each $\gamma \in I_P$ on $\Psi^q_{X/S}$
is trivial, and hence it is also trivial on the graded groups
$E_\infty^{p,q}$ associated to the filtration $F$ of the abutment
$H^{p+q}(\tilde X_\lag , \bz)$. Then $\gamma - \id$ maps
$F^pH^{p+q}(\tX_\lag,\bz) $ to $F^{p+1}H^{p+q}(\tX_\lag,\bz)$ and
induces a map
\begin{equation}
  \label{Ngamma.e}
N_\gamma \colon E_\infty^{p,q}  \To E_\infty^{p+1,q-1}.
\end{equation}

We explain in  Theorem~\ref{monthm.t} that (a derived category version of)  this map is given by ``cup product'' with the extension class in $\Ext^1( \Psi^1_{X/S},\bz) \cong \Ext^1(\cM_{X/S}^\g,\bz(1))$ obtained  from the pushout of the  log Kodaira-Spencer extension (\ref{logks.e}) along ${\gamma
 \in \Hom(P^\g, \bz(1))}$.     We present two proofs: the first, 
which works only in the smooth case and with \mbox{$\bc$-coefficients}, is an easy argument  based on a logarithmic construction of the Steenbrink complex; 
the second   uses more complicated homological algebra techniques to prove the result with \mbox{$\bz$-coefficients}.

We also give a logarithmic formula for the $d_2$ differentials of the nearby cycle spectral sequence. Thanks to formula~(\ref{ikn5.e}), these differentials
can be interpreted as maps:
\[
  H^p(X_\tp, \bigwedge^q \cM_{X/S}^\g(-q)) \To H^{p+2}(X_\tp, \bigwedge^{q-1}\cM_{X/S}^\g(1-q)).
\]
Theorem~\ref{monthm.t} shows that these maps are obtained by cup-product with the derived  morphism $\cM_{X/S}^\g \to \bz(1)[2]$ obtained from the log Chern class sequence (\ref{logch.e}), up to a factor of $q!$.  We do not have a formula for the higher differentials, but recall from~\cite[2.4.4]{ill.atml} that, in the case of a projective semistable reduction with smooth irreducible components,  these higher differentials vanish.

To illustrate  these techniques, we study the case in
which $X/S$ is a  smooth log  curve  over the standard log point.  In this case it is very easy to interpret our formulas explicitly in terms of the combinatorial data included in the  ``dual graph'' which is typically attached to the nodal curve $\und X$ underlying $X$.  The log structure provides some  extra information when $X/S$ is  log smooth but non-semistable.  In particular, we recover the classical Picard--Lefschetz formula, and we show that the $d_2$ differential in the nearby-cycle spectral sequence coincides with the differential in the chain complex computing the homology of the dual graph.  

For clarity of exposition, we focus mainly on the complex analytic
setting.  However, one of the main strengths of  log geometry is the
bridge it provides between analysis and algebra and between Betti,
\'etale, and de~Rham cohomologies.  For the sake of arithmetic
applications, we therefore also provide a sketch of how to formulate
and prove analogs of our results in the context of the Kummer \'etale
topology.  The case of $p$-adic cohomology looks more challenging at
present.  

Both authors wish to express their gratitude to the I.H.E.S., which
provided outstanding hospitality and a rich scientific environment
during which much of this work was carried out.  Discussions with Luc
Illusie, Ofer Gabber, Dennis Gaitsgory, Vivek Shende, and Martin Olsson were especially valuable.

\section{Homological preliminaries}

In this section, after reviewing some standard material in \S\ref{ss:homalg-notation}, we provide a few results in homological algebra which will be important in our study of the nearby cycle complex $\Psi_{X/S}$ together with its multiplicative structure and the monodromy action of the group~$\li P$.

\subsection{Notation and conventions}
\label{ss:homalg-notation}

We follow the conventions of \cite{bbm.tdc}  with regard to homological algebra, particularly when it comes to signs. For simplicity, we shall work in the abelian category $\cA$ of sheaves of modules on a ringed topological space (or more generally a ringed topos) $(X, A_X)$. Readers will gather from our exposition that  keeping track of  signs
presented   a  considerable challenge.

{\bf Shifts, cones, and distinguished triangles.} If $A = (A^n, d^n:A^n\to A^{n+1})$ is a~complex in an abelian category $\cA$, the  \textit{shift} $A[k]$  of $A$ by an integer $k$ is the complex $(A^{n+k}, (-1)^k d^{n+k})$. We shall use the  canonical identification $\Hh^n(A[k]) = \Hh^{n+k}(A)$ induced by the identity on $A^{n+k}$. If $u:A\to B$ is a morphism of complexes, its shift $f[k]:A[k]\to B[k]$ is given by $f^{n+k}:A^{n+k}\to B^{n+k}$ in degree $n$. 

The \textit{mapping cone}, written
 $\Con(u)$  or $C(u)$,  is the complex with \[C(u)^n := B^n \oplus A^{n+1}\] with differential $d(b,a) = (db +u(a), -da)$,
which comes with a sequence of maps of complexes:
\begin{equation} \label{eqn:distinguished} 
  A \rTo u B \rTo i C(u) \rTo {-p} A[1]
\end{equation}
where $i(b) := (b,0)$ and $p(b,a) := a$. This sign convention is used
 in \cite{bbm.tdc} but  differs from  the convention used in \cite[Chapter I]{kash.sm}
and many other authors. 
A \textit{triangle} is a sequence of maps in $D(\cA)$ of the form $A\tto{u} B\tto{v}C\tto{w} A[1]$ (abbreviated as $(u,v,w)$).
A triangle $(u,v,w)$ is \textit{distinguished} if it is isomorphic in the derived category to a triangle of the form \eqref{eqn:distinguished}. Then $(u,v,w)$ is distinguished if and only if $(v,w,-u[1])$ is distinguished. More generally, if $(u,v,w)$ is distinguished, so is $((-1)^k u[k], (-1)^k v[k], (-1)^k w[k])\cong (u[k],v[k],(-1)^k w[k])$ for any $k\in \bz$.



{\bf Total complex and tensor product.} Given a double complex 
\[A= (A^{p,q}, d^{p,q}_h:A^{p,q}\to A^{p+1,q}, d^{p, q}_v:A^{p,q}\to A^{p, q+1})\]
 in $\cA$, its \emph{total complex} is the complex ${\rm Tot}(A) = (\bigoplus_{p+q=n} A^{p,q}, d^n)$ where $d^n$ is given by $d^{p,q}_h + (-1)^p d^{p,q}_v$ on $A^{p,q}$,
 {so that the differentials form commutative squares.} The tensor product $A\otimes B$ of two complexes is by definition the total complex of the double complex $(A^p\otimes B^q, d_A^p\otimes \id, \id\otimes d_B^q)$. Note that the shift functor $(-)[k]$ equals $A_X[k]\otimes (-)$, while $(-)\otimes A_X[k]$ is the `naive shift,' that is, shift without sign change. Moreover, the cone $C(u)$ of a map $u:A\to B$ is the total complex of the double complex $[A\tto{u} B]$ where $B$ is put in the zeroth column. 

{\bf Truncation functors.} We use the truncation functors $\tau_{\leq q}$ and $\tau_{\geq q}$ (c.f. 
\cite[Exemple 1.3.2 (i)]{bbd.fperv} or \cite[(1.3.12)--(1.3.13), p. 33]{kash.sm}  on the category of complexes of sheaves on $X$:
\begin{eqnarray*}
  \tau_{\leq q} K = \ [\cdots \to K^{q-1} \to &\Ker(d^q)& \to 0 \to \cdots]\\
  \tau_{\geq q} K\ \  = \ \ \  \ \ \ [\cdots \to 0 \to &\cok(d^{q-1}) &\to K^{q+1} \to \cdots].
\end{eqnarray*}
These functors descend to the derived category $D(X)$, although they 
do not preserve distinguished triangles. 
 For a pair of integers $a\leq b$, we write $\tau_{[a, b]} = \tau_{\geq a} \tau_{\leq b} = \tau_{\leq b}\tau_{\geq a}$ and $\tau_{[a, b)} = \tau_{[a, b-1]}$.
For example,  $\tau_{[q, q]} K = H^q (K)[-q]$. 
\begin{proposition}
\label{truncation-triangle.p}
\label{tabc.p} For each triple of integers $(a,b,c)$  with 
 $a< b<c$,  and each complex $K$,  there is a functorial   distinguished triangle:
\begin{equation} \label{truncation-triangle.e}
  \tau_{[a, b)}(K) \To \tau_{[a, c)}(K) \To \tau_{[b, c)}(K) \rTo\delta \tau_{[a, b)}[1].
\end{equation}
The  map $\delta$ 
above is the unique 
morphism  making the triangle distinguished.
\end{proposition}
\begin{proof}
The natural map of complexes $\tau_{[a, b)}(K) \to \tau_{[a, c)}(K)$ is injective with cokernel 
\[ 
  C := [ \cdots \to 0\to K^{b-1}/\ker d^{b-1}\to K^b\to \ldots \to \ker d^{c-1} \to 0\to\cdots  ]
\]
and the evident map $C\to \tau_{[b, c)}(K)$ is a quasi-isomorphism commuting with the maps from $\tau_{[a,c)}(K)$. This way we obtain the distinguished triangle \eqref{truncation-triangle.e}. For the uniqueness, observe that given two such maps $\delta,\delta'$, there is  a map $\zeta:\tau_{[b, c)}(K)\to \tau_{[b, c)}(K)$ completing $(\id, \id)$ to a morphism of distinguished  triangles:
\[ 
  \xymatrix{
    \tau_{[a, b)}(K) \ar@{=}[d] \ar[r] & \tau_{[a, c)}(K) \ar@{=}[d] \ar[r] & \tau_{[b, c)}(K) \ar@{.>}[d]^\zeta \ar[r]^{\delta} & \tau_{[a, b)}[1]. \ar@{=}[d] \\
    \tau_{[a, b)}(K)  \ar[r] & \tau_{[a, c)}(K)  \ar[r] & \tau_{[b, c)}(K)  \ar[r]_{\delta'} &  \tau_{[a, b)}[1]. \\    
  }
\]
Applying the functor $\tau_{[b, c)}$ to the middle square of the above
diagram, we see that $\zeta=\id$, and hence
that $\delta=\delta'$. 
\end{proof}

{\bf First order attachment maps.} 
If $K$ is a complex and $q\in \bz$, the distinguished  triangle \eqref{truncation-triangle.e} for $(a, b, c)=(q-1, q, q+1)$ is:
\[ 
\Hh^{q-1}(K)[1-q] \To \tau_{[q-1,q+1)}( K) \To \Hh^q(K)[-q] \rTo{\delta_K^q[-q]} \Hh^{q-1}(K)[2-q],
 \]
which  yields a `first order attachment morphism'
\begin{equation} \label{d2der.e}
  \delta_K^q \colon \Hh^q(K)\To \Hh^{q-1}(K)[2],
\end{equation}
embodying  the $d_2$ differential of the spectral sequence 
\[ 
  E_2^{p,q}= H^p(X,\Hh^q(K)) \quad \Rightarrow \quad H^{p+q}(X,K).
\]
Note that $\delta_K^q[-q]$ is the unique morphism making the triangle above distinguished.

We shall need the following result, stating that the maps $\delta$ form a 
``derivation in the derived category.''

\begin{lemma} \label{d2mult.l}
  Let $A$ and $B$ be complexes in the abelian category $\cA$,  and let $i$ and $j$ be  integers such that $\Hh^i(A)$ and $\Hh^j(B)$ are flat $A_X$-modules. Then the following diagram commutes: 
  \[ \quad
    \xymatrix@C=4em{
      \Hh^i(A) \ot  \Hh^j(B) \ar[rr]^-{\delta_A^i \ot 1+ (-1)^i\ot\delta_B^j} \ar[d] & &  \Hh^{i-1}(A)[2] \ot \Hh^{j}(B) \oplus \Hh^{i}(A)\ot \Hh^{j-1}(B)[2]  \ar[d] \\
      \Hh^{i+j}(A\ot B)  \ar[rr]_-{\delta^{i+j}_{A\ot B}} & & \Hh^{i+j-1}(A\ot B)[2]
    }
  \]

 Here $A\ot B$ denotes the derived tensor product.
\end{lemma}

\begin{proof}
Call the diagram in question $D(A, B)$ ($i$ and $j$ are fixed throughout).  We shall first prove that $D(A, B)$ commutes if $\Hh^q(A)=0$ for $q\neq i$. Recall that $\delta^j_B$ is the unique map such that the triangle
\[ 
  \Hh^{j-1}(B)[1-j]\To \tau_{[j-1,j]} B \To \Hh^j(B)[-j] \rTo{\delta^j_B[-j]} \Hh^{j-1}(B) [2-j]
\]
is distinguished. Applying $A_X[-i]\otimes (-) =(-)[-i]$, we see that $\delta^{i+j}_{B[-i]} = (-1)^i \delta^j_B[-i]$ under the identifications $\Hh^q(B) = \Hh^{q+i}(B[-i])$, $q=j-1, j$. This implies that $D(A_X[-i], B)$ commutes. Since $A\otimes(-) = \Hh^i(A)[-i]\otimes (-) = \Hh^i(A)\otimes (A_X[-i]\otimes (-))$, we see that $D(A, B)$ commutes as well.

Similarly if $\Hh^q(B)=0$ for $q\neq j$: $A\otimes A_X[-i]$ is the $(-i)$-th naive shift of $A$, preserving exactness, and we have $\delta^{i+j}_{A\ot A_X[-i]} = \delta^i_A[-j]$ (note that the effect of naive and usual shift on maps is `the same'), so $D(A, A_X[-j])$ commutes; again, so does $D(A, B)$. 

To treat the general case, note that $D(A, B)$, even if not commutative, is clearly a~functor of $A$ and $B$.  Let $A':= \tau_{\leq i} A$ and observe that the natural map $A' \to A$ induces isomorphisms on the objects in the top row of the diagrams $D(A', B)$ and $D(A, B)$.  Thus  $D(A, B)$ commutes if $D(A', B)$ does, and hence we may assume that $\Hh^q(A)=0$ for $q>i$. Analogously, we can assume that $\Hh^q(B)=0$ for $q>j$.  

Under these extra assumptions, the hypertor spectral sequence (see \cite[Proposition~6.3.2]{EGAIII2})
\[ 
  E_{pq}^2 = \bigoplus_{i'+j'=q}{\rm Tor}_{-p}(\Hh^{-i'}(A), \Hh^{-j'}(B)) 
  \quad \Rightarrow \quad
  \Hh^{-p-q}(A\ot B)
\]
shows that the vertical maps in $D(A, B)$ are isomorphisms. Let $u_{A, B}=({\rm right})^{-1}\circ ({\rm bottom})\circ ({\rm left})$ in $D(A, B)$, then $D(A, B)$ commutes if and only if $u_{A, B} = ({\rm top}): = \delta_A^i \ot 1+ (-1)^i\ot\delta_B^j$. The target of $u_{A, B}$ is a product of two terms $\Hh^{i-1}(A)[2] \ot \Hh^{j}(B)$ and $\Hh^{i}(A)\ot \Hh^{j-1}(B)[2]$; let us denote the two projections by $p_A$ and $p_B$.

Let us set $A''=\Hh^i(A)[-i]$; we know that $D(A'', B)$ commutes. This diagram reads:
\[ 
    \xymatrix@C=4em{
      \Hh^i(A) \ot  \Hh^j(B) \ar[rr]^-{ (-1)^i\ot\delta_B^j} \ar[d] & &  \Hh^{i}(A)\ot \Hh^{j-1}(B)[2]  \ar[d] \\
      \Hh^{i+j}(A\ot B)  \ar[rr]_-{\delta^{i+j}_{A'\ot B}} & & \Hh^{i+j-1}(A''\ot B)[2]
    }
  \]
and the vertical maps are isomorphisms. The map between the top-right corners of $D(A, B)$ and $D(A'', B)$ induced by the canonical map $A\to A''$ is the projection $p_B$. It follows that $p_B\circ u_{A, B} = (-1)^i\ot\delta^j_B$. 

Similarly, considering $B''=\Hh^j(B)[-j]$ and the canonical map $B\to B''$, and using the fact that $D(A, B'')$ commutes, we see  that $p_A\circ u_{A,B} = \delta^i_A\ot 1$. We conclude that $u_{A,B} = p_A\circ u_{A,B} + p_B\circ u_{A, B}
= \delta^i_A\ot 1 + (-1)^i\ot\delta^j_B$, as desired.
\end{proof}

{\bf Maps associated to short exact sequences.} Consider a short exact sequence of complexes
\begin{equation} \label{eqn:some-ses}
  0\To A\rTo{u} B \rTo{\pi} C \To 0. 
\end{equation}
The map $\tilde \pi \colon C(u) \to C$ sending $(b,a)$ to $\pi(b)$ is a quasi-isomorphism.

\begin{definition}\label{xiu.d}
In the above situation, $\xi_u \colon C \to A[1]$ is the morphism in the derived category $D(\cA)$ defined by
\[ \xi_u \colon C   \rTo{\tilde \pi^{-1}} C(u) \rTo {-p} A[1] .\]
We shall also refer to $\xi_u$ as the map corresponding to the short exact sequence \eqref{eqn:some-ses} (rather than the injection $u$).
\end{definition}

Thus the triangle 
\begin{equation}\label{xiudist.e}
 A \rTo u  B \rTo \pi C \rTo {\xi_u} A[1]
\end{equation}
is distinguished, and the map $\Hh^q(\xi)\colon \Hh^q(C) \to \Hh^q(A[1])=\Hh^{q+1}(A)$ agrees with the map defined by the standard diagram chase in the snake lemma. Moreover, $\xi_{-u} = -\xi_u$. 

In the special case when $A$ and $B$ are objects of $\cA$ concentrated in a single degree $q$, the map $\xi_u$ is the unique map making the triangle \eqref{xiudist.e} distinguished \cite[10.1.11]{kash.sm}. 

\subsection{Exterior powers and Koszul complexes}
\label{ss:koszul}
Let us first review some relevant facts about exterior and
symmetric powers.  Recall that if $E$ is a flat $A_X$-module and
$q\geq 0$, then the exterior power $\bigwedge^q E$,  the symmetric power $S^q E$,
and the divided power $\Gamma^q(E)$ modules are  again flat.
 For $i+j=q$, there are natural multiplication and comultiplication transformations:
\[ 
  {\mu}:\bigwedge^i E \ot \bigwedge^{j} E\To \bigwedge^q E
  \quad \text{and} \quad
  \eta: \bigwedge^q E \To \bigwedge^i E \ot \bigwedge^{j} E,
\]
\[ 
  {\mu}:S^i  E\ot S^{j} E\To S^q E
  \quad \text{and} \quad
  \eta: S^qE \To S^i E \ot S^{j} E,
\]
  \[{\mu}:\Gamma^i  E\ot \Gamma^{j} E\To \Gamma^q E
  \quad \text{and} \quad
  \eta: \Gamma^qE \To \Gamma^i E\ot \Gamma^{j} E.
\]
 We shall only use the maps $\eta$ with $i=1$. In this case they  are  given by the formulas
 \begin{eqnarray*}
    \eta(x_1\wedge \cdots \wedge x_q) &=& \sum_i (-1)^{i-1} x_i \ot x_1 \wedge\cdots \hat x_i \cdots \wedge x_q\\
   \eta(x_1 \cdots x_q) &=& \sum_i x_i \ot x_1 \cdots \hat x_i \cdots x_q\\
  \eta(x_1^{[q_1]} \cdots x_n^{[q_n]}) & = & \sum_i x_i\ot  x_1^{[q_1]} \cdots x_i^{[q_i-1]} \cdots x_n^{[q_n]} .
 \end{eqnarray*}
It follows that each composition
\begin{align*} 
  \bigwedge^q E \rTo{\eta} E &\ot  \bigwedge^{q-1} E \rTo{\mu} \bigwedge^q E \\
  S^q  E\rTo{\eta} E &\ot S^{q-1} E  \rTo{\mu} S^q E \\
  \Gamma^q  E\rTo{\eta} E &\ot \Gamma^{q-1} E  \rTo{\mu} \Gamma^q E
\end{align*}
is multiplication by $q$. Furthermore,  $\eta$ is  a derivation,  by
which  we mean that the following diagram commutes:
\[\xymatrix@C+=3em
{\bigwedge^i E\ot \bigwedge^j E\ar[r]^-{\eta \ot \id,\id\ot \eta}\ar[d]_\mu
   &(E \ot \bigwedge^{i-1}E \ot \bigwedge^j E) \oplus (\bigwedge^{i}E
   \ot E\ot\bigwedge^{j-1} E) \ar[d]^{\id\oplus t_i\ot \id} \cr
 \bigwedge^{i+j} E\ar[d]_\eta &(E \ot \bigwedge^{i-1}E \ot \bigwedge^j
 E) \oplus (E \ot    \bigwedge^{i}E \ot \bigwedge^{j-1}
 E)\ar[ld]^{\id\ot \mu,\id\ot \mu} \cr
E\ot \bigwedge^{i+j-1} E,
}\]
where $t_i\colon \bigwedge^i E \ot E \to E \ot \bigwedge^i E$ is
$(-1)^i$ times the commutativity isomorphism for tensor products.
The diagram for the symmetric and divided power products is similar (without the sign).

In fact, $\{\eta_q \colon \bigwedge^q E \to E \ot \bigwedge^{q-1}E : q \ge 1\}$
 is the unique derivation  such that ${\eta_1} = \id$, 
because  the  multiplication map $\mu$ is an
epimorphism.    This argument  fails in the derived category, and we will need another argument
which gives a slightly weaker result. To understand the context, let
$\alpha \colon E \to F$ be a morphism in $D(X)$,
where $E$ is  flat  and concentrated in degree zero, and for each $q$,
define $\alpha_q$ as the composition
\begin{equation}   \label{alphaq.d}
  \alpha_q \colon \bigwedge^q E \rTo{\eta} E \ot \bigwedge^{q-1} E \rTo{\alpha \ot \id }  F \ot \bigwedge^{q-1} E.
\end{equation}
 Then   the family
$\{ \alpha_q \colon \bigwedge^q E \to F \ot \bigwedge^{q-1} E: q \ge 1
\}$ is  derivation in $D(X)$, in the sense that the diagram:
\begin{equation}  \label{derivedder.e}
\xymatrix@C+=4em
{\bigwedge^i E\ot \bigwedge^j E\ar[r]^-{\alpha_i \ot \id,\id\ot \alpha_j}\ar[d]_\mu
   &(F \ot \bigwedge^{i-1}E \ot \bigwedge^j E) \oplus (\bigwedge^{i}E
   \ot F\ot\bigwedge^{j-1} E) \ar[d]^{\id\oplus t_i\ot \id} \cr
 \bigwedge^{i+j} E\ar[d]_{\alpha_{i+j}} &(F \ot \bigwedge^{i-1}E \ot \bigwedge^j
 E) \oplus (F \ot    \bigwedge^{i}E \ot \bigwedge^{j-1}
 E)\ar[ld]^{\id\ot \mu,\id\ot \mu} \cr
F\ot \bigwedge^{i+j-1} E
}\end{equation}
commutes. 
We shall see that this property almost determines the maps
$\alpha_q$. 

\begin{proposition}\label{alunique.p}
  Let $E$ be a flat $A_X$-module, let $F$ be an object of $D(X)$, 
 and let 
\[\{ \alpha'_j  \colon \bigwedge^j E \to F \ot \bigwedge^{j-1} E : j \ge 1 \}\] 
be a family of morphisms in $D(X)$.  Let   $\alpha = \alpha'_1 \colon
E  \to F$, let $\alpha_q$ be as in (\ref{alphaq.d}) for $q \ge 1$,  and  assume that   $q  \in \bz^+$ is such that for  $1 \le j
<q$,   the diagrams
\[\xymatrix@C+=3em
{ E\ot \bigwedge^{j} E\ar[r]^-{\alpha \ot \id,\id\ot \alpha'_{j}}\ar[d]_\mu
   &(F  \ot \bigwedge^j E) \oplus (E\ot F\ot\bigwedge^{j-1} E) \ar[d]^{\id\oplus t\ot \id} \cr
 \bigwedge^{j+1} E\ar[d]_{\alpha'_{j+1}} &(F\ot  \bigwedge^{j}  E) \oplus (F \ot   E \ot \bigwedge^{j-1} E)\ar[ld]^{\id,\id\ot \mu} \cr
F\ot \bigwedge^{j} E,
}\]
commute, where $t \colon E\ot F \to F\ot E$ is the negative of
the standard isomorphism. 
Then $q!\alpha'_q = q! \alpha_q$.
\end{proposition}
\begin{proof}
  The statement is vacuous for $q = 1$, and we proceed by induction on $q$.  
Let $\tau_F  := (\id, \id\ot \mu) \circ (\id \oplus t \ot\id)$ in  the diagram above.
 Then, setting $j = q-1$, 
we have the following commutative diagram:
\[\xymatrix@C+=3em{
 \bigwedge^q E\ar[r]^-\eta \ar[rd]_q &  E \ot \bigwedge^{q-1} E\ar[d]_{\mu}\ar[rr]^-{\alpha \ot \id,\id\ot \alpha'_{q-1}} & &(F\ot  \bigwedge^{q-1}  E)
 \oplus (E \ot  F \ot \bigwedge^{q-2} E)\ar[d]^-{\tau_F}  \cr
&  \bigwedge^{q} E \ar[rr]_{\alpha'_q}& &  F \ot \bigwedge^{q-1} E.
}\]
 In other words,
\[  q \alpha'_q = \tau_F \circ (\alpha\ot \id, \id \ot  \alpha'_{q-1})\circ \eta.\]
Similarly, 
\[  q \alpha_q = \tau_F \circ (\alpha\ot\id,\id \ot \alpha_{q-1})\circ \eta.\]
Then using the induction hypothesis, we can conclude:
\begin{align*}
  q! \alpha'_q &=  \tau_F \circ\left ( (q-1)!\alpha\ot \id, \id \ot (q-1)!\alpha'_{q-1}\right)\circ\eta\\
 &=  \tau_F \circ\left ( (q-1)!\alpha\ot \id, \id \ot (q-1)!\alpha_{q-1}\right) \circ \eta \\
 & =  q!\alpha_q \qedhere
\end{align*}
\end{proof}

Next we discuss connecting homomorphisms, exterior powers, and Koszul complexes. Consider a short exact sequence of flat $A_X$-modules
\[ 
  0\To A\rTo{u} B\rTo{\pi} C\To 0,
\]
and the associated morphism $\xi=\xi_u:C\to A[1]$ (cf.~Definition~\ref{xiu.d}).  
The \textit{Koszul filtration}  is  the decreasing filtration of $\bigwedge^q B$ defined by:
\[
  K^i \bigwedge^q B = \im\left( \bigwedge^i A \otimes \bigwedge^{q-i} B\rTo{\wedge^i u\otimes \id} \bigwedge^i B \otimes
    \bigwedge^{q-i} B\rTo{\mu} \bigwedge^q B   \right).
\] 
There are canonical isomorphisms  
\begin{equation}  \label{grkos.e}
\bigwedge^i A\ot \bigwedge^{q-i} C \cong \gr_K^i(\bigwedge^q B)  
\end{equation}
We can use this  construction to give a convenient expression for the composed morphism
$\xi_q \colon  \bigwedge^q C \to A \ot \bigwedge^{q-1} C[1]$ defined in Equation~\eqref{alphaq.d} 
above. 

\begin{proposition}\label{cupxi.p}
Let $0 \to A \tto u   B \tto \pi   C \to 0$ be an exact sequence
of flat \mbox{$A_X$-modules}, with corresponding morphism
$\xi := \xi_u \colon C \to A[1]$ in $D(X)$.   For each $q \in \bn$, let $K^\cx$ be
the Koszul filtration on $\bigwedge^q B$ defined by the inclusion $u:A
\to B$ and consider the exact sequence
\[    0 \rTo{}  A\ot \bigwedge^{q-1} C \rTo {u_q}   \bigwedge^q B / K^2 \bigwedge^q B\rTo{\pi_q}  \bigwedge^q C \rTo{}  0.\]
obtained from the filtration $K$ and the isomorphisms~(\ref{grkos.e}) above.
Then 
\[ \xi_{u_q} = \xi_q :=(\xi \ot \id) \circ \eta \colon \bigwedge^q C \rTo{} C \ot \bigwedge^{q-1} C\rTo{} A \ot \bigwedge^{q-1} C[1].\]  \end{proposition}
\begin{proof}
Observe that the composition
\[ \bigwedge^q B \rTo \eta B \ot \bigwedge^{q-1} B \rTo{\id \ot \pi}B \ot \bigwedge^{q-1} C\]
annihilates $K^2\bigwedge^q B$. Then we find 
the following commutative diagram, in
which the rows are exact.
\[\xymatrix{  
0 \ar[r] & A \ar[r] \ar@{=}[d] \ot \bigwedge^{q-1}C & \bigwedge^qB/K^2 \bigwedge^q B \ar[r] \ar[d]^\zeta  & \bigwedge^qC \ar[r]\ar[d]^\eta & 0\cr
0 \ar[r] & A \ot \bigwedge^{q-1} C\ar[r] & B \ot \bigwedge^{q-1} C\ar[r]& C \ot \bigwedge^{q-1} C\ar[r] & 0
}\]
We consequently get a commutative diagram in $D(X)$:
\[
\begin{gathered}[b] 
\xymatrix{\bigwedge^q C \ar[r]^-{\xi_{u_q}}\ar[d]_\eta & A \ot \bigwedge^{q-1}  C[1]\ar@{=}[d] \cr
C\ot \bigwedge^{q-1} C \ar[r]_-{\xi \ot \id} & A \ot \bigwedge^{q-1} C[1].
}\\[-\dp\strutbox] 
\end{gathered}
\qedhere
\]
\end{proof}
Let us now recall the definition of the Koszul  complex of a homomorphism (see \cite[Chapitre I, 4.3.1.3]{ill.cctd} and \cite[1.2.4.2]{katosaito.cfb}).

\begin{definition} \label{kozdr.d}
  Let  $u\colon A\to B$ be
  homomorphism of $A_X$-modules, and let $q\geq 0$.
  Then  the $q$-th Koszul complex $\Kos^q(u)$ of $u$ 
 is the cochain complex whose $p$-th term is $\Gamma^{q-p}(A)\ot \bigwedge^p B$ and  with differential  
  \[
    d_{u,q}^{p} \colon
    \xymatrix{
      \Gamma^{q-p}(A) \ot   \bigwedge^p B \ar[r]^-{\eta\ot\id} &  \Gamma^{q-p-1}(A)\ot A\ot \bigwedge^p B  \ar[d]^{\id\ot u \ot \id} \\
      & \Gamma^{q-p-1}(A)\ot B\ot \bigwedge^p B  \ar[r]_-{\id\ot \mu} & \Gamma^{q-p-1}(A) \ot \bigwedge^{p+1} B.
    }
  \] 
\end{definition}

Observe that $\Kos^q(u)$ (treated as a chain complex) is
$\Lambda^q(u:A\to B)$ in the notation of \cite[1.2.4.2]{katosaito.cfb},
and is the total degree $q$ part of $\Kos^\cx(u)$ in the notation of
\cite[Chapitre I, 4.3.1.3]{ill.cctd}. If $A$ and $B$ are flat,
$\Kos^q(u)[-q]$ coincides with the derived exterior power of the
complex $[A\to B]$ (placed in degrees $-1$ and $0$),
cf. \cite[Corollary 1.2.7]{katosaito.cfb}.
Note that $\Kos^1(u)$ is the
complex  $[A \to B]$ in degrees $0$ and $1$, \ie, $\Kos^1(\theta) =
\Con(-\theta)[-1]$.  If $u= \id_A$, its Koszul
complex identifies with the divided power de Rham complex
of $\Gamma^\cx(A)$.  
In most of our applications,  $A_X$ will contain $\bq$, and we can and
shall identify $\Gamma^q(A)$ with $S^q(A)$, the $q$th symmetric power
of $A$.

We recall the following well-known result (cf. \cite[Lemma~1.4]{steen.levnc}):

\begin{proposition} \label{kozcoh.p}
Suppose that 
\[ 0 \rTo{} A \rTo u  B \rTo{} C \rTo{} 0\]
is an exact sequence of flat $A_X$-modules.
Then the  natural map:
   \[
   e_q\colon   \Kos^q(u)[q] \to \bigwedge^q C
   \]
   is a quasi-isomorphism.
\end{proposition} 
\begin{proof}
We include a proof
for the convenience of the reader. The last  nozero term
of the complex $\Kos^{q}(u)[q] $ is $\bigwedge^q B$ in degree $0$,
and the natural map $\bigwedge^q B \to \bigwedge^q C$ induces
the morphism $e_q$. 
 The Koszul filtration $K$  on
$\bigwedge^\cx B$ makes $\Kos^q(u)$ a filtered complex, with
\[K^i \Kos^q(u)^n := \Gamma^{q-n}(A) \ot K^i \bigwedge^n B.\]
Note that the differential $d$ of $\Kos^q(u)$ sends $K^i\Kos^q(u)$ to
$K^{i+1}\Kos^q(u)$.  Then the spectral sequence of  the filtered complex
$(\Kos^{q}(u), K)$ has
\[ E_1^{i,j} = H^{i+j}(\gr_K^i\Kos^q(u))= \gr_K^i \Kos^q(u)^{i+j} =\Gamma^{q-i-j}(A)\ot\bigwedge^i A \ot \bigwedge^j C, \]
and the complex
$(E_1^{\cx,j},d_1^{\cx,j})$ identifies with the complex
$\Kos(\id_A)^{q-j} \ot \bigwedge^j C[-j]$, up
to the sign of the differential.
This complex is acyclic unless $j =q$, in which  case it reduces to
the single term complex $\bigwedge^qC[-q]$.  It follows that the
map $e_q[-q] \colon \Kos^q(u) \to \bigwedge^q C[-q]$ is a quasi-isomorphism.
\end{proof}

The following technical result compares the various Koszul complexes
associated to~$u$.
\begin{proposition} \label{vwdr.p}
  Let $0 \to A \tto u   B \tto \pi   C \to 0$ be an exact sequence of flat \mbox{$A_X$-modules}. For each $q \ge 0$, let $K$ be the  Koszul filtration of
  $\bigwedge^q B$ induced by $u$,  and let 
\[u_q\colon A\ot \bigwedge^{q-1} C \To \bigwedge^q B/K^2\bigwedge^q
  B\]
 be as in Proposition~\ref{cupxi.p}.
  \begin{enumerate}
    \item There is a  natural commutative diagram of quasi-isomorphisms:
    \[ 
      \xymatrix{
        \Kos^q(u) \ar[r]^-{a_q}[q] \ar[dr]_{e_q} & \Con((-1)^q u_q) \ar[d]^{b_q} \\
        &  \bigwedge^q C.
      }
    \]
  \item There exist morphisms of complexes $c_q$ and $f_q$ as indicated below.
Each of these is a quasi-isomorphism of complexes, and the resulting
diagram is commutative.
    Hence there is a unique morphism $g_q$  in $D(X)$ making
the following diagram commute:
    \[
      \xymatrix{
        \Kos^q(u)[q] \ar[r]^-{a_q} \ar[d]_{c_q}\ar@/^2em/[rr]^{e_q}
        &\Con((-1)^qu_q) \ar[dr]^{f_q} 
         \ar[r]^{b_q} & \bigwedge^q C\ar@{.>}[d]^{g_q} \\
        A\ot  \Kos^{q-1}(u)[q] \ar[rr]_-{\id\ot e_{q-1}} & & A \ot  \bigwedge^{q-1} C[1],
      }
    \]
    \item  In the derived category $D(X)$, 
$g_q = (-1)^{q-1} \xi_{u_q}$, where 
$\xi_{u_q}$ is the morphism defined by $u_q$
as in Definition~\ref{xiu.d}.  Consequently,
$g_q$ is $(-1)^{q-1}$
times cup-product (on
      the left) with the morphism $\xi_u$ defined by $u$.

  \end{enumerate}
\end{proposition}
\begin{proof}
The vertical arrows in the following commutative diagram of complexes
are the obvious projections.  The first set  of these  defines the morphism
of complexes $a_q$ and  the second defines the morphism $b_q$. 
\[ 
  \xymatrix{
\Gamma^q (A)\ar[r] & \ldots \Gamma^{2} (A) \ot  \bigwedge^{q-2} B  \ar[r] \ar[d]^{a_q} & A\ot \bigwedge^{q-1} B \ar[d]^{a_q}\ar[r] & \bigwedge^q B \ar[d]^{a_q} \\
    & 0 \ar[r] & A \ot \bigwedge^{q-1} C  \ar[r]^{(-1)^q u_q} & \bigwedge^q B/K^2 \bigwedge^q B \ar[d]^{b_q} \\
    & & & \bigwedge^q C.
  }\]
Here the top row is placed in degrees $-q$ through $0$, and its 
differential is  the Koszul 
differerential multiplied by $(-1)^q$, and thus is the complex
$\Kos^q(u)[q]$.  The middle row is placed in degrees $-1$
and $0$, and hence is the mapping cone of $(-1)^qu_q$. 
Since the sequence
\[ 
     0 \to   A \ot \bigwedge^{q-1} C  \to  \bigwedge^q B/K^2\bigwedge^q B\to \bigwedge^q C \to 0
\]
is  exact, the map $b_q$ is a quasi-isomorphism.  We already observed in 
Proposition~\ref{kozcoh.p}
that $e_q$ is a quasi-isomorphism, and it follows that $a_q$ is also
a quasi-isomorphism. 
This proves statement (1) of the proposition. 

The morphism $f_q$ is defined by the usual projection
\[
  \Con((-1)^qu_q) = \left[  A\ot \bigwedge^{q-1} C\rTo{(-1)^q u_q} \bigwedge^q B / K^2 \bigwedge^q B \right] \rTo{p=(\id, 0)} A\ot \bigwedge^{q-1} C.
\]
The diagram  below commutes, with the sign shown,
because of the conventions in (\ref{eqn:distinguished}) and (\ref{xiu.d}).
\[\xymatrix{ 
\Con((-1)^qu_q) \ar[d]_p\ar[r]^-\sim & \bigwedge^q C \ar[ld]^{-\xi_{(-1)^q{u_q}}}\\
A \ot \bigwedge^{q-1} C
}\]


One checks easily that the  square below commutes, so that the vertical map
defines a morphism of complexes $ \Kos^q(u) \to
A\ot \Kos^{q-1}(u)$ whose  shift  is $c_q$.  
\[ 
  \xymatrix{
    \Gamma^{q-n}(A)\ot \bigwedge^nB \ar[r]^{d_u} \ar[d]_{\eta \ot \id} & \Gamma^{q-n-1}(A)  \ot \bigwedge^{n+1}B \ar[d]^{d_\id \ot \id } \\
    A \ot   \Gamma^{q-n-1}(A) \ot \bigwedge^nB \ar[r]_{\id \ot d_u} & A \ot  \Gamma^{q-n-2}(A)  \ot \bigwedge^{n+1}B
  }
\]
The diagram of  statement (2)  in degree $q-1$ is given  by the following obvious set of maps:
\[ 
  \xymatrix{
    A \ot \bigwedge^{q-1} B \ar[r]\ar[d]_\id  & A\ot \bigwedge^{q-1} C \ar[r] \ar[dr]^\id &  0 \\
    A \ot  \bigwedge^{q-1} B \ar[rr] & & A\ot \bigwedge^{q-1}C,
  }
\]
and in degree $q$  by:
\[ 
  \xymatrix{
    \bigwedge^q B \ar[r]\ar[d]  & \bigwedge^q B/K^2\bigwedge^q B \ar[r]\ar[dr] &  \bigwedge^q  C \\
    0 \ar[rr] & & 0.
  }
\]
This proves statement (2).  
%
It follows that
$g_q = -\xi_{(-1)^{q} u_q} = (-1)^{q-1} \xi_{u_q}$ and the rest of statement (3) then
follows from Proposition~\ref{cupxi.p}.
\end{proof}

\subsection{\texorpdfstring{$\tau$-unipotent maps in the derived category}{tau-unipotent maps in the derived category}}
\label{ss:zero-on-cohomology}


One frequently encounters unipotent automorphisms of objects, or more precisely, automorphisms  $\gamma$ of filtered objects $(C, F)$ which induce the identity on the associated graded object $\gr^\cx_F(C)$.  Then $\gamma- \id$ induces a map $\gr^\cx_F(C)  \to \gr_F^{\cx-1} C$ which serves as an approximation to $\gamma$.  For example, if $\gamma$ is an automorphism of a complex $C$ which acts as the identity on its cohomology, 
this construction can be applied to the
 canonical filtration $\tau_\le$ of $C$ and  
carries over to the derived category.

\begin{lemma} \label{lemma:def-L}
Let $\lambda:A\to B$ be a map in $D(\cA)$, and let $q$ be an integer such that the maps $\Hh^i(\lambda):\Hh^i(A)\to \Hh^i(B)$ are zero for $i=q-1, q$. Then there exists a unique morphism $L_\lambda^q : \Hh^q(A)[-q]\to \Hh^{q-1}(B)[1-q]$ making the following diagram commute
\[ 
    \xymatrix{
      \tau_{[q-1, q]} (A) \ar[d]_{\tau_{[q-1, q]}(\lambda)} \ar[r] & \Hh^q(A)[-q] \ar[d]^{L^q_\lambda} \\
      \tau_{[q-1, q]} (B) & \Hh^{q-1}(B)[1-q]. \ar[l]
    }
  \]
The same map $L_\lambda^q$ fits into the following commutative diagram
\[ 
    \xymatrix{
      \tau_{\leq q} A\ar[d]_{\tau_{\leq q}(\lambda)} \ar[r] \ar@{.>}[dr] & \Hh^q(A)[-q] \ar@{.>}[dr]^{L^q_\gamma} \\
      \tau_{\leq q} B & \tau_{\leq q-1} B \ar[l]\ar[r] & \Hh^{q-1}(B)[1-q].
    }
  \]
\end{lemma}

\begin{proof}
Consider the following commutative diagram with exact rows and columns.
  \[ 
    \xymatrix@C-0.5pc{
        \Hom(\tau_{[q]} A, \tau_{[q]} B[-1]) \ar[r]\ar[d] & \Hom(\tau_{[q-1, q]} A, \tau_{[q]} B[-1]) \ar[r]\ar[d] & \Hom(\tau_{[q-1]} A, \tau_{[q]} B[-1]) \ar[d] \\
         \Hom(\tau_{[q]} A, \tau_{[q-1]} B) \ar[r]\ar[d] & \Hom(\tau_{[q-1, q]} A, \tau_{[q-1]} B)\ar[r]\ar[d] & \Hom(\tau_{[q-1]} A, \tau_{[q-1]} B) \ar[d]\\
        \Hom(\tau_{[q]} A, \tau_{[q-1,q]} B) \ar[r]\ar[d] & \Hom(\tau_{[q-1, q]} A, \tau_{[q-1,q]} B) \ar[r]\ar[d] & \Hom(\tau_{[q-1]} A, \tau_{[q-1,q]} B) \ar[d]\\
        \Hom(\tau_{[q]} A, \tau_{[q]} B) \ar[r] & \Hom(\tau_{[q-1, q]} A, \tau_{[q]} B) \ar[r] & \Hom(\tau_{[q-1]} A, \tau_{[q]} B).\\
    }
  \]
  Note that the groups in the top row and the group in the bottom right corner are zero, as $\Hom(X, Y)=0$ if there exists an $n\in \bz$ such that $\tau_{\geq n} X = 0$ and $\tau_{\leq n-1} Y = 0$. Similarly, the left horizontal maps are injective. The first claim follows then by diagram chasing.  

  For the second assertion, we can first assume that $A=\tau_{\leq q}A$ and $B=\tau_{\leq q}B$. We can then reduce further to the case  $A=\tau_{[q-1,q]}(A)$ and $B=\tau_{[q-1,q]}(B)$, whereupon the claim becomes identical to the first assertion.    
\end{proof}

\begin{proposition} \label{hom.p}
  Let $C \tto{i} A \tto{\lambda} B \tto{\rho} C[1]$ be a distinguished triangle in the derived category $D(\cA)$, and consider the corresponding exact sequence:
  \[
    \cdots \To \Hh^{q-1}(A)  \rTo{\lambda} \Hh^{q-1}(B) \rTo{\rho} \Hh^q(C) \rTo{i} \Hh^q(A)  \rTo{\lambda} \Hh^q(B) \To \cdots
  \]
  Assume that $\Hh^q(\lambda) = \Hh^{q-1}(\lambda) = 0$, so that  we have a short exact sequence:
  \[
    0 \To \Hh^{q-1}(B) \rTo{\rho} \Hh^q(C) \rTo{i} \Hh^q(A) \To 0.
  \]
  Let $\kappa^q:  \Hh^q(A) \to \Hh^{q-1}(B)[1]$ be the corresponding 
derived map, as in Definition~\ref{xiu.d}.
  Then $\kappa^q = (-1)^{q-1}L^q_\lambda[q]$, where $L^q_\lambda$ is the map defined in Lemma~\ref{lemma:def-L}.
\end{proposition}
\begin{proof}
First note that if $\lambda' \colon A' \to B'$ satisfies $\Hh^{q-1}(\lambda') = \Hh^q(\lambda') = 0$ and there is a~commutative diagram of the form
\[ 
  \xymatrix{
    A' \ar[r]^{\lambda'} \ar[d]_a & B' \ar[d]^b \\
    A \ar[r]_\lambda & B
  }
\]
with the property that $\Hh^i(a)$ and  $\Hh^i(b)$ are isomorphisms for $i = q-1,q$, then the proposition holds for $\lambda'$ iff it holds for $\lambda$. Indeed, any distinguished triangle containing $\lambda'$ fits into a commutative diagram:
\[ 
  \xymatrix{
    C' \ar[d]_c \ar[r]^{i'} & A' \ar[r]^{\lambda'} \ar[d]_a & B' \ar[d]^b \ar[r]^{\rho'} & C'[1] \ar[d]^{c[1]}   \\
    C \ar[r]_i & A \ar[r]_\lambda & B \ar[r]_\rho & C[1]
  }
\]

Applying the functor $\tau_{\leq q}$ leaves $\Hh^i$ unchanged for $i \le q$, and applying $\tau_{\geq q-1}$ leaves $\Hh^i$ unchanged for $i \ge q-1$.  Thus we may without loss of generality assume that $A=\tau_{[q-1,q]}(A)$ and $B=\tau_{[q-1,q]}(B)$.  We have a morphism of distinguished triangles:
\[ 
  \xymatrix{
    \Hh^{q-1}(A)[1-q] \ar[r]^-a \ar[d]_{\lambda=0} & A \ar[r]^-b\ar[d]^\lambda & \Hh^q(A)[-q] \ar[r]\ar[d]^{\lambda =0} & \Hh^{q-1}(A)[2-q] \ar[d]^{\lambda=0} \\
    \Hh^{q-1}(B)[1-q] \ar[r]^-a  & B \ar[r]^-b & \Hh^q(B)[-q] \ar[r] & \Hh^{q-1}(B)[2-q] 
  }
\]
The left map being zero by hypothesis, we have $\lambda \circ a = 0$, and hence $\lambda$ factors through $b\colon A \to \Hh^q(A)[-q]$. It thus suffices to prove the assertion with the morphism $\Hh^q(A)[-q] \to B$ in place of $\lambda$.  Similarly, since $\Hh^q(\lambda) = 0$, we may replace $B$ by $\tau_{\leq q-1}(B)$. Thus we are reduced to the case in which $A=\Hh^q(A)[-q]$ and $B=\Hh^{q-1}(B)[1-q]$.  It follows that $C=\Hh^q(C)[-q]$. Note that $\lambda = L^q_\lambda$ in this situation.
 Therefore we have a commutative diagram whose vertical maps are isomorphisms.
\[ 
  \xymatrix{
    B[-1] \ar[d]\ar[r]^{-\rho} & C \ar[r]^i \ar[d] & A \ar[d] \ar[r]^\lambda & B \ar[d] \\
    \Hh^{q-1}(B)[-q] \ar[r]_{-\rho} & \Hh^{q}(C)[-q] \ar[r]_i & \Hh^q(A)[-q] \ar[r]_-\lambda & \Hh^{q-1}(B)[1-q]
  }
\]
Since the top row is distinguished, it follows that the bottom row is distinguished as well. Applying $[q]$ shows that 
\[ 
    \Hh^{q-1}(B) \rTo{-(-1)^q \rho}  \Hh^{q}(C) \rTo{(-1)^q i}  \Hh^q(A) \rTo{(-1)^q \lambda[q]}  \Hh^{q-1}(B)[1]  
\]
is distinguished. This is isomorphic to 
\[
  \Hh^{q-1}(B) \rTo{\rho}  \Hh^{q}(C) \rTo{i}  \Hh^q(A) \rTo{(-1)^{q+1} \lambda[q]}  \Hh^{q-1}(B)[1].
\] 
As we observed after Definition~\ref{xiu.d}, the  fact that these complexes are concentrated
in a single degree implies that the last map  is the unique one making the triangle
distinguished.  Thus 
$\kappa = (-1)^{q+1} \lambda[q] = (-1)^{q+1} L_\lambda^q[q]$, as desired.  
\end{proof}

\section{Logarithmic preliminaries}
\label{prelim.s}

\subsection{Notation and conventions}
\label{ss:log-notation}
For the basic facts about log schemes,
 especially the definitions of log differentials and log smoothness,  we refer to Kato's seminal paper~\cite{kato.lsfi} and the forthcoming book~\cite{o.llogg}.  Here we recall a few essential notions and constructions for the convenience of the reader. 

{\bf Monoids and monoid algebras.} If $(P, +, 0)$ is a commutative monoid, we denote by $P^*$ the subgroup of units of $P$, by $P^+$ the complement of $P^*$, and by $\ov P$ the quotient of $P$ by $P^*$.  A monoid $P$ is said to be \textit{sharp} if $P^* = 0$.  If $R$ is a fixed base ring, we write $R[P]$ for the monoid algebra on $P$  over $R$.  This is the free $R$-module with basis  
\[
  e \colon P \to R[P] \quad :\quad p \mapsto \mex p
\]
and with multiplication defined so that $\mex p \mex q = \mex{p+q}$.  The corresponding scheme $\ms P := \spec (R[P])$ has a natural structure of a monoid scheme.  There are two natural augmentations $R[P] \to R$. The first of these, corresponding to the identity section of $\ms P$,  is  given by the homomorphism $P \to R$ sending every element to the identity  element 
$1$  of $R$.  The second, which we call the \textit{vertex} of $\ms P$, is defined by the homomorphism sending $P^*$ to $1 \in R$ and $P^+$ to $0 \in R$.  The two augmentations coincide if $P $ is a group.

A commutative monoid $P$ is said to be \textit{integral} if the universal  map $P  \to P^\g$ from $P$ to a group is injective. An integral monoid $P$ is said to be \textit{saturated} if for every  $x \in P^\g$  such that  $nx \in P$ for some positive integer $n$, in fact $x \in P$.   A monoid is \textit{fine} if it is  integral and finitely generated and is \textit{toric} if
it is  fine and  saturated and   $P^\g$ is torsion free.
An \textit{ideal in a monoid} $P$ is a subset of $P$ which is stable
under addition by elements of $P$. 
If $J$ is an ideal in  $P$, then $R[P,J]$ denotes the quotient
of the monoid algebra $R[P]$  by the ideal generated by $J$.
The complement of $J$ in $P$ is a~basis for the underlying
$R$-module of $R[P,J]$.

{\bf Log structures.} A \textit{prelog structure} on a ringed space $(X, \oh X)$ is a homomorphism $\alpha$ from a sheaf of commutative monoids $\cM$ to the multiplicative
monoid underlying $\oh X$.  A \textit{log structure} is a prelog structure such that the induced map $\alpha^{-1}(\oh X^*) \to \oh X^*$ is an isomorphism.  The
\textit{trivial} log structure is the inclusion $\oh X^* \to \oh X$. A ringed space{$X$} endowed with a log structure {$\alpha_X$} is referred to as a \textit{log space.}  
An \textit{idealized log space} is a log space $(X, \alpha_X)$ together with a sheaf of ideals $\cK_X$ in $\cM_X$ such that $\alpha_X(\cK_X) = 0$~\cite[1.1]{o.lrhc},\cite[\S III,1.3]{o.llogg}.
   A prelog structure $\alpha \colon P \to \oh X$  on a ringed space factors through a universal associated log structure $\alpha^a \colon P^a \to  \oh X$. A log structure
   $\alpha$ on $X$ is said to be \textit{fine} (resp. fine and saturated) if locally on $X$ there exists a fine (resp. fine and saturated) {constant sheaf of monoids} $P$ and a prelog structure $P  \to  \oh X$ whose associated log structure is $\alpha$.  There is an evident way to form a category of log schemes, and the category
of fine (resp. fine saturated) log schemes admits fiber products, although their construction is subtle. 
Grothendieck's deformation theory provides a geometric way to define smoothness for morphisms
of log schemes, and many standard ``degenerate'' families become logarithmically smooth
when endowed with a suitable log structure.
A morphism  of integral  log spaces $f \colon X \to Y$ is {\em vertical}
if the quotient $\cM_{X/Y}$ of the map $f^*_\lag (\cM_Y) \to \cM_X$, computed
in the category of sheaves of monoids, is in fact a group. 
  We shall use
the notions of exactness, integrality, and saturatedness for morphisms of log schemes,
for which we refer to the above references and also the papers~\cite{ts.smls} and~\cite{ikn.qlrhc}.


If $P$ is a commutative monoid and $\beta \colon P\to A$ is a homomorphism into the multiplicative monoid underlying a commutative ring $A$, we denote by $\spec(\beta)$ the scheme $\spec A$ endowed with the log structure associated to the prelog structure induced by $\beta$.   In particular, if $R$ is a fixed base ring
and $ P \to R[P]$ is the canonical homomorphism from $P$ to the monoid $R$-algebra  of $P$, then
 $\ls P$  denotes the log scheme $\spec (P \to R[P])$, and if  $P$ is  fine and $R = \bc$,  we write $\lsan P$ for the log analytic space associated to $\ls P$. (If the analytic
context is clear, we may just write $\ls P$ for this space.)   If $v  \colon P \to R$ is the vertex of $\ms P$
(the homomorphism sending $P^+$ to zero and $P^*$ to $1$), the log scheme $\spec (v)$ is called
the   \emph{split log point} associated to $P$; it is called the \emph{standard log point} when $P = \bn$. 
If $J$ is an ideal in the monoid $P$,  we let $\ls {P,J}$ denote the closed 
idealized log  subscheme of $\ls P$ defined by the ideal $J$ of $P$.  The
underlying scheme of $\ls {P,J}$ is
the spectrum of the algebra $R[P,J]$, and the 
 points of $\lsan {P,J}$ are the homomorphisms $P  \to \bc$ sending $J$ to zero.

If $P$ is a toric monoid, the log analytic space $\lsan P$ is a partial compactification of  its dense open subset $\mss P:= \lsan {P^\g}$, and the  (logarithmic) geometry of $\lsan P$  expresses the geometry of this compactified set, a manifold with boundary. The underlying topological space of $\mss P$ is $\Hom(P, \bc^*)$, and its fundamental group $\li P$ (the ``log inertia group'') will play a fundamental role in  what follows.

\subsection{Some groups and extensions associated to a monoid}\label{geam.ss}

Let us gather here the key facts and notations we shall be using. If $P$ 
is a toric monoid (\ie, if $P$ is fine and saturated and $P^\g$ is torsion free)  we define:
\begin{align*}
  \lt P &:= \Hom(P, \sone),   \text{ where } \sone := \{ z \in \bc : |z| =  1\},\\
  \lr P &:= \Hom (P, \br_\ge),   \text{ where } \br_\ge := \{ r \in \br : r   \ge 0, 
   \text{ with its multiplicative monoid law}\} \\
  \li P &:= \Hom(P, \bz(1)),   \text{ where } \bz(1) := \{ 2\pi  i n : n \in \bz\} \subseteq \bc, \\
  \lv P &:= \Hom(P, \br(1)),  \text{ where } \br(1) := \{ ir  : r  \in \br \} \subseteq \bc, \\ 
  \lL P &:= \{ \text{affine mappings  }\li P \to \bz(1) \}, \\
  \chi &: P^\g \isommap \Hom( \lt  P, \sone ) \,:\, p \mapsto \chi_p,   \text{ where }  \chi_p(\sigma) := \sigma(p), \\ 
  \tilde\chi &: P^\g \isommap \Hom(\li P,\bz(1))  \subseteq \lL P : p \mapsto \tilde\chi_p,   \text{ where } \tilde\chi_p(\gamma) := \gamma(p). 
\end{align*}

An affine mapping $f \colon \li P \to \bz(1)$ can be written uniquely as a sum $f = f(0) + h$, where $h$ is a homomorphism $\li P \to \bz(1)$.  Since  $P$ is toric,
 the map $\tilde \chi$ is an isomorphism, so $h = \tilde \chi_p$ for  a unique $p \in P^\g$.  Thus the group $\lL P$ is a direct sum $\bz(1) \oplus P^\g$, which we write as an exact sequence:
\begin{equation} \label{univextp.e}
  0 \To \bz(1) \To {\lL P} \rTo{\xi} P^\g \To 0,
\end{equation}
for reasons which will become apparent shortly. 

The inclusion $\sone \to \bc^*$ is a homotopy equivalence, and hence so is the induced map $\lt P \to \mss P$, for any  $P$.  Thus the fundamental groups of $\mss P$ and $\lt P$ can be canonically identified.  The exponential mapping $\theta \mapsto e^\theta$ defines  the universal covering space $\br(1) \to \sone$, and there is an induced covering space $ \lv P \to \lt P$.    The subgroup $\li P = \Hom (P, \bz(1))$ of $\lv P$ acts naturally on  $\lv P$ by translation: $(v,\gamma) \mapsto v + \gamma $ if $v \in \lv  P$ and $\gamma \in \li P$.  The  induced action  on $\lt P$ is trivial, and  in fact $\li P$  can be identified with the covering group of the covering $\lv P \to \lt P$, \ie, the fundamental group of $\lt P$.  (Since the group is abelian we do not need to worry about base points.) We view  $\li P$  as  acting on the right on the
geometric object $\lv P$ and on the left on the set of functions on $\lv P$: if  $f$ is a function on $\lv P$, we have  
\[
  (\gamma f)(v) =f (v+\gamma).
\]
In particular, if $f $ is constant, then $\gamma f = f$, and  if $p \in
P^\g$, 
\begin{equation} \label{gamact.e}
  \gamma \tilde \chi_p = \tilde \chi_p +\gamma(p) 
\end{equation}
It follows that the set $\lL P$ of affine mappings $\li P \to \bz(1)$  is stable under the \mbox{$\li P$-action~(\ref{gamact.e})}. The homomorphism $\tilde \chi \colon P^\g \to \lL P$ is a  canonical splitting of the  exact sequence~(\ref{univextp.e}),  but the splitting is not stable under the action of $\li P$,  as the formula~(\ref{gamact.e}) shows.   The formula also shows that the exact sequence~(\ref{univextp.e}) can be viewed as an extension of trivial $\li  P$-modules.  Any $f \in \lL P$ extends naturally to an affine transformation $\lv P \to \br(1)$, and  in fact $\lL P$ is the smallest $\li P$-stable subset of the set of functions $\lv P \to \br(1)$ containing  $\tilde \chi_p$ for all $p \in P^\g$. 

The dual of the extension (\ref{univextp.e}) has an important geometric interpretation.  Consider the group algebra $\bz[\li P]$ with basis $e \colon \li P \to \bz[\li P]$. It is equipped with a right action of $\li P$ defined by 
$e^\delta \gamma  = e^{\delta+ \gamma}$.  Its augmentation ideal $J$ is generated by elements of the form $e^\delta - 1$ for $\delta \in \li P$ and is stable under the action of $\li P$.  The induced action on $J/J^2$ is trivial, and there is an isomorphism of abelian groups:
\begin{equation} \label{1-gamma.e}
  \lambda \colon \li P \To J/J^2, \quad  \gamma \mapsto [ \mex\gamma - \mex 0].
\end{equation}
Identifying $\li P$ with $J/J^2$, we have a split exact sequence of $\li P$-modules:
\begin{equation} \label{groupal.e}
  0 \To \li P \To \bz[\li P]/J^2 \To \bz \To 0,
\end{equation}
where the action of $\li P$ on $\li P$ and on $\bz$ is trivial. 

\begin{proposition} \label{extens.p}
  There is a natural isomorphism
  \[
    \lL P \isommap \Hom(\bz[\li P]/J^2, \bz(1)),
  \]
  compatible with the structures of extensions (\ref{univextp.e}) and (\ref{groupal.e}) and the  (left) actions of $\li P$.  The boundary map $\partial$ arising from the extension~(\ref{univextp.e})
  \[ 
    \partial \colon P^\g \To H^1(\li P ,\bz(1)) \cong \Hom(\li P,\bz(1)) \cong P^\g
  \]
  is the identity.  
\end{proposition}

\begin{proof}
Since $\bz[\li P]$ is the free abelian group with basis $\li P$, the map $f \to h_f$  from the set of functions $f \colon \li P \to \bz(1)$ to the set of homomorphisms $\bz[\li P] \to \bz(1)$  is an isomorphism, compatible with the natural left actions of $\li P$. If $f : \li P \to \bz(1)$, then $h_f$ annihilates $J$ if and only if $f(\gamma) = f(0)$ for every $\gamma$, \ie, iff $f \in \bz(1) \subseteq L_P$.  Furthermore, $h_f$ annihilates $J^2$ if and  only if  for every  pair  of elements $\gamma, \delta$ of $\li P$,  
\[ 
  h_f((\mex \delta - 1)(\mex \gamma -1))= 0, \text{  i.e.,  iff }
  f(\delta +\gamma) - f(\delta) -f(\gamma) + f(0) = 0.
\]
But this holds if and only if $f(\delta +\gamma) - f(0) = f(\delta) - f(0) + f(\gamma) - f(0)$, \ie, if and only if $f - f(0) $ belongs to $\tilde \chi(P^\g)$,
\ie, if and only if $f \in \lL P$. 

To check the claim about the boundary map $\partial$, and in particular its sign, we must clarify our conventions. If $\Gamma$ is a group and $E$ is a $\Gamma$-module, then we view  $H^1(\Gamma, E)$ as the set of isomorphism classes of $E$-torsors in the category of $\Gamma$-sets.  If the action of $\Gamma$ on $E$ is trivial and $L$ is  such a torsor, then for any $\ell \in L$ and any $ \gamma \in \Gamma$, the element $\phi_{L,\gamma} :=\gamma(\ell) -\ell$ is independent of the choice of $\ell$, the mapping $\gamma \to \phi_{L,\gamma}$ is a~homomorphism $\phi_L \colon \Gamma \to E$, and the correspondence $L \mapsto \phi_L$ is the  isomorphism:
\begin{equation} \label{h1hom.e}
  \phi \colon H^1(\Gamma,E) \To \Hom(\Gamma,E).
\end{equation}

To verify the claim, let $p$ be an element of $P^\g$. Then  $\partial(p) \in H^1(\li P,\bz(1))$ is the $\bz(1)$-torsor of all $f \in L_P$ whose image under $\xi \colon \lL P \to P^\g$ is $p$.  Choose any such $f$, and write $f =f(0) + \tilde \chi_p$.  Then if $\gamma \in \li P$,  we have $\gamma(f) = f + \gamma(p)$, and thus 
\[ 
  \partial (p) \mapsto\phi_L(\gamma) = \gamma(f) - f = \gamma(p) = \tilde \chi_p(\gamma).
\]
This equality  verifies our claim.
\end{proof}

\subsection{Betti realizations of  log schemes}

Since the Betti realization $X_\lag$ of an fs-log analytic space $X$
plays a crucial role here, we briefly review its construction.
As a set, $X_\lag$  consists of pairs $(x,\sigma)$,  where $x$ is a point of $X$ and $\sigma$ is a homomorphism  of monoids making the following diagram commute:
\[ 
  \xymatrix{
    \oh {X,x}^* \ar[r] \ar[d]_{f\mapsto f(x)} & \cM_{X,x}  \ar@{.>}[d]^\sigma \\
    \bc^* \ar[r]_\arg & \sone.
  }
\]
The map $\tau_X \colon X_\lag \to X$ sends $(x,\sigma)$  to $x$.  A  (local) section $m$ of $\cM_X$ gives rise to a (local) function $\arg(m) \colon X_\lag \to\sone$, and the topology on $X_\lag$ is  the weak topology coming from the map $\tau_X$ and these functions. The map $\tau_X$ is proper, and for $x \in X$, the fiber $\tau_X^{-1}(x)$ is naturally a torsor under the group $\lt{X,x} :=\Hom(\ov  \cM_{X,x}^\g, \sone)$.    Thus the fiber is connected if and only if $\ov \cM_{X,x}^\g$ is torsion free, and if this is the case, the fundamental group $\li {X,x}$ of the fiber is canonically isomorphic to $\Hom(\ocM_{X,x}, \bz(1))$. 
The  map $\tau_X \colon X_\lag\to X_\tp$ is characterized by the property that for every topological space $T$,  the set of morphisms 
$T  \to  X_\lag$ identifies with the  set of pairs $(p,c)$, where
 $p \colon T \to X_\tp$ 
is a continuous map and $c \colon p^{-1}(\cM_X^\g) \to \sone_T$
 is a homomorphism from  $p^{-1}(\cM_X^\g)$ to the sheaf  $\sone_T$ of continuous $\sone$-valued functions on $T$ such that 
$ c(f)=\arg(f)\text{ for all }f\in p^{-1}(\oh X^*)$.

When $X= \ls {P,J}$, the construction of $S_\lag$ can be understood easily as the introduction of ``polar coordinates.''  The multiplication map ${\br_\ge \times \sone }\to \bc$ maps polar coordinates to  the  standard complex coordinate.
Let $\lr {P,J} := \{ \rho \in \lr P : \rho(J) = 0 \}$.  Multiplication
induces a  natural surjection:
\begin{equation}\label{taudef.e}
  \tau \colon \lr {P,J} \times \lt P \To \lsan {P,J}: (\rho,\sigma) \mapsto \rho \sigma.
  \end{equation}
Then $\lslag{P,J} \cong \lr {P,J}\times \lt P$, and $\tau$ corresponds to the canonical map  $\tau_{\lslag {P,J}}$.  The exponential map induces a universal covering:
\begin{equation} \label{etacov.e}
  \eta \colon \tlslag {P,J} := \lr {P,J} \times \lv P  \To \lslag {P,J},
\end{equation}
whose covering group identifies naturally with $\li P$.  Thus the group $\li P$ is also the fundamental group of $\lslag{P,J}$.

\begin{remark}\label{xlogbund.r}
If $X$ is a smooth curve endowed with the compactifying log structure induced by the complement of a point $x$, then $\tau_X\colon X_\lag\to X_\tp$ is the ``real oriented blow-up''of $X$
 at $x$, and there is a natural bijection between $\tau_X^{-1}(x)$ and the set of ``real tangent directions'' $(T_x \und X\setminus \{0\})/\br_>$ at $x$. Below we provide a more general and robust identification of this kind.  

If $X$ is any log analytic space and $q$ is a global section of $\ocM_{X}^\g$, 
let $\cLL_q^*$ denote the  sheaf of sections of $\cM_X^\g$ which
map to $q$.  This sheaf has a natural structure of an \mbox{$\oh
X^*$-torsor}, and we let $\cLL_q$ denote the corresponding invertible
sheaf of $\oh X$-modules
and $\cLL_q^\vee $ its dual.    A local section $m$ of $\cLL^*_q$
defines a local generator for the invertible sheaf $\cLL_q$.
If $(x,\sigma)$ is a point of $X_\lag$ and $m $ is a local
section of $\cLL^*_q$,  let $m(x)$ be the value of $m$
in the one-dimensional $\bc$-vector space $\cLL_q(x)$
 and let   $ \phi_m  \in \cLL^\vee_q(x)$ 
be the unique linear map $\cLL(x) \to \bc$ taking $m(x)$ to
$\sigma(m)$. 
If $m'$ is another local section of $\cLL^*_q$, there is a unique local
section $u$ of $\oh X^*$ such that $m' = um$, and then
${\phi_{m'}} = |u(x)|^{-1} \phi_m$.  Indeed:
\begin{align*}
  \phi_{m'}( m(x)) &= u(x)^{-1}\phi_{m'}(m'(x)) = u(x)^{-1} \sigma(m')\\
  &= u(x)^{-1}\arg(u(x)) \sigma (m) = |u(x)|^{-1} \phi_m(m(x)).
\end{align*}
Thus $\phi_m' $ and $\phi_m$ have the same image in the quotient of 
$\cLL^\vee_{q}(x)$ by the action of $\br_>$. This quotient corresponds
to the set of directions in the one-dimesional  complex vector space
$\cLL_q^\vee(x)$.   If $L$ is any one-dimensioonal complex vector space,
it seems reasonable to denote the quotient $L/\br_>$
by $\sone(L)$.
Thus we see that there is a natural map:
$\beta \colon \tau_X^{-1}(x)\to \sone({\cLL_q^\vee(x)})$. 
The source of this continuous map is
a torsor under  $\lt{X,x} = \Hom(\ocM_{X,x}^\g.\sone)$
and its target is naturally a torsor under $\sone$.  One verifies
immediately that
if  $ \zeta \in \Hom(\ocM_X^\g, \sone)$ and $\sigma \in \tau_X^{-1}(x)$, 
then $\beta(\zeta \sigma)  = \zeta(q) \beta(\sigma)$.

When $\alpha_X$ is the  the log structure
coming from a divisor $D$ on $X$, the divisor $D$ gives
rise to a global section $q$  of $\ocM_X$, the invertible sheaf $\cLL_q$ 
is the ideal sheaf defining $D$, and $\cLL_q^\vee(x)$ is the 
normal bundle to $D$ at $x$.
 In particular, if $X$ is a curve, then $\cLL_q^\vee(x)\cong T_x (\und X)$ and $\sone({\cLL_q^\vee(x)})$ is the aforementioned space $(T_x (\und X) \setminus\{0\})/\br_>$ of real tangent directions at $x$.
\end{remark}

On the space $X_\lag$ one can make sense of logarithms of sections of $\cM_X$. There is an exact sequence of abelian sheaves:
\begin{equation} \label{logexp.e}
  0 \To \bz(1) \To \mathcal{L}_X \rTo{\pi} \tau_X^{-1}(\cM^\g_X) \To 0
\end{equation}
where $\mathcal{L}_X$ is defined via the Cartesian diagram
\[ 
  \xymatrix{
    \mathcal{L}_X \ar[r] \ar[d]_\pi  & \br(1)_{X_\lag} \ar[d]^\exp  \\
    \tau_X^{-1}(\cM^\g_X) \ar[r]_\arg  & \sone_{X_\lag}.
  }
\]
There is also a homomorphism:
\begin{equation} \label{otol.e}
  \ep \colon \tau_X^{-1}(\oh X) \To \mathcal{L}_X \quad :\quad f \mapsto (\exp f, \im(f)),
\end{equation}
and the sequence
\begin{equation} \label{otolex.e}
  0 \To  \tau_{X}^{-1}(\oh X) \To \mathcal{L}_X \To \tau_X^{-1}(\ocM_X^\g) \To 0
\end{equation}
is exact.
When the log structure on $X$ is trivial the map $\ep$  is an isomorphism, and the exact sequence~(\ref{logexp.e}), called the ``logarithmic exponential sequence''
reduces to  the usual exponential sequence on $X$.

We can make this construction explicit in a special ``charted'' case.

\begin{proposition}\label{lslogp.p}
Let $X := \lsan {P,J}$, where $J$ is an ideal in a sharp toric monoid $P$,
let $\eta \colon \tilde X_\lag  \to X_\lag$ be the covering~(\ref{etacov.e}),
and let $\tilde \tau_X := \tau_X \circ  \eta$.
Then on $\tilde X$, the pullback 
\[
  0 \To \bz(1) \To \mathcal{L}_P \To P^\g \To 0,
\]
of the extension~(\ref{logexp.e}) along the natural map
$P^\g \to  \tilde \tau_X^{-1}(\cM_X^\g)$ 
identifies with  the sheafification 
of the extension~(\ref{univextp.e}).
This identification is compatible with the actions of $\li P$. 
\end{proposition}

\begin{proof}
It is enough to find a commutative diagram
\[ 
  \xymatrix{
    0 \ar[r] & \bz(1) \ar[r] \ar@{=}[d] & \lL P \ar[r] \ar[d] & P^\g \ar[d] \ar[r] & 0 \\
    0 \ar[r] & \bz(1) \ar[r] & \eta_S^* (\mathcal{L}_S ) \ar[r] &  \tilde \tau_S^{-1}(\cM_S^\g) \ar[r] & 0.  
  }
\]
We define a map $\lL P \to \eta_S^*( \mathcal{L}_S)$ as follows.  Every $f  \in \lL P$ can be written uniquely as $f = f(0) + \tilde \chi_p$,  where $p \in P^\g$ and $f(0)  \in \bz(1)$.  Let $(\rho,\theta)$ be a point of $\tilde S_\lag$, with image $(\rho, \sigma) \in S_\lag$.  Then  the pair  $(f(\theta), p)  \in \br(1)\times P^\g$ defines an element of $\mathcal{L}_{S,(\rho,\sigma)}$, because  
\[
  \sigma(p) = \exp (\theta(p)) = \exp(\tilde \chi_p(\theta) )
  = \exp(f(\theta)-f(0)) = \exp(f(\theta)),
\]
since $f(0) \in \bz(1)$. 
\end{proof}

\section{Logarithmic degeneration}
\label{logdeg.s}

\subsection{Log  germs and log fibers}
\label{logfib.s}
We begin with an illustration of the philosophy that the local geometry of a suitable morphism can  be computed from its log fibers. We use the following notation and terminology. If  $\tau\colon X' \to X$ is a continuous map of topological spaces, then $\Cyl(\tau)$ is the \textit{(open) mapping cylinder} of $\tau$, defined as the pushout in the diagram:
\[ 
  \xymatrix{
    X' \ar[d]_\tau \ar[r] & X'\times [0, \infty)\ar[d]^{\pi} \\
    X \ar[r] & \Cyl(\tau), 
  }
\]
where the top horizontal arrow is the embedding sending $x' \in X'$ to $(x',0)$.  In $\Cyl(\tau)$, the point $(x',0)$ becomes identified with the point $\tau(x)$.   A commutative diagram:
\[ 
  \xymatrix{
    X' \ar[d]_{\tau_X} \ar[r]^{f'} & Y'\ar[d]^{\tau_Y} \\
    X \ar[r]_f & Y 
  }
\]
induces a mapping $\Cyl_f \colon \Cyl(\tau_X)  \to \Cyl(\tau_Y)$.

\begin{theorem} \label{discr.t}
  Let $ f\colon X \to Y$ be a morphism of fine saturated log analytic
  spaces, where $Y$ is
  an open neighborhood of the origin $v$  of the standard log disc
  $\ls \bn$. 
 Assume that $f$ is proper, smooth, and vertical.
 Then after $Y$  is replaced by a possibly smaller
  neighborhood of $v$,  there is a commutative diagram:
   \[ \xymatrix{
      \Cyl(\tau_{X_v}) \ar[r]^\sim \ar[d]_{\Cyl_{f_v}} & X_\tp \ar[d]^{f_\tp} \\
      \Cyl(\tau_v) \ar[r]^\sim & Y_\tp,
    }
  \]
  in which the horizontal arrows are isomorphisms.
  {(These arrows are neither unique nor canonical, and
  depend on a choice of a trivialization of a fibration
  (see Lemma~\ref{trivfib.l}).)}
\end{theorem}
\begin{proof}
  Note that since the stalks of $\ov \cM_{Y}$ are
  either $0$ or $\bn$,
  the morphism $f$ is automatically exact. 
We may  assume that $Y = \{ z \in \bc : |z| < \ep \}$ for some $\ep > 0$.  Then  $Y_\lag  \cong v_\lag \times [0,\ep)  \cong \sone \times [0, \infty) $. 
With this identification, the map $\tau_Y$ is just the collapsing map
shrinking $\sone \times 0$ to a point, and hence induces a
homeomorphism $\Cyl(\tau_v) \to Y_\tp$.
The following lemma generalizes this construction.

\begin{lemma}\label{taupush.l}
  Let $X$ be a fine log analytic space and let $X^+$ be
the closed subspace of $X$ on which the log structure is nontrivial,
endowed with the induced log structure.  Then the
diagram
\[\xymatrix{
X^+_{\lag} \ar[r]\ar[d]_{\tau_{X'}} & X_\lag \ar[d]^{\tau_X} \cr
X^+_{\tp} \ar[r] & X_\tp  \cr
}\]
is cocartesian as well as cartesian.
\end{lemma}
\begin{proof}
  The diagram is cartesian because formation of $X_\lag$ is compatible
with strict base change.  To see that it is cocartesian, observe that since
$\tau_X$ is surjective and proper, $X_\tp$ has the quotient topology
induced from $X_\lag$.  Since $\tau_X$ is an isomorphism over $X_\tp
\setminus X^+_{\tp}$, the equivalence relation defining $\tau_X$ is generated by the equivalence relation defining $\tau_{X^+}$.  It follows that the square is a pushout, \ie, is cocartesian.
\end{proof}
Let $Y$ be an open disc as above and fix an identification
$Y_\lag \cong \sone \times [0,{\infty})$. 
\begin{lemma}\label{trivfib.l}
Let $f \colon X \to Y$ be a smooth and  proper morphism of
fine saturated log analytic spaces, where $Y$ is an open log disc as
above.  
Then there exist a homeomorphism $ X_{v,\lag}\times [0, {\infty})  \to X_\lag$ 
and a commutative diagram:
\[\xymatrix{
X_{v,\lag}\times [0,{\infty})\ar[d]_{f_{v,\lag}\times \id} \ar[r] &X_\lag \ar[d]^{f_\lag} & \cr
v_\lag \times [0, {\infty}) \ar[r]^-\sim  &Y_\lag ,
}\]
where the restrictions of the  horizontal arrows to
 $X_{v,\lag}\times 0$ and  $v_\lag\times 0$
are the inclusions. 
\end{lemma}
\begin{proof}
  Since $Y$ is  a log disc, the morphism $f$
  is automatically exact.  Thne by \cite[5.1]{no.rr}, the map $f_\lag$ is a topological fiber bundle,
and since $Y_\lag$ is connected, all fibers are homeomorphic.
Let $r \colon Y_\lag = \sone \times [0, \ep) \to v_\lag$
be the obvious projection and let $i \colon v_\lag \to 
Y_\lag$ be the embedding at $0$.  
Then $f_{v,\lag}\times \id$  {identifies with} the pullback of $f_\lag$ along
$ir$.  The space of isomorphisms of fibrations
$f_{v,\lag}\times \id \to f_{\lag}$ is a principal $G$-bundle, where
$G$ is the group of automorphisms of the fiber, endowed with the
compact open topology.  Since $ir $ is homotopic to the identity,
it follows from \cite[IV,9.8]{hus.fb} that this principal $G$-bundle is trivial,
proving the lemma. 
\end{proof}

The diagram of Lemma~\ref{trivfib.l} forms the rear square of the
  following   diagram:

  \[ 
   \xymatrix{
     &X_{v,\lag} \times [0,\infty)
     \ar[rr]\ar[dl]_\pi\ar[dd]_>>>>>>>>>{f_{v\lag}\times \id} &&X_\lag\ar[dl]_{\tau_X} \ar[dd]^{f_\lag}\cr
      \Cyl(\tau_{X_v}) \ar@{.>}[rr] \ar[dd]_{\Cyl_{f_v}} && X_\tp \ar[dd]^>>>>>>>>>>>>>>>>>{f_\tp}\\
     &{v_\lag}\times[0,\infty) \ar[rr]\ar[dl]_\pi&& Y_\lag\ar[dl]^{\tau_Y} \cr 
      \Cyl(\tau_v) \ar@{.>}[rr] && Y_\tp,}
  \]
The map $\pi$, from the definition of the mapping cylinder, is part of
the pushout diagram which identifies a point $(x_\lag,0)$ with
$\tau_{X_v}(x) \in X_{v,\tp} \subseteq X_\tp$, and the existence
of the dotted  arrows follows. Because the morphism $f$ is
vertical, the subset $X^+$ of $X$ where the log structure is
nontrivial is just $X_v$, and Lemma~\ref{taupush.l} tells us that
the morphism $\tau_X$ is also a pushout making the same
identifications.  Thus the horizontal arrows are
homeomorphisms, and the proposition follows.
\end{proof}

\begin{remark}
  Although we shall not go into details here, let us mention that the same result, with the same proof,
holds if $X\to Y$ is only {relatively smooth}, as defined in \cite{no.rr}.
\end{remark}

More generally, suppose that
$P$ is a sharp toric monoid and   that
$Y$ is a neighborhood of the vertex $v$ of $\ls P$.
Note that $v$ has  a neighborhood basis
of sets of the form $V_P := \{ y \in \ls P : |y| \in V \}$,
where $V$ ranges over the open neighborhoods
of the vertex of $\lr P$.  If $f \colon X \to V_P$
is a morphism of log spaces,
let $g \colon X \to V := |f|$,
 and note that $X_{v,\lag} = g^{-1}(0) = (\tau_Y\circ
 f_\lag)^{-1}(v)$.  For each $x \in X_v$, the fiber $\tau^{-1}_X(x)$
 is a torsor under the action of $\lt {X,x} := \Hom(\ov \cM_{X,x},
 \sone)$.  For $\rho \in V \subseteq \lr P$, let $F(\rho) :=
 \rho^{-1}(\br_>)$, a face of $P$,  and let $G(\rho) $ be  the face of $\ocM_{X,x}$ generated by the image of $F(\rho)$ in $\ocM_{X,x}$ via the homomorphism $f_x^\flat \colon P \to \ocM_{X,x}$.  Then we set:
 \begin{eqnarray*}
\lt {Y,\rho} &:= &\Hom(P/F(\rho),\sone) \subseteq \lt P \\
  \lt {X_{x,\rho}} &:=& \Hom(\ov \cM_{X,x}/G(\rho), \sone )\subseteq \lt {X,x}.
 \end{eqnarray*}
There is a natural map $\lt{X_{x,\rho}} \to \lt{Y,\rho}$ induced by $f_x^\flat$. 

\begin{conjecture} \label{noglob.c}
  With the notation of the previous paragraph,
  let    $f \colon X \to Y = V_P$ be a smooth proper and exact
  morphism of fine saturated  log analytic spaces.  Then, after
  possibly shrinking $V$, there is a commutative diagram:
  \[ 
    \xymatrix{
      X_{v,\lag}\times V \ar[d]_{f_{v,\lag}\times\id} \ar[r] & X_\tp \ar[d]^{f_\tp} \\
      v_\lag \times V \cong Y_\lag \ar[r] & Y_\tp,
    }
  \]
   where the bottom arrow is (the restriction of) the map $\tau_Y$~(\ref{taudef.e}),  and the top arrow is the quotient map which identifies $(x_1,\rho_1)$ and $(x_2,\rho_2)$ 
if and only if:
\begin{enumerate}
\item  $\rho_1 = \rho_2$,
\item  $\tau_X(x_1) = \tau_X(x_2)$,
\item  $x_1$ and $x_2$ are in the same orbit under the action of $\lt {X_{\tau(x_i)}}(\rho)$ on $\tau^{-1}_X(\tau(x_i))$.
\end{enumerate}
In particular, the log fiber $f_v:X_v\to v$ determines $f$ topologically in a neighborhood of $v$.
\end{conjecture}
This conjecture is suggested by Remark~2.6 of \cite{no.rr}, which
implies that such a~structure theorem holds locally on $X$.

Motivated by the above philosophy, we now turn to a more 
careful study of log schemes which are smooth over a
log point $S$.   
We shall see that the normalization
of such a scheme provides a canonical way of cutting it into pieces,
each of whose Betti realizations 
 is a manifold with boundary and is canonically trivialized 
over $S_\lag$.  In fact this cutting process works more
generally, for ideally smooth log schemes.

\begin{theorem}\label{idealcut.t}
  Let $X$ be a fine, smooth, and saturated idealized log scheme over a field $k$
such that  $\cK_X \subseteq \cM_X$ is a 
sheaf of radical ideals.   Let
$\ep  \colon \und X' \to \und X$ be the normalization of the underlying
scheme  $\und X$.
\begin{enumerate}
\item The set $U$ of points $x$ such that $ \cK_{X,\xx} = \cM_{X,\xx}^+ $
for some (equivalently every) geometric point $\xx$ over $x$
is an open and dense subset of $ X$.  Its underlying scheme $\und U$
is smooth over $k$, 
and its complement
$Y$ is defined by a coherent sheaf of ideals $\cJ$  in $\cM_X$.
\item The log scheme $X'$ obtained by endowing $\und X'$ with
the compactifying  log structure associated to the open subset 
$\ep^{-1}(\und U)$ is fine, saturated, and smooth over $k$.  
\item Let $X''$ be the log scheme obtained by endowing $\und X'$
with the log structure induced from $X$.  There exists
a unique morphism $h \colon X'' \to X'$ such that $\und h$
is the identity.  The homomorphism
$h^\flat\colon \cM_{X'} \to \cM_{X''}$ is injective and 
identifies $\cM_{X'} $ with  a sheaf of faces in $\cM_{X''}$, and the 
quotient $\cM_{X''/X'}$ is a locally constant sheaf of 
fine sharp monoids.
\end{enumerate}
\end{theorem}
\begin{proof}
All these statements can be checked \'etale locally on $X$.
Thus we may assume that there exists a
chart $\beta \colon (Q, K) \to (\cM_X, \cK_X)$ for $X$, 
where $(Q,K)$ is a fine saturated  idealized monoid, where 
the order of $Q^\g$ is invertible in  $k$,  and where 
the corresponding morphism
 $b \colon \und X \to \ms {(Q,K)}$ is \'etale \cite[IV,
 3.3.5]{o.llogg}.  
 Thanks to the existence of the chart, we can work
with ordinary points instead of geometric points.
Furthermore we may, after a further localization,
assume that the chart is local at some point $x$ of $X$,
so that the map $\ov Q \to \ocM_{X,x}$ is an isomorphism.
Since $\cK_{X,x}$ is a radical ideal,  it follow that the 
same is true of $K$.   Statements (1)--(3) are stable
under \'etale localisation, so we are reduced to proving
them when $X = \ls {(Q,K)}$. 


Since $K$ is a radical ideal  of $Q$, it   is the intersection of a finite number of primes $\pee_1,
\ldots, \pee_r$,  and we may assume that  each $\pee_i$ minimal among those ideals
containing $K$~\cite[I, 2.1.13]{o.llogg}. 
 Let $\qu_1, \ldots, \qu_s$ be the remaining
prime ideals of $Q$ which contain $K$ and let $J := \qu_1  \cap \cdots 
\cap \qu_s$.  If $x  \in X$, 
let $\beta_x \colon Q \to \cM_{X,x}$ be the homomorphism induced by
$\beta$ and  let $\qu_x := \beta_x^{-1}(\cM_{X,x}^+)$. 
 Then $x \in Y$ if and only if
$K_{\qu_x} \subsetneq Q^+_{\qu_x}$, which is the case if and only if
$\qu_x = \qu_i$ for some $i$, or equivalently, if and only if $J
\subseteq \qu_x$.  Thus $Y$  is  the closed subscheme of $X$
defined by the coherent sheaf of ideals $\tilde J$  associated to $J$.

To see that $U$  is dense, observe that the irreducible
components of $X$ are defined by the prime ideals $\pee_i$ of $Q$ above.
Let $\zeta_i$ be the generic point of the irreducible component
corresponding to $\pee_i$.  Then $K_{\pee_i} = Q^+_{\pee_i}$, so $\zeta_i  \in U$. 
It follows that $U$ is dense in $X$.   To see that $\und U$ is smooth, let
$x$ be a point of $U$ and replace $\beta$ by its localization at $x$. 
Then it follows from the definition of $U$ that $K = Q^+$ and hence
that  $\ms{Q, K}  \cong  \ms {Q^*}$
which is indeed smooth over $k$. 

To prove statement (2) we continue to assume that
$X = \ls {(Q,K)}$.  For each minimal $\pee_i$ over $K$,
let $F_i$ be the corresponding face.  Then $\ms {Q, \pee_i} \cong \ms {F_i}$.
Since $Q$ is saturated, so is each $F_i$, and hence each 
scheme $X_{F_i} :=\ms {F_i}$ is normal.   Thus the disjoint union
 $\sqcup \{\ms F_i\}$ is the normalization
of $\ms{Q, K}$.  A point $x'$ of $\und X_{F_i}$ lies in $\ep^{-1}(U)$
if and only if its image in $\ms {F_i}$ lies in $\ms {F_i^\g}$. 
 It follows from \cite[III,1.9.5]{o.llogg}  
that the compactifying log structure on $\und X_{F_i}$  is coherent, charted by
$F$, and hence from \cite[IV, 3.1.7]{o.llogg} that the resulting log
scheme $X'_{F_i}$ is smooth over $k$.  Thus $X'/k$ is smooth.  This
completes the proof of statements (1) and (2).

To define the morphism $h$, it will be convenient to first introduce
an auxiliary log structure. 
 Let $\oh {X'}' \subseteq \oh {X'}= \oh {X''}$ be the sheaf of nonzero divisors
 in $\oh {X'}$ and let $ \cM'$ be its inverse image in $\cM_{X''}$ via
the map $\alpha_{X''} \colon \cM_{X''} \to \oh {X'}$.   
Then $\cM'$ is a sheaf of faces in $\cM_{X''}$,
 and  the induced map $\alpha' \colon \cM' \to \oh {X'}$ is a log
 structure on $\und X'$.  If $\xx'$ is a geometric
 point of $U'$, then 
$\cK_{X',\xx'} = \cM^+_{X',\xx'}$,  so the map
$\cM_{X'',\xx'}^+ \to \oh{X'_{\xx'}}$ is zero.  Hence  $\cM'_{\xx'} = \oh {X',\xx'}^*$,
and thus  $\alpha'$ is trivial on $U'$.  
 It follows that there is a natural morphism from $\alpha'$ to the
 compactifying log structure  $\alpha_{U'/X'}$. We check that this
 morphism is an isomorphism  at each  point $x' $ of $X'$.
Since both log structures are trivial if $x' \in U'$, we may
assume that $x' \in Y':= X'\setminus U'$ and that
$X$ admits a chart $\beta$ as above, local at $\ep(x')$.
 Let $F$ be the face
of $Q$ such that $x' \in   X_F$.   If $q \in\pee := Q \setminus F$,
then $\alpha_{X''}(\beta(q))$ vanishes in $\oh
{X_F}$.  If on the other hand $q\in F$, then $
\beta(q)$ is a nonzero divisor
 on $X_F$   
  Thus $\beta^{-1}(\cM'_{x'}) = F$.   Since $F$ is a chart for
  $\cM_{X'}$ and $\cM'_{x'}$ is a face of $\cM_{X'',x'}$,   it follows that the map
$\cM'_{x'} \to \cM_{X',x'}$ is  an isomorphism.  
 The  inverse of this isomorphism followed by the inclusion
$\cM' \to \cM_{X''}$ defines a morphism of log structures
$\alpha' \to \alpha_{X''}$ and hence a morphism  of log schemes
$h \colon X'' \to X'$ with $\und h = \id$.  

To prove that $h$ is unique, note that
since $\cM' \to \cM_{X'}$ is an isomorphism and $\alpha_{X'}$ is
injective, the homomorphism $\alpha'\colon \cM
\to \oh {X'}$ is also injective.
Let  $h' \colon X'' \to X'$ be any morphism of log schemes
with $\und h' = \id$ and let  $m$ be a local section
of $\cM_{X'}$.  Then $\alpha_{X''}(h'^\flat(m)) = \alpha_{X'}(m)$
is a nonzero divisor in $\oh {X'}$, so the
homomorphism $h'^\flat \colon \cM_{X'} \to
\cM_{X''}$ necessarily factors through $\cM'$.  Since $\alpha'\circ
h'^\flat = \alpha_{X'}$ and $\alpha'$ is injective, necessarily
$h'^\flat = h^\flat$.  

We have already observed that the
image $\cM'$ of $h^\flat$ is a sheaf of faces of $\cM_{X''}$,
and it follows that the quotient monoid $\cM_{X''/X'}$ is sharp.
To check that it is locally constant, we may assume 
that $X$ admits a chart as above  and work on  the subscheme $X''_F$
of of $X''$ defined by a face $F$ as above.   
 Then $\beta'' \colon  Q \to \cM_{X''}$
is a chart for $\cM_{X''}$.  Assume that $\beta''$ is local at a point
$x''$ of $X''$ and that $\xi''$ is a generization of $x''$.  
Then   $\cM_{X''/X',x''} = Q/F$ and  $\cM_{X''/X',\xi''} = Q_G/F_G$,
where $G := \beta_{\xi''}^{-1}(\oh {X'',\xi''}^*)$.  Since $G \subseteq
F$, the cospecialization map $Q/F \to Q_G/ F_G$ is an isomorphism.
It follows that $\cM_{X''/X'}$ is (locally) constant.
\end{proof}

Let us now return to the case of smooth log schemes over  a
log point. 

\begin{corollary}\label{vertons.c}
Let $f \colon X \to S$ be a smooth and  saturated
 morphism
from a fine saturated log scheme to the log point $\spec (P \to k)$,
where $P$ is a fine saturated and sharp monoid.
  Let     $\und  \ep \colon \und X' \to \und X$
    be the normalization of the underlying scheme  $\und X$. 
    \begin{enumerate}
\item The set  $U := \{ x \in X: \cM_{X/S,\xx} = 0\}$ 
is a dense open  subset of $X$.  Its underlying scheme $\und U$
is smooth over    $\bc$, and $\ep$ induces an
  isomorphism  $\und U' := \und \ep^{-1}(\und U) \to \und U$.  
    \item The log scheme $X'$ obtained by endowing $\und X'$
    with the compactifying log structure associated to the open
    subset $\und U'$  is fine, saturated and smooth over      $\bc$.  
   \item  Let $X''$ be the log scheme
   obtained by endowing $\und X'$ with the log structure induced from
   $X$. There exist a unique morphism $h \colon X'' \to X'$ 
    such that $\und h$ is the identity.
\item     
The homomorphism $f^\flat$ induces an isomorphism
$P \to \cM_{X''/X'}$.   Thus there is a 
    unique homomorphism $\rho \colon  \cM_{X''} \to P_{X'}$ 
    such that $\rho\circ f''^\flat = \id$.
The homomorphism $h^\flat \colon \cM_{X'} \to \cM_{X''}$
     induces isomorphisms  $\cM_{X'} \cong \rho^{-1}(0)$ and
  $\ocM_{X'}^\g \to \ep^*(\cM_{X/S}^\g)$.  
     \end{enumerate}
\end{corollary}
\begin{proof}
Let $\cK_X \subseteq \cM_X$ be the sheaf of ideals
generated by $f^\flat(\cM_S^+)$.  Since $X/S$ is saturated,
this is a radical  sheaf of ideals of $\cM_X$~\cite[I, 4.8.14]{o.llogg}.
Since $X \to S$ is smooth, so is the base changed map 
$(X, \cK_X) \to (S,\cM_S)$, and  since
$(S, \cM_S^+) \to \und S$ is
smooth,  it follows that $(X, \cK_X) \to \und S$ is  smooth. 
Note that if $x \in U$, then $P^\g \to \cM_{X,\xx}^\g$
is an isomorphism, and since $f$ is exact, it follows
that $P = \cM_{X,\xx}$ and hence that $\cK_{X,\xx} = \cM^+_{X,\xx}$.
Conversely, if $\cK_{X,\xx} = \cM^+_{X,\xx}$,  then $P^+$
and $\cM_{X,\xx}^+$ both have height zero, and since $f$
is saturated it follows from statement (2) of \cite[I, 4.8.14]{o.llogg} that
$P \to \cM_{X,\xx} $ is an isomorphism
and hence that $\cM_{X/S} = 0$.  Thus the open set $U$
defined here is the same as the set $U$  defined in 
Theorem~\ref{idealcut.t}.  Hence statements (1), (2), and (3)
follow from that result. 

We check that the map $P  \to \cM_{X''/X'}$ is an isomorphism
locally on $X$, with the aid of a chart as in the proof
of Theorem~\ref{idealcut.t}.  Then $\cM_{X''/X'} =
Q/F$, where $F$ is the face corresponding to a minimal prime $\pee$
of the ideal $K$ of $Q$ generated by $P^+$.  Then  $Q/F$ and $P$
have the same dimension, so $ \pee  \subseteq Q$  and $ P^+ \subseteq
P$  have the same height.  Then
it follows from  (2) of \cite[4.18.4]{o.llogg} that
the homomorphism $P \to Q/F$ is an isomorphism.
It remains only to prove that the map $\ocM_{X'}^\g \to
\ep^*(\cM_{X/S}^\g)$
is an isomorphism.  We have a commutative diagram:
\begin{equation}\label{mxsg.d}
\xymatrix{
& \ep^{*}(\cM^\g_{X/S}) \cr
\ocM_{X'} ^\g\ar[r]\ar[ru] & \ocM_{X''}^\g\ar[u] \ar[r]& \cM^\g_{X''/X'} \cr
& f'^{-1}(\ocM_S^\g)\ar[u]\ar[ur]_\sim
}
\end{equation}
The rows and columns of this diagram are short exact sequences, and
the diagonal map on the bottom right is an isomorphism.  It follows
that the diagonal map on the top left is also an isomorphism.
\end{proof}

\begin{proposition}\label{pastelog.p}
With the hypotheses of Corollary~\ref{vertons.c}, let 
$g \colon X'' \to   X'\times S$ be the morphism
induced by $f\circ \ep$ and $h$.  The morphism
of underlying schemes $\und g$ is an isomorphism,
and $g^\flat$ induces an isomorphism of abelian sheaves:
$g^{\flat\g}\colon \cM_{X'\times S'}^\g \to \cM_{X''}^\g$. 
 The horizontal arrows in the
commutative diagram
\[\xymatrix{
X''_\lag \ar[r]^{g_\lag} \ar[d]_{\tau_{X''}} & X'_\lag \times S_\lag \ar[d]^{pr\circ\tau_{X'}}\cr
X''_\tp \ar[r]^{g_\tp} & X'_\tp
}\]
are isomorphisms.
  \end{proposition}
  \begin{proof}
    Since $\und h$ is an isomorphism and $\und S$ is a point,
the morphism $\und g$ is also an isomorphism.  Since 
$\ocM_{X'\times S} = \ocM_{X'} \oplus P_{X'}$ and
$P_{X'} \cong \cM_{X''/X'}$, it follows from the
horizontal exact sequence in diagram~(\ref{mxsg.d}) that 
the homomorphism $\ocM_{X'\times S}^\g \to \ocM_{X''}^\g$ is an isomorphism,
and hence the same is true of  $g^{\flat\g}$.  
It follows that $g_\lag$ is bijective, and since it is proper, it is a homeomorphism.
  \end{proof}

The corollary below shows that
the fibration $X_\lag \to S_\lag \cong \lt P$ can be cut into pieces
(the connected components of $X'_\lag$),  each of which
is  a trivial fibration
whose fiber is   a manifold with boundary, in a canonical way.  We shall make the
gluing data  needed to undo the cuts more explicit in the case
of curves; see \S\ref{lcn.ss}.

  \begin{corollary}\label{glue.c}
    With the hypotheses of Corollary~\ref{vertons.c}, there is a natural
    commutative  diagram:
\[\xymatrix{
X'_\lag \times  \lt P \ar[rd]_{pr} \ar[r]^-p & X_\lag \ar[d]^{f_\lag} \cr
&\lt P
}\]
where $X'_\lag$ is a topological manifold with boundary and  where $p$ is a proper surjective morphism with finite fibers and is an
isomorphism over  $U_\lag$.  
  \end{corollary}
  \begin{proof}
Let $p :=  \ep_\lag \circ g_\lag^{-1} $, which is proper and surjective
and has  finite fibers.     Recall  from \cite[2.14]{no.rr} that $X'_\lag$ is a topological manifold with boundary,
and that its boundary is $Y'_\lag$. 
  \end{proof}

\subsection{Log nearby cycles}
\label{lognearby.s}


Let $f \colon X  \to S$ be a morphism of fine saturated log schemes, where $S$ is the split log point associated to a sharp monoid $P$. We assume that for every $x \in X$, the map $P^\g \to \ocM^\g_{X,x}$ is injective, and that the quotient group $\ocM_{X/S.x}^\g$ is torsion free. These assumptions hold if, for example, $f$ is smooth  and saturated. We form the following commutative diagram: 
\begin{equation}\label{taudiag.e}
  \xymatrix@R=3em@C=3em{
    \tilde X_\lag \ar@/_2pc/[dd]_{\tilde f}\ar[d]^{\tilde \tau_{ X/S}}
    \ar[r]^{\eta_X}\ar@/^3.5pc/[rrd]^{\tilde \tau_X}
& X_\lag \ar[d]_{\tau_{X/S}} \ar[dr]^{\tau_X}  \\
    X_{\tilde S,\lag} \ar[r]^\eta \ar[d]^{\tilde f_S^\lag}  & X_{S,\lag} \ar[r]^{\tau_{X_S}} \ar[d]_{f_S^\lag} & X_\tp \ar[d]_{f_\tp} \\
    \tilde S_\lag \ar@/_1.5pc/[rr]_{\tilde \tau_S}\ar[r]^{\eta_S} & S_\lag \ar[r]^{\tau_S} & S_\tp
  }
\end{equation}
where the squares are Cartesian. 
Thus $S_\lag \cong \lt P$, $\tilde S_\lag \cong \lv P$, $  X_{\tilde S,\lag} = X_\tp\times \tilde S_\lag$, and $\tX_\lag = X_\lag\times_{S_\lag} \tilde S_\lag$.  
We let  $\tilde \tau_{X}:= \tau_{X} \circ \eta_X $, $\tilde \tau_{X_S} := \tau_{X_S} \circ \eta$, and $\tilde \tau_S := \tau_X  \circ \eta_S$, so that we have  the diagram:
\begin{equation}\label{taudiag2.e}
  \xymatrix{
    \tilde X_\lag \ar[r]^-{\tilde \tau_{ X/S}} \ar[dr]_{\tilde\tau_X} & X_\tp \times \lv P \ar[d]^{\pi} \ar[r]^\cong & X_{\tilde S,\lag}\ar[ld]^{\tilde \tau_{X_S}} \\
    & X_\tp. 
  }
\end{equation}

The logarithmic inertia group $\li P$ acts on $\tilde S_\lag$ over $S_\tp$ and hence also on  $\tX_\lag$ over $X_\tp$. Our goal is to describe the cohomology of $\tilde X_\lag$, together with its  $\li P$-action, using this diagram and the log structures on $X$ and $S$.  We set 
\[
  \Psi^q_{X/S} := R^q\tilde \tau_{X*} \bz, \quad (\text{resp.} \quad \Psi_{X/S} := R\tilde \tau_{X*} \bz),
\]
viewed as a sheaf (resp. object in the derived category of sheaves) of
$\bz[\li P]$-modules  on $X_\tp$. When $S$  is the standard log point and $f$ is obtained by base change from a smooth proper morphism over the standard log disk, 
the complex $\Psi_{X/S}$ can be identified with the usual complex of nearby cycles, as was proved in \cite[8.3]{ikn.qlrhc}. Then $H^*(\tilde X_\lag, \bz) \cong H^*(X_\tp, \Psi_{X/S})$, and there is the   (Leray) spectral sequence
\[
  E_2^{p,q} = H^p(X_\tp, \Psi^q_{X/S} ) \quad \Rightarrow \quad H^{p+q}(\tX_\lag, \bz).
\]
Our first ingredient, due  to Kato and Nakayama~\cite[Lemma 1.4]{kkcn.lblelsc}, is the  following computation of the cohomology sheaves $\Psi_{X/S}^q$.

\begin{theorem}[Kato and Nakayama] \label{logsymb.t}
  Let $f \colon X \to S$ be a saturated morphism of 
log schemes, where $X$ is fine and saturated and $S$ is the split log point over
$\bc$ associated to a fine sharp monoid $P$. Then on the 
topological space $X_\tp$ associated to $ X$,   there are canonical isomorphisms 
  \begin{equation} \label{sigma.e}
    \sigma^q \colon \bigwedge^q \cM_{X/S}^\g(-q) \isomlong \Psi^q_{X/S}
  \end{equation}
  for all $q$.  In particular, the logarithmic inertia group  $\li P$ acts trivially on $\Psi^q_{X/S}$.  
\end{theorem}

\begin{proof}
The construction of these isomorphisms depends on the  logarithmic exponential sequence~(\ref{logexp.e}) on  $X_\lag$.  In the absolute case it is shown in \cite{kkcn.lblelsc} that the boundary map associated to (\ref{logexp.e}) induces a homomorphism $\ocM_X^\g \to R^1\tau_{X*}(\bz(1))$, and then  one finds by cup-product the homomorphisms $\sigma^q$ for all $q \ge 0$. These can be seen to be isomorphisms by using the proper base change theorem to reduce to the case in which $X$ is a log point.

The argument in our relative setting is similar.  Let $\cM_{X/P}$ be the quotient of the sheaf of monoids $\cM_X$ by $P$.  Since $P^\g \to \ocM_X^\g$ is
injective, the sequence
\begin{equation} \label{mptompb.e}
  0 \To \oh X^* \To \cM_{X/P}^\g \To \cM_{X/S}^\g \To 0
\end{equation}
is exact.  The homomorphism $P^\g \to f^{-1} (\cM^\g_S)\to \cM_X^\g$
does not lift to $\mathcal{L}_X$ on $X_\lag$, but the map $\tilde \chi
\colon P^\g \to \mathcal{L}_S$ (defined at the beginning of \S\ref{geam.ss}) defines such a lifting on $\tilde S_\lag$ and hence also on $\tilde X_\lag$.  Letting $ \mathcal{L}_{X/P}$ be the quotient of $\mathcal{L}_X$ by $\tilde \chi(P^\g)$, we find an exact sequence:
\begin{equation} \label{tildel.e}
  0 \To \bz(1) \To \mathcal{L}_{X/P} \To \tilde \tau_X^{-1}(\cM^\g_{X/P}) \To 0
\end{equation}
The boundary map associated with this sequence produces a map 
$\cM^\g_{X/P} \to R^1\tilde\tau_{X*}(\bz(1))$ which factors through 
$\cM^\g_{X/S}$ because, locally on $X$, the inclusion 
$\oh X^* \to \cM_{X/P}$ factors through $\tilde \tau_{X*} (\mathcal{L}_{X/P})$.  Then cup product induces maps $\bigwedge^q \cM_{X/S}^\g \to R^q\tilde \tau_{X*}(\bz(1))$ for all $q$, which we can
check are isomorphisms on the stalks.   The map $\ttau_{X/S}$ is proper, and its fiber
over a point $(x,v)$ of $X_\tp\times \lv P$ is a torsor under $\Hom(\cM_{X/S,x},\sone)$.  
It follows that the maps
$\bigwedge^q \cM_{X/S,x}^\g \to (R^q\tilde \tau_{X/S*}(\bz(1))_{(x,v)}$  are isomorphisms.
In particular, the sheaves $R^q\ttau_{X/S*} (\bz(1))$ are locally constant along the fibers
of $\pi\colon  X_\tp \times \lv P
 \to \lv P$.  Then it follows from \cite[2.7.8]{kash.sm} that the map $\pi^*R\pi_*(R\ttau_{X/S*}(\bz(1))
\To R\ttau_{X/S*}(\bz(1))$ is an isomorphism.  Thus the maps 
$(R\ttau_{X*}(\bz(1))_x \to (R\ttau_{X/S*}(\bz(1))_{(x,v)} $ are  isomorphisms, and the result follows.
\end{proof}


Our goal is to use the Leray spectral sequence for the morphism $\ttau_X$ to describe the cohomology of $\tilde X_\lag$ together with its monodromy action.  In fact it is convenient to work on the level of complexes, in the derived category. The `first order attachment maps' defined in \S\ref{ss:homalg-notation} are maps
\[
  \delta^q : \Psi^q_{X/S}\To \Psi^{q-1}_{X/S}[2].
\]

On the other hand, the ``log  Chern class'' sequence~(\ref{logch.e}) defines a morphism
\[
  {\rm ch}_{X/S} \colon \cM_{X/S}^\g \To \bz(1)[2] 
\]
and hence for all $q \ge 0$, maps:
\[ 
  {\rm ch}^q_{X/S} \colon \bigwedge^q\cM^\g_{X/S} \To \bigwedge^{q-1}\cM_{X/S}^\g(1)[2],
\]
defined as the composition
\[ 
  \bigwedge^q\cM_{X/S}^\g \rTo{\eta} \cM^\g_{X/S}\ot \bigwedge^{q-1}\cM_{X/S}^\g
\rTo{{\rm ch}_{X/S}\ot \id}  \bz(1)[2]\ot \bigwedge^{q-1}\cM^\g_{X/S}\cong
\bigwedge^{q-1}\cM^\g_{X/S}(1)[2],
\]
where $\eta$ is the comultiplication map as defined in \S\ref{ss:koszul}. We show below that the maps $\delta^q$ and ${\rm ch}^q_{X/S}$ agree, at least after multiplication by $q!$. 

To describe the  monodromy action $\rho$ of $\li P$  on $\Psi_{X/S}$, observe that, since  each $\gamma \in \li P$ acts trivially on $\Psi^q_{X/S}$, the
endomorphism  $\lambda_\gamma:= \rho_\gamma- \id$ of $\Psi_{X/S}$ annihilates $\Psi^q_{X/S}$ and hence induces maps (see \S\ref{ss:zero-on-cohomology})
\[ 
  \lambda^q_\gamma: \Psi^q_{X/S} \To \Psi_{X/S}^{q-1}[1].
\]
On the other hand,  the pushout of the ``log Kodaira--Spencer'' sequence~(\ref{logks.e}) along $\gamma  \colon P^\g \to \bz(1)$
is  a sequence: 
\[ 
  0 \To \bz(1) \To \ov \cM_{X,\gamma}^\g \To  \cM_{X/S}^\g \To 0.
\]
The stalk of this sequence at each point of $X$ is a splittable
sequence of finitely generated free abelian groups, so the exterior power construction of \S\ref{ss:koszul} provides a sequence
\[
  0 \To  \bigwedge^{q-1}\cM^\g_{X/S} (1)\To \bigwedge^q\ocM_X^\g/K^2\bigwedge^q\ocM_X^\g \To \bigwedge^q \cM_{X/S}^\g \To 0,
\]
which  gives rise to a morphism in the derived category
\begin{equation} \label{logksq.e}
  \kappa_{\gamma}^q \colon  \bigwedge^q \cM_{X/S}^\g \To \bigwedge^{q-1}\cM^\g_{X/S} (1)[1]
\end{equation}
Recall from Proposition~\ref{cupxi.p}
 that $\kappa_\gamma$  is ``cup product with $\kappa$,''  that is, that
$\kappa_\gamma  = (\id\ot \kappa)\circ(\id\ot \gamma)\circ \eta$.
 We show below that this morphism agrees with the monodromy morphism $\lambda_\gamma^q$ up to sign.  We shall provide a version of this result for the \'etale topology in Theorem~\ref{monthmetale.t}.
A similar formula, in the context of a semistable reduction and \'etale cohomology, is at least implicit in statement (4) of a result \cite[2.5]{sai.wsil} of T.~Saito.

\begin{theorem} \label{monthm.t}
  Let $S$ be the  split log point associated to a fine sharp and
  saturated monoid $P$ and let $f \colon X \to S$ be a saturated
  morphism of fine saturated log analytic spaces.
  \begin{enumerate}
    \item  For each $q \ge 0$, the following diagram commutes:
      \[ 
        \xymatrix{
          \bigwedge^q \cM_{X/S}^\g(-q)  \ar[r]^-{{\rm ch}_{X/S}^q} \ar[d]_{q!\sigma^q} & \bigwedge^{q-1}\cM_{X/S}^\g{(1-q)}[2] \ar[d]^{q!\sigma^{q-1} } \\
          \Psi^q_{X/S} \ar[r]_{\delta^q} &  \Psi^{q-1}_{X/S}[2].
        }
      \]
    \item  For each $q\ge 0$  and each $\gamma \in \li P$, the following diagram commutes:
      \[ 
        \xymatrix{
          \bigwedge^q \cM_{X/S}^\g(-q) \ar[r]^-{\kappa^q_\gamma} \ar[d]_{\sigma^q} & \bigwedge^{q-1} \cM_{X/S}^\g(1-q) [1] \ar[d]^{\sigma^{q-1}} \\
          \Psi^q_{X/S}  \ar[r]_-{(-1)^{q-1}\lambda^q_\gamma}  & \Psi^{q-1}_{X/S} [1] .
        }
      \]
  \end{enumerate}
\end{theorem}

\begin{proof}
The main ingredient in the proof of statement (1) is the quasi-isomorphism:
\begin{equation} \label{steenfor.e}
  \left [ \oh X \rTo\exp \cM^\g_{X/P} \right ] \rTo\sim \tau_{\leq 1} \Psi_{X/S}(1),
\end{equation}
which is obtained as follows.
The exact sequence~(\ref{tildel.e}) defines an isomorphism in $D^+(\tilde X,\bz)$ 
\[
  \bz(1) \rTo\sim \left[ \mathcal{L}_{X/P}\to  \ttau_X^{-1}(\cM_{X/P}^\g) \right],
\]
and there is an evident  morphism of complexes:
\[ 
  \left[ \ttau_X^{-1}(\oh X) \to \ttau_X^{-1}(\cM_{X/P}^\g )\right]  \To
  \left[ \mathcal{L}_{X/P}\to  \ttau_X^{-1}(\cM_{X/P}^\g) \right],
\]
defined by the homomorphism $\ep \colon \tau_{X}^{-1}(\oh X) \to \mathcal{L}_X $~(\ref{otol.e}). Using these two morphisms and adjunction, we find a morphism
\[
  \left [ \oh X \rTo\exp \cM^\g_{X/P} \right ] \To R\ttau_X(\bz(1)) :=\Psi_{X/S}(1).
\]
Since this morphism induces an isomorphism on cohomology sheaves in degrees $0$ and $1$, it induces a quasi-isomorphism after the application of the truncation functor $\tau_{\leq 1}$. This is the quasi-isomorphism~(\ref{steenfor.e}). Since the map $\delta^1$ of the complex $(\oh X \rTo\exp \cM^\g_{X/P})$ is  precisely the map $\ch_{X/S}$,  we see that the diagram in statement (1) commutes when $q = 1$.  

To deduce the general case, we use induction and the multiplicative structure on cohomology.  Let $E:= \cM_{X/S}^\g(-1)$ and let $F := \bz[2]$.  Using the isomorphisms $\sigma^q$,
we can view $\delta^q$ as a morphism $\bigwedge^q E \to \bigwedge^{q-1} E[2] = F \ot \bigwedge^{q-1} E$. 
Lemma~\ref{d2mult.l}  asserts that the family of maps $\delta^q$
form a derivation in  the sense that  diagram~(\ref{derivedder.e}) commutes.   Then by 
 the definition of  $\ch^q_{X/S}$, it follows from Proposition~\ref{alunique.p}
that $q!\ch_{X/S}^q = q!\delta^q$ for all $q$. 

We defer the proof of the monodromy formula described in statement (2) to Section~\ref{steenbrink.s} (with complex coefficients) and Section~\ref{monpf.s} (the general case).
\end{proof}

\section{Monodromy and the  Steenbrink complex}
\label{steenbrink.s}

Our goal in this section is to extend Steenbrink's formula~(\ref{steenfor.e}) for $\tau_{\leq 1} \Psi_{X/S}$ to all of $\Psi_{X/S}$.
We shall see that
there is a very natural logarithmic generalization of the classical
 Steenbrink complex ~\cite[\S2.6]{stee.lhs} which computes
the logarithmic nearby cycle complex
 $\bc\ot \Psi_{X/S}$.
  The advantage of this complex is that  it is a canonical
  {differential graded algebra} with an explicit action of $\li P$, from which it is straightforward to
prove the monodromy formula of Theorem~\ref{monthm.t}
(tensored with $\bc$). 
Since the construction is based on logarithmic de Rham cohomology, 
we require that $X/S$  be (ideally) smooth.
 Note that once we have tensored with $\bc$,
there is no point in keeping track of the Tate twist, since
there is a~canonical isomorphism $\bc(1)  \isommap \bc$.

\subsection{Logarithmic construction of the Steenbrink complex}
 Steenbrink's  original construction,
 which took place in the context of a  semistable
family of analytic varieties over a complex disc with parameter $z$,  was obtained by formally adjoining 
the powers of $\log z$ to the   complex of differential forms with log poles. Our construction
is based on the logarithmic de Rham complex on $X_\lag$ constructed in~\cite[\S3.5]{kkcn.lblelsc}.

Let us begin by recalling Kato's construction of the logarithmic de Rham complex on $X$~\cite{kato.lsfi,ikn.qlrhc}.
If $f \colon X \to Y$ is a morphism of log analytic spaces,
the sheaf of logarithmic differentials $\Omega^1_{X/Y}$  is  
the universal   target of a pair   of maps 
\[
  d \colon  \oh X \to \Omega^1_{X/Y} , \quad  d\lag \colon \cM_X^\g \to \Omega^1_{X/Y},
\]
where $d$ is a derivation relative to $Y$,
 where $d\lag$ is a homomorphism of abelian sheaves
annihilating the image of $\cM^\g_Y$, and where
$d\alpha_X(m) = \alpha_X(m) d\lag (m)$ for every local section $m$ of $\cM_X$.
The sheaf $\Omega^1_{X/Y}$ is locally free if $f$ is a smooth morphism
 of (possibly idealized) log spaces.
Then $\Omega^i_{X/S} := \bigwedge^i \Omega^1_{X/Y}$, and  there is a natural way
to make $\bigoplus \Omega^i_{X/Y}$ into a complex satisfying the usual
derivation rules and such that $d \circ d\lag = 0$.  In particular the map
$d\lag\colon \cM_X^\g \to \Omega^1_{X/Y}$ factors through the sheaf of closed one-forms,
and  one finds maps:
\begin{equation}
  \label{sigmadr.e}
\sigma_{DR}\colon \bc \ot \bigwedge^i  \cM_{X/Y}^\g \to \mathcal{H}^i(\Omega^\cx_{X/Y}).
\end{equation}
When  $S = \bc$ (with trivial log structure) and 
$X/\bc$ is ideally log smooth, these maps fit into a~commutative 
diagram of  isomorphisms~(c.f. \cite[4.6]{kkcn.lblelsc} and its proof):
\[\xymatrix{
\bc \ot \bigwedge^i  \ocM_{X}^\g \ar[r]^-{\sigma_{DR}}\ar[rd]_\sigma& \mathcal{H}^i(\Omega^\cx_{X/\bc}) \cr
&R^i\tau_{X*}(\bc)\ar[u]
}\]

As explained in \cite[\S3.2]{kkcn.lblelsc},
to obtain  the de Rham complex on $X_\lag$, one  begins with 
  the  construction of the universal sheaf of $\tau_X^{-1}(\oh X)$-algebras $\oh X^\lag$
which fits into a~commutative diagram:
\[\xymatrix{
\mathcal{L}_X \ar[r] & \oh X^\lag \cr 
& \tau_X^{-1}(\oh X). \ar[lu]^\ep\ar[u]
}\]
This sheaf  of $\tau_X^{-1}(\oh X)$ modules 
admits a unique integrable connection
\[
  d \colon \oh X^\lag \to  \oh X^\lag \ot_{\tau_X^{-1}(\oh X)} \tau_X^{-1}(\Omega^1_{X/\bc})
\]
such that $d (\ell) = d\pi(\ell) $ (see (\ref{logexp.e}))  for each section $\ell $ of $\mathcal{L}_X$ and
which is compatible with  the multiplicative structure of $\oh X^\lag$. 
 The de~Rham complex of this connection 
 is a complex whose terms are sheaves of  $\oh X^\lag$-modules 
 on $X_\lag$, denoted by $\Omega^{\cx,\lag}_{X/\bc}$.  In particular,
$\Omega^{i,\lag}_{X/\bc} := \oh X^\lag \ot_{\tau^{-1}_X(\oh X)} \tau_X^{-1}(\Omega^i_{X/\bc})$

When  $S$ is the split log point associated to a fine sharp  saturated
monoid $P$, the sheaf  $\oh S^\lag$
on  the torus $S_\lag \cong \lt P$ 
 is locally constant, and hence is determined by
$\Gamma(\tilde S_\lag ,\eta^*(\oh S^\lag))$ together with its natural action of $\li P$.   These
data are easy to describe explicitly.
The  structure sheaf $\oh S$ is $\bc$
and $\Omega^1_{S/\bc}$ is $\bc \ot P^\g$.
Twisting the exact sequence~(\ref{univextp.e}) yields the sequence:
\[
  0 \To \bz \To \lL P(-1) \To P^\g(-1) \To 0.
\]
For each $n$, the map $\bz \to \lL P(-1)$ induces a map 
 $S^{n-1}( \lL P(-1)) \to S^n (\lL P(-1))$, and we let 
\[ \toh P^\lag := \dirlim S^n (\lL P(-1)).\]
  The action of $\li P$ on $\lL P$ induces an action on
 $\toh P^\lag$, compatible with its  ring structure.   Let $N_n \toh  P$ denote 
the  image of the map $S^n (\lL P(-1)) \to \dirlim S^n (\lL P(-1)) = \toh P^\lag$. 
Then $N_\cx$ defines an $\li P$-invariant filtration on $\toh P^\lag$.  The action of $ \li P$ on
$\gr^N_n \toh P^\lag \cong S^n (P^\g(-1))$
is trivial and  thus the action on $\toh P^\lag$ is unipotent.

    The splitting $\tilde \chi$ defines a splitting $P^\g(-1) \to \lL P(-1)$ and thus an isomorphism
\[ \toh P^\lag \cong \oplus_n\gr^N_n \toh P^\lag \cong \oplus_n S^n (P^\g(-1));\] 
this isomorphism is ``canonical'' but not $\li P$-equivariant.

For $\gamma \in \li P$, denote by $\rho_\gamma$ the corresponding automorphism of $\toh P^\lag$, and let 
\begin{equation}\label{lambdagamma.e}
  \lambda_\gamma := \log (\rho_\gamma) := \sum_i(-1)^{i+1} (\rho_\gamma- \id )^i/i.  
\end{equation}
The above formula defines, a priori, an endomorphism of $\bq\ot \toh P^\lag$, 
 but, as we shall soon see,  in fact this endomorphism preserves $\toh P^\lag$.

\begin{claim} \label{lambda.c} 
  For $\gamma \in \li P = \Hom(P^\g(-1), \bz)$, the endomorphism $\lambda_\gamma$ of 
$\bq\ot\toh P^\lag$   defined above
is given by interior multiplication with $\gamma$:
\[\toh P^\lag \cong S^\cx(P^\g(-1)) \rTo\eta P^\g(-1)\ot S^\cx (P^\g(-1)) \rTo{\gamma \ot \id}
S^\cx(P^\g(-1)) \cong \toh P^\lag, \]
where $\eta$ is the map defined in \S\ref{ss:koszul}. 
  The subspace $N_n \toh P^\lag$ of $\toh P^\lag$ is the annihilator of  the ideal $J^{n+1}$ of the group algebra $\bz[\li P]$. 
\end{claim}

\begin{proof}
 Let $V:=\bq \ot  P^\g (-1) $  and  let $\phi$ be an element of $\Hom( V, \bq)$.   Interior
multiplication by $\phi$ is the unique derivation $\lambda$ of the algebra $S^\cx V$ such that
$\lambda(v)  = \phi(v)$ for all $v \in V$.   There is also a unique 
automorphism $\rho $ of $S^\cx V$  
such that $\rho(v) = v + \phi(v)$ for all  $v \in V$.   We claim 
that $\lambda = \log \rho$, or, equivalently, that $\rho = \exp \lambda$.
(These are well-defined because $\rho - \id$ and $\lambda$ are locally nilpotent.)
  Since $\lambda$ is a derivation of $S^\cx V$, we have 
\[ \lambda^k (ab)/k! = \sum_{i+j= k} (\lambda^i a /i!) (\lambda^j b/j!),\]
hence 
\[
  \exp(\lambda)(ab) =  \sum_k \lambda^k(ab)/k! = \left (\sum_i \lambda^i (a)/i! \right)
  \left (\sum_j \lambda^j (b )/j! \right) = \exp(\lambda(a))\cdot \exp(\lambda(b)).
\]
Thus $\exp \lambda$ is an automorphism  of the algebra $S^\cx V$.  Since it sends
$v$ to $v + \phi(v)$, it agrees with $\rho$, as claimed. 

If $v_1, v_2, \ldots, v_n$ is a sequence of elements of $V$, then
\begin{align*}
\rho(v_1 v_2 \cdots v_n) &= (v_1 + \phi(v_1)) (v_2+ \phi(v_2)) \cdots (v_n + \phi(v_n))   \\
 & =  v_1 v_2 \cdots v_n + \sum_i \phi(v_i) v_1 \cdots \hat v_i \cdots v_n  + R,
\end{align*}
where the symbol $\hat v_i$ means that the $i$th element is omitted
and where $R \in N_{n-2}S^nV$.  In particular, $\rho -\id$ maps $N_n S^\cx V$ to
$N_{n-1} S^\cx V$ and acts on $\gr_n S^\cx V \cong S^n V$ as interior multiplication by $\phi$. 
Since $\gr^N S^\cx(P^\g(-1))$ is torsion free, the analogous results hold for $ S^\cx(P^\g(-1))$.
The augmentation ideal $J$ of the group algebra $\bz [I_P]$ is generated by
elements of the form $\gamma - 1$, and it follows that $J$ takes $N_n \toh P^\lag$ 
to $N_{n-1} \toh P^\lag$ and hence that $J^{n+1}$ annihilates $N_n\toh P^\lag$.  Moreover, the natural map
$S^n \li P \to J^n/J^{n+1}$ is an isomorphism and identifies the pairing
$J^n/J^{n+1} \times \gr_n^N \toh P^\lag \to \bz$ with the  standard pairing
$S^n\li P \times S^n P^\g(1) \to \bz$.  Since this pairing is nondegenerate  
over $\bq$, it follows that $N_n \toh P^\lag$ is the annihilator of $J^{n+1}$. 
\end{proof}

The map $\tilde \chi  \colon P^\g \to \lL P$ defines a homomorphism $P^\g \to \eta_S^*( \mathcal{L}_S)$ and hence  also  a~homomorphism $P^\g \to \tilde \tau_{S*}\eta_S^*(\oh S^\lag)$.  In fact, one checks easily that the induced map
\begin{equation} \label{ohslag.e}
   \bc\ot \toh P^\lag \to \Gamma(\tilde S,\eta_{S}^*(\oh S^\lag))
\end{equation}
is an isomorphism, compatible with the action of $\li P$.  The map 
$d \colon  \toh {S}^\lag  \to  \toh S^\lag \ot \Omega^1_{S/\bc} $ identifies with the 
map 
\begin{equation}\label{dlogs.e}
\eta \colon \bc \ot S^\cx P^\g \to \bc \ot S^\cx P^\g \ot P^\g:
p_1 \cdots p_n \mapsto  \sum_i p_1 \cdots \hat p_i \cdots p_n \ot p_i ,
\end{equation}
 and  the action of $\gamma \in \li P$ on  $\toh P^\lag$ is given by
the unique ring homomorphism  taking 
$p \ot 1$ to $ p\ot 1 + \gamma(p)$.

More generally, suppose that $x$ is
 a point of a  fine saturated log analytic space $X$.
Let $Q := \ocM_{X,x}$ and choose a splitting
of the map $\cM_{X,x} \to Q$.   This splitting
induces an isomorphism $\tau_X^{-1}(x)  \cong \lt Q$,
which admits a universal cover $\lv Q \to \lt Q$.
An element $q  $ of $Q^\g$ defines a function
$\lv Q \to \br(1)$ which in fact is a global section
of the pullback of $\cLL_X \subseteq \oh X^\lag$ to $\lv Q$.  
Since  $\oh X^\lag$ is a sheaf of rings, there is an induced
ring homomorphism:
$S^\cx (Q^\g) \to \Gamma(\lv Q, \oh X^\lag)$.  
These constructions result in the  Proposition~\ref{pttilde.p} below.
For more details, we refer again to \cite[3.3]{kkcn.lblelsc}, \cite[{3.3.4}] {o.lrhc}, 
and~\cite[{V.\S3.3}]{o.llogg}.

\begin{proposition}\label{pttilde.p}
Let $x$ be a point of  fine saturated log analytic space $X$.
Then  a~choice of a splitting $\cM_{X,x} \to Q:= \ocM_{X,x}$
yields:
\begin{enumerate}
\item an isomorphism: $\tau_X^{-1}(x) \isomlong  \lt Q:=\Hom(Q, \sone)$,
\item a universal cover:  $ \lv Q := \Hom(Q, \br(1)) \to \tau_X^{-1}(x)$,
\item for each $i$, an isomorphism
$\Omega^i_ {X,x} \ot S^\cx Q^\g\isomlong
\Gamma(\lv Q,  \tilde \eta_x^{-1}(\Omega^{i,\lag}_X)),$ 
where 
\[ 
  \tilde \eta_x \colon V_Q \to T_Q  \to \tau_X^{-1}(x) \to X_\lag
\]
is the natural map.
\end{enumerate}
If $\gamma \in \li P := \Hom(P^\g, \bz(1))$ then the action of 
$\rho_\gamma$  on  $ \Gamma(\lv X, \oh X^\lag)$ is 
given by $\exp(\lambda_\gamma)$, where $\lambda_\gamma$
is interior multiplication by $\gamma$. \qed
\end{proposition}

Since $\bc\ot\toh P^\lag$ is a module with connection on the log point $S$,
its pull-back $f^*(\toh P^\lag)$ to $X$  has an induced connection
$f^*(\bc\ot\toh P^\lag) \to \toh P^\lag \ot \Omega^1_{X/\bc}$.

In the following definition and theorem
 we use the notation of diagrams~(\ref{taudiag.e}) and~(\ref{taudiag2.e}),
and  if $\cF$ is a sheaf on $X_\lag$ (resp. $S_\lag$), 
we write $\tilde \cF$ for its pullback to $\tilde X_\lag$ (resp. $\tilde S_\lag$).

\begin{definition} \label{steen.d}
Let $f \colon X \to S$ be a smooth morphism of fine saturated log analytic spaces
over the split log point $S$
 associated to a fine sharp monoid $P$.  
The \textit{Steenbrink complex} of $X/S$ is the de Rham complex
\[
  K^{\cx}_{X/S} :=\toh{P}^{\lag} \ot_\bz \Omega^\cx_{X/\bc} 
=  \tilde \tau_{X_S*}\left (\tilde f_S^{\lag*}(\toh S ^\lag)\ot\tilde \tau_{X_{ S}}^* (\Omega^\cx_{X/\bc})\right)
\]
 of the $\oh X$-module with connection $f^*\toh S^{\lag}$, given by
\[  
 S^\cx P^\g  \ot \oh X \to  S^\cx P^\g \ot \Omega^1_{X/\bc}, 
 \quad\quad
 p \mapsto p \ot d\lag \,p
\]
endowed with its natural $\li P$-action.
\end{definition}

\begin{theorem} \label{steenb.t}
  Let $S$ be the split log point associated to a fine sharp and
  saturated  monoid $P$ and  let $f \colon X \to S$ be a smooth
  saturated morphism of fine saturated log analytic spaces.  Let
  $\tilde \Omega^{\cx,\lag}_{X/\bc}: = \eta_X^*(\Omega^{\cx,\lag}_{X/\bc})$ on $ \tilde X_{\lag} := X_\lag\times_{S_\lag} \tilde S_\lag$. Then in the derived category $D^+(X_\tp,\bc[\li P])$ of complexes of sheaves of $\bc[\li P]$-modules on $X_\tp$, there are natural isomorphisms: 
  \[ 
    R\tilde \tau_{X*} (\bc) \isomlong R\tilde\tau_{X*} (\tilde{\Omega}^{\cx,\lag}_{X/\bc})\isomlongleft K^\cx_{X/S} 
  \]
\end{theorem}

\begin{proof}
It is proved in \cite[3.8]{kkcn.lblelsc} that, on the space $X_\lag$,  the natural map
\[
  \bc \To  \Omega^{\cx,\lag}_{X/\bc}
\]
is a quasi-isomorphism.  Its pullback via $\eta_X$ is  a quasi-isomorphism:
\[
  \bc \To  \eta_X^*(\Omega^{\cx,\lag}_{X/\bc}) =\tilde{\Omega}^{\cx,\lag}_{X/\bc}
\]
on $\tilde X_\lag$, invariant under  the action  of $\li P$. 
Applying the derived functor $R\tilde\tau_{X*}$, we obtain the  isomorphism 
\[
  R\tilde \tau_{X*} (\bc) \isomlong R\tilde\tau_{X*}(\tilde \Omega^{\cx,\lag}_{X/\bc})
\]
in the theorem.

The natural map $f_\lag^{-1}(\oh S^\lag )\to \oh X^\lag$ induces a map 
\[
 \tilde f^*(\toh S^\lag) \ot \tilde \tau_X^*(\Omega^\cx_{X/\bc}) \to \tilde \Omega^{\cx,\lag}_{X/\bc},
\]
and hence by adjunction a map
\[ 
 \toh{P}^\lag \ot \Omega^\cx_{X/\bc} \to R\tilde\tau_{X*}(\tilde{\Omega}^{\cx,\lag}_{X/\bc}).
\]
The lemma below shows that this map is an isomorphism and completes the proof of the theorem.
\end{proof}

\begin{lemma}
  The terms of the complex $\tilde \Omega^{\cx,\lag}_{\tilde X}$ 
 are acyclic for $\tilde \tau_{X*}$, and for each $q$ the natural map
  \[
    K^q_{X/S}  \to \tilde\tau_{X*}(\tilde \Omega^{q,\lag}_{X/\bc}) 
  \]
  is an isomorphism. 
\end{lemma}

\begin{proof}
The morphism $\tau_{X/S}$  in diagram (\ref{taudiag.e}) is proper and the left upper square is Cartesian, and hence $\tilde \tau_{X/S}$ is also proper.  Let $\tilde x = (x,\theta)$ be a point in $\tilde X_S^\lag \cong X_\tp \times \lv P$.  By the proper base change theorem, the natural map
\[
  (R^i\tilde \tau_{X/S*}\tilde\Omega^{q,\lag}_{X/\bc})_{\tilde x}  \to H^i(\tilde \tau^{-1}_{X/S}(\tilde x), \tilde \Omega^{q,\lag}_{X/\bc})
\]
is an isomorphism. 
(Here the term on the right
means the $i$th cohomology of the sheaf-theoretic restriction of
$\tilde \Omega^{q,\lag}_{X/\bc}$ to the fiber.) 
The fiber $\tilde \tau_{X/S}^{-1}(\tilde x)$ is a torsor under  the group 
\[
  \lt{X/S,x} :=\Hom(\ov \cM_{X/S,x} \sone) \subseteq \lt{X,x} := \Hom(\ov \cM_{X,x},\sone).
\]
 Hence the fiber is homeomorphic to this torus, and 
 $\tilde \Omega^{q,\lag}_{X/\bc}$ is locally  constant on the fiber, as follows from Proposition~\ref{pttilde.p}.
Since the fiber
is a $K(\pi, 1)$,  its cohomology can be calculated as group cohomology. More precisely, view $x$ as a log point (with its log structure inherited from $X$), so that we have a morphism of log points $x \to S$ and hence a  morphism: $x_\lag  \to S_\lag$.  Then 
a choice of a point $\ov  x$ of  $\tau_X^{-1}(x)$ allows us to make identifications:
\[\tau_X^{-1}(x) \cong x_\lag \cong \lt{X,x} \mbox{ and  }\tilde \tau_{X/S}^{-1}(\ov x) \cong \lt {X/S,x}.\]
 The second  torus has a universal cover $\lv{X/S,x} := \Hom( \ov \cM_{X/S,x}, \br(1))$, and  every locally constant sheaf  $\cF$ on $\lt {X/S,x}$ is constant when pulled back to this cover, so the natural map 
$\Gamma(\lv{X/S,x}, \cF) \to \cF_{\ov x}$ is an isomorphism.
These groups have a natural action
of the covering group $\li {X/S,x} = \Hom(  \cM_{X/S,x}^\g, \bz(1))$.   Then 
\[
  H^i(\tilde \tau_{X/S}^{-1}(\tilde x), \cF) \cong H^i(\li {X/S,x}, \cF_{\ov x}),
\]
 In our case, we have 
\[
  \tilde \Omega^{q,\lag}_{X/\bc,\ov x} = \oh {X,\ov x}^\lag \ot
\Omega^q_{X/\bc,x} \cong S^\cx(\ov \cM_{X,x}^\g) \ot \Omega^q_{X/\bc,x}. 
\]
Choosing  a splitting of $P^\g \to \ov \cM_{X,x}^\g$, we can write
\[
  S^\cx\ov \cM_{X,x}^\g  \cong S^\cx P^\g \ot S^\cx \cM_{X/S,x}^\g,
\]
compatibly with the action of $\li {X/S,x}$.  
Let $V:= \bc \ot \cM_{X/S.x}^\g$, and for $\gamma \in I_{X/S,x}
\subseteq \Hom (V,\bc)$, let $\lambda_\gamma$ denote
interior multiplication by $\gamma$ on $S^\cx V$.  An
 analog of Claim~\ref{lambda.c} shows that $\rho_\gamma = \exp \lambda_\theta$.
Then a standard calculation  shows that
\[
  H^i(\li {X/S,x},  \bc \ot S^\cx  \cM^\g_{X/S,x}) \cong
    \begin{cases} 
      \bc & \text{if }i = 0 \\ 
      0 & \text{if }i > 0.
    \end{cases}
\]

Here is one way to carry out this calculation.   
As we have seen, the representation $(S^\cx V,\rho)$ of $I_{X/S,x}$
is the exponentional of the
locally nilpotent  Higgs field $\lambda \colon S^\cx V \to S^\cx V \ot V$
given by the exterior derivative.
It follows from \cite[1.44]{o.lrhc} that one can use Higgs cohomology
to calculate the group cohomology of such locally unipotent  representations.
In our case the Higgs complex of $\lambda$ identifies with
the de Rham complex of the symmetric algebra $S^\cx V$,
and the result follows. 

We conclude that  $H^i(\tilde \tau^{-1}_{X/S}(\tilde x), \tilde \Omega^{q,\lag}_{  X/\bc})$  vanishes if $i > 0$, and that the natural map
\[
  \toh P^\lag  \ot \Omega^{q,\lag}_{ X/\bc,x} \To  H^0(\tilde \tau^{-1}_{X/S}(\tilde x), \tilde \Omega^{q,\lag}_{  X/\bc})
\]
is an isomorphism. 
Then the proper base change theorem implies that
$R^i\tilde \tau_{X/S*}\Omega^{q,\lag}_{\tilde X/\bc} $ vanishes for $i > 0$ and
that   the natural map
\[ 
  \tilde{\tau}_{X_S}^*(\toh P^\lag  \ot\Omega^q_{X/\bc}) \To \tilde{\tau}_{X/S*} (\Omega^{q,\lag}_{\tilde X/\bc}).
\]
is an isomorphism.   But the map $\tilde \tau_{X_S}$ is just the projection $X_\tp \times \lv P \to X$, so for any abelian sheaf  $\cF$ on $X_\tp$, $R^i\tilde \tau_{X_S*}\tilde \tau_{X_S}^*\cF = 0$ and  $\cF \cong \tilde \tau_{X_S*}\tilde \tau_{X_S}^*\cF$, by \cite[2.7.8]{kash.sm}.  Since  $\tilde \tau_X = \tilde \tau_{X_S} \circ \tilde \tau_{X/S}$,  we conclude that $R^i\tilde \tau_{X*}(\tilde \Omega^{q,\lag}_{X/\bc})$ vanishes if $i > 0$ and that the natural map $\toh{S}^\lag  \ot \Omega^q_{X/\bc} \to \tilde{\tau_{X*}}(\Omega^{q,\lag}_{X/\bc})$ is an isomorphism.   The lemma follows.
\end{proof}

\begin{corollary}\label{cohsteen.c}
  In the situation of Theorem~\ref{steenb.t}, the maps $\sigma_{DR}$ 
\eqref{sigmadr.e} factor through
 isomorphisms:
\[  \bc \ot \bigwedge^q \cM_{X/S}^\g\isomlong \mathcal{H}^q (K^\cx_{X/S}).   \]
\end{corollary}

\begin{proof}
There is an evident inclusion
$\Omega^\cx_{X/\bc} \to K^\cx_{X/S}$, and hence we find natural maps
\[ \bc\ot \ocM_X^\g  \rTo {}\Hh^1(\Omega^\cx_{X/ \bc}) \rTo{} \Hh^1(K_{X/S}^\cx).\]
It follows from the formula~(\ref{dlogs.e}) that the image of
each element of $P^\g$ becomes exact in $K^1_{X/S}$, and hence
this composed map factors through  $\bc\ot\cM_{X/S}^\g$.  The maps
in the statement of the corollary are then obtained by cup product.   We now
have a~commutative diagram
\[\xymatrix{
\bc \ot\bigwedge^q \cM_{X/S}^\g \ar[r]\ar[rd]_{\tilde\sigma} & \Hh^q(K^\cx_{X/S})\ar[d]^\cong  \cr
& R^q\tilde\tau_{X*}(\bc),
}\]
where the vertical arrow is the isomorphism coming from Theorem~\ref{steenb.t}.  
Since $\tilde \sigma$ is an isomorphism by Theorem~\ref{logsymb.t},
the horizontal arrow is also an isomorphism.
\end{proof}

\subsection{Monodromy and the canonical filtration}
The filtration $N_\cx$ of $\toh{S}^\lag $ is stable under $\li P$ and the connection and hence induces a~filtration of the complex $K^\cx_{X/S}$.  Claim~\ref{lambda.c} 
shows that $N_n$ corresponds to the $n$th level of the ``kernel'' filtration defined by the monodromy action on the complex $K^\cx_{X/S}$. We shall see that this filtration coincides up to quasi-isomorphism with the canonical filtration $\tau_\le$.   Since we prefer to work with decreasing filtrations, we set
\[
  N^{k} K^q_{X/S} :=( N_{-k} \toh{S}^\lag ) \ot \Omega^q_{X/\bc},
\]
and
\[ 
  T^{k}K_{X/S}^q := 
    \begin{cases} 
      K_{X/S}^q & \text{if }k \le 0 \\
      0 & \text{otherwise.}
    \end{cases}
\]
In particular, $N^0K^\cx_{X/S} = \Omega^\cx_{X/S}$ and $N^1 K^\cx_{X/S} = 0$, so 
$N^iK^\cx_{X/S} \subseteq T^iK^\cx_{X/S}$ for all $i$,
 that is, 
 the filtration $N^\cx$ is finer than the filtration $T^\cx$.

Recall from \cite[1.3.3]{de.thII} that if $F$ is a filtration of a complex $K^\cx$,
 the  ``filtration d\'ecal\'ee''  $\tilde F$ is the filtration of $K^\cx$ defined by
\[
  \tilde F^iK^n:= \{ x \in F^{i+n}K^n : dx \in F^{i+n+1}K^n\}.
\]
Then there are natural maps
\[
  E_0^{i,n-i}(K^\cx,\tilde F) =\gr^i_{\tilde F} K^n \to \Hh^n(\gr^{i+n}_F K^\cx)= E^{i+n,-i}_1(K^\cx,F)
\]
inducing   quasi-isomorphisms:
\[  (E_0^{i,\cx}(K^\cx,\tilde F),d^{i,\cx}_0) \to (E_1^{\cx+i,-i}(K^\cx,F), d_1^{\cx+i,-i})\]
This equation says that the natural maps (induced by the
identity map of $K^\cx$),  are
quasi-isomorphisms
\begin{equation} \label{decqis.e}
   (E_0^{-q,\cx}(K^\cx,\tilde F),d^{-q,\cx}_0) \to (E_1^{\cx,q}(K^\cx,F), d_1^{\cx,q})[q]'
\end{equation}
 where the symbol $[q]'$ means the {\em naive shift} of the complex
(which does not change the sign of the differential).  
More generally, there are isomorphisms of spectral sequences, after a suitable 
renumbering~\cite[1.3.4]{de.thII}:
\[  (E_r^{\cx,\cx}(K^\cx,\tilde F),d^{\cx,\cx}_r) \to (E_{r+1}^{\cx,\cx}(K^\cx,F), d_{r+1}^{\cx,\cx}). \]

Let $\tilde N^\cx$ denote the filtration d\'ecal\'ee of $N^\cx$, and similarly for $T^\cx$; note that  $\tilde T^i = \tau_{\le -i}$,  the ``filtration canonique.''  
  Since the filtration  $N^\cx$ is finer than $T^\cx$, the filtration $\tilde N$ is finer than the filtration $\tilde T^\cx$, and we find a morphism of filtered complexes:
\begin{equation} \label{tntott.e}
  (K^\cx_{X/S},  \tilde N^\cx  ) \To (K^\cx_{X/S}, \tilde T^\cx)
\end{equation}

\begin{theorem} \label{tntott.t}
  Let $f \colon X \to S$ be a smooth and  saturated morphism of fine
  saturated log analytic spaces, where $S$ is the split log point
  associated to a sharp toric monoid.
  Then there are natural filtered quasi-isomorphisms:
  \[
    (K^\cx_{X/S},  \tilde N^\cx  ) \rTo\sim  (K^\cx_{X/S}, \tilde T^\cx) \xleftarrow{\sim} (\Psi_{X/S} , \tilde T^\cx)
  \]
\end{theorem}

The existence of the second filtered quasi-isomorphism of the theorem follows from the canonicity of the filtration $\tilde T$ and Theorem~\ref{steenb.t}. The proof that the first  arrow is  a filtered quasi-isomorphism is a consequence
of the following  more precise result.

Recall from Definition~\ref{kozdr.d} that associated to the homomorphism
$$\theta \colon  \bc \ot \ocM_{S^\g} \to \bc\ot \ocM_X^\g$$
we have for each $q$
a complex $\Kos^{q,\cx}(\theta)$ and whose
$n$th term  is  given by
\[Kos^{n,q}(\theta) = \bc \ot S^{q-n} \ocM_{{S^\g}} \ot \Lambda^n\ocM_X^\g.\] 
\begin{theorem} \label{ntot.t}
  Let $f \colon X  \to S$ be as in Theorem~\ref{steenb.t},  let $K^\cx_{X/S}$ be the Steenbrink complex on $X_\tp$, and let 
  \[
    0 \rTo \  \bc \ot\ocM^\g_S \rTo \theta  \bc\ot \ocM^\g_X \rTo  \pi \bc \ot \ocM^\g_{X/S} \rTo \  0
  \]
  be the exact sequence of sheaves of $\bc$-vectors spaces on $X$ obtained by tensoring the log Kodaira-Spencer sequence~(\ref{logks.e}) with $\bc$.  
  \begin{enumerate}
\item For each $q \ge 0$, there are natural morphisms of complexes:
\[
\gr^{-q}_{\tilde N} K^\cx_{X/S} \rTo \sim E^{\cx,q}_1(K^\cx_{X/S},N)[-q]'
\rTo \cong \Kos^\cx_q(\theta) \rTo \sim \bc \ot \bigwedge^q \ocM^\g_{X/S} [-q],
\]
where the first and last maps are quasi-isomorphisms
and the second map is an isomorphism. (The notation $[-q]'$ means the naive shift of the complex, and $\Kos^\cx_q$ is the complex
defined in Definition~\ref{kozdr.d}.)
    \item The morphism of spectral sequences induced by the map of filtered complexes $(K_{X/S}^\cx, N^\cx) \to (K_{X/S}^\cx, T^\cx)$ is an isomorphism at the $E_2$-level and beyond.
    \item The map of filtered complexes $(K_{X/S}^\cx, \tilde N^\cx)  \to (K_{X/S}^\cx,\tilde T^\cx)$ is a filtered quasi-isomorphism.
  \end{enumerate}
\end{theorem}

\begin{proof}
The first arrow in (1) is the general construction of Deligne as expressed in
equation~(\ref{decqis.e}).  
It follows from the definitions that
\[
  E_0^{-p,q}(K^\cx_{X/S},N)  = \gr^{-p}_N K_{X/S}^{q-p} \cong S^p \ocM^\g_S\ot \Omega^{q-p}_{X/\bc}.
\]
Since the elements of $S^p\ocM^\g_S$ are horizontal sections of $\gr^{-p}_N \toh{S}^\lag $, the differential $d_0^{p,\cx}$ of the complex $E_0^{p,\cx}(K^\cx_{X/S},N)$ can be identified with the identity map of $S^p \ocM^\g_S$ tensored with the  differential of $\Omega^\cx_{X/\bc}$.  Then  the isomorphism~(\ref{sigmadr.e}) allows us to write:
\[ E_1^{-p,q}(K^\cx_{X/S}, N) 
  = \begin{cases}
      S^p\ocM^\g_S \ot \Hh^{q-p}(\Omega^\cx_{X/\bc}) \cong \bc\ot S^p\ocM^\g_S \ot  \bigwedge^{q-p} \ocM^\g_{X} & \text{if }0 \le p\le q \\  
      0  & \text{otherwise.} 
    \end{cases} 
\]
The isomorphism appearing above is the identity of $S^p \ocM^\g_S$
tensored with the isomorphism~(\ref{sigmadr.e}).
The differential $d_1^{-p,q}$
becomes identified with a map
\[\xymatrix{\bc\ot S^p \ov \cM^\g_S \ot \bigwedge^{q-p} \ocM^\g_X \ar[r]\ar[d]&
 \bc\ot S^{p-1} \ocM^\g_S \ot \bigwedge^{q-p+1} \ocM^\g_X \ar[d]\\
\Kos^{q-p,q}(\theta) \ar[r]& \Kos^{q-p+1,q}(\theta).
 }\]
It follows from  formula~(\ref{dlogs.e}) that  
this differential is indeed the Koszul   differential.  
Thus we have found  the isomorphism
\[E^{\cx,q}_1(K^\cx_{X/S}, N)[-q]'\rTo \cong \Kos^{\cx,q}(\theta).\]
The quasi-isomorphism
$\Kos^{\cx,q}(\theta) \rTo \sim \bigwedge^q\ocM^\g_{X/S} [-q]$ comes 
from   Proposition~\ref{kozcoh.p}. This completes the proof
of statement (1) of the theorem.

We have  natural maps of filtered complexes
\begin{eqnarray*}
  (K_{X/S}^\cx,N^\cx) & \to &  (K_{X/S}^\cx,T^\cx) \quad \mbox{hence also}\\
  (K_{X/S}^\cx, \tilde N^\cx) & \to &  (K_{X/S}^\cx,\tilde T^\cx).
\end{eqnarray*}
These maps produce the map of spectral sequences in statement (2).
Consider the spectral sequence associated to the filtered complex $(K_{X/S}^\cx, T^\cx)$, in the category of abelian sheaves.   We have 
\[
  E_0^{-p, q}(K^\cx_{X/S}, T) = \gr^{-p}_TK_{X/S}^{q-p}  = 
  \begin{cases} 
    K_{X/S}^{q-p}  & \text{if }p = 0 \\
    0 & \text{otherwise,} 
  \end{cases}
\]
hence an isomorphism of complexes:
\[(E^{0,\cx}_0,d) \cong K^\cx_{X/S}\]
 and of cohomology groups:
\[
  E_1^{-p,q}(K_{X/S}, T) = 
  \begin{cases}
    \Hh^{q}(K_{X/S}^\cx) & \text{if }p = 0 \\
    0 & \text{otherwise.} 
  \end{cases}
\]
Thus the complex $E_1^{\cx,q}(K,T)$ is isomorphic to the sheaf $\Hh^q(K_{X/S})$,
viewed as a complex
 in degree zero, and  the spectral sequence degenerates at $E_1$.  Then
\[
  E_\infty^{0,q}(K,T) =  E_1^{0,q}(K,T)  = \Hh^q(K_{X/S}^\cx) \cong R^q\tilde \tau_X(\bc) \cong
 \bigwedge^q  \bc  \ot \ocM^\g _{X/S},
\]
by Corollary~\ref{cohsteen.c}.
Since the maps are all natural, statement (2) 
of Theorem~\ref{ntot.t} follows.

Using the naturality of the maps in (\ref{decqis.e}), we find for every $i$ a commutative diagram of complexes. 
\[
  \xymatrix{
    \gr^{-i}_{\tilde N}K_{X/S}^\cx  \ar[r] \ar[d] & \gr^{-i}_{\tilde T} K_{X/S}^\cx \ar[d] \\
    E_1^{\cx,-i}(K_{X/S}^\cx,N)\ar[r] &  E_1^{\cx,-i}(K_{X/S}^\cx,T).
  }
\]

Since the bottom horizontal arrow is a quasi-isomorphism and the  vertical arrows are quasi-isomorphisms,  it follows that the top horizontal arrow is also a quasi-isomorphism. Since $\tilde T^i$ and $\tilde N^i$ both vanish for $i > 0$, it follows by induction that  for every $i$, the map $\tilde N^iK^\cx_{X/S} \to \tilde T^iK^\cx_{X/S}$ is a quasi-isomorphism.
\end{proof}

Combining the above results with our study of Koszul complexes in \S\ref{ss:koszul}, we can now give our first proof of the monodromy formula in Theorem~\ref{monthm.t} after tensoring with~$\bc$.

Any  $\gamma  \in \li P$ induces a homomorphism $\ocM_S^\g   \to \bc$, which we denote also by $\gamma$.  By  Proposition~\ref{pttilde.p}
the action of $\tilde \lambda_\gamma := \log (\rho_\gamma)$ on $\toh S^\lag \cong S^\cx \ocM^\g_S$
 corresponds to interior multiplication by $\gamma$. Thus for every $i$,  $\lambda_\gamma$ maps $N^{-i}K_{X/S}^\cx $ to $N^{1-i}K_{X/S}^\cx$ and hence $\tilde N^{-i}K_{X/S}^\cx $ to $\tilde N^{1-i}K_{X/S}^\cx$.
We need to compute the induced map
\[\tilde \lambda_\gamma^i \colon \gr^{-i}_{\tilde N} K^\cx_{X/S} \to \gr^{1-i}_{\tilde N} K^\cx_{X/S}.\]
Using the quasi-isomorphism of statement (1) of Theorem~\ref{ntot.t}, we can 
identify this as the map
$\gamma_i \colon \Kos^\cx_i(\theta) \to  \Kos^\cx_{i-1}(\theta)$ which
in degree $n$ is the composition
\[\xymatrix{
\bc\ot S^{i-n} \ocM^\g_S \ot \bigwedge^{n} \ocM^\g_X \ar[r]^-{\eta\ot\id}\ar[dr]
&\bc\ot  \ocM^\g_S \ot S^{i-n-1}\ocM^\g_S \ot \bigwedge^{n} \ocM^\g_X  \ar[d]^{\gamma \ot \id} \cr
&\bc\ot S^{i-n-1}\ocM^\g_S  \ot \bigwedge^{n} \ocM^\g_X ,
}\]
where $\eta$ is the map defined at begining of \S\ref{ss:koszul}.
In other words, our map is 
the composition of the morphism
\[
  c_q \colon \Kos^{\cx,q}(\theta) \to \ov \cM^\g_S \ot \Kos^{\cx,{q-1}}(\theta),
\]
constructed in Proposition~\ref{vwdr.p} with $\gamma \ot \id$. 
We thus find  a  commutative diagram in the derived category:
\[ 
  \xymatrix{
    \Psi^q_{X/S} \ar[d]^{\lambda_\gamma^q} & \ar[l]_-\sim  \gr^{-q}_{\tilde N} K^\cx_{X/S}[q] \ar[r]^-\sim \ar[d]^{\tilde \lambda_\gamma^q}& \Kos^\cx_q(\theta)[q]\ar[d]^{(\gamma\ot\id)\circ c_q} \ar[r]^-\sim & \bigwedge^q \ocM^\g_{X/S} \ar[d]^{g_q} \\
    \Psi^{q-1}_{X/S}[1] & \ar[l]_-\sim \gr^{1-q}_{\tilde N} K^\cx_{X/S} [q]\ar[r]^-\sim &  \Kos^\cx_{q-1}(\theta)[q] \ar[r]^-\sim  & \bigwedge^{q-1} \ocM^\g_{X/S}[1]
} 
\]
The  horizontal arrows in the leftmost square come from
Theorem~\ref{tntott.t}  and those in the remaining squares
come from statement (1) of Theorem~\ref{ntot.t}.
Statement (3) of Proposition~\ref{vwdr.p} shows that $g_q = (-1)^{q-1}\kappa_\gamma$, and 
statement (2) of Theorem~\ref{monthm.t}, tensored with $\bc$, follows. 

\section{Proof of the integral monodromy formula}\label{monpf.s}

We present a proof of the monodromy formula Theorem~\ref{monthm.t}(2) with integral coefficients. In contrast with the proof with complex coefficients presented in the previous section, this one uses more abstract homological algebra; not only does this method work with $\bz$-coefficients in the complex analytic context, it can be adapted to the algebraic category, using the Kummer \'etale topology, as we shall see in Section~\ref{etale.s}. 

\subsection{Group cohomology}

Our proof of  the monodromy formula with integral coefficients is hampered by the fact that we have no  convenient explicit complex of sheaves of  $\li P$-modules representing $\Psi_{X/S}$.  Instead we will need some abstract arguments in homological algebra, which require some preparation. Recall 
that the \textit{cocone} $\Con'(u)$ of a morphism $u$ is the shift by $-1$ of the cone  $\Con(u)$ of $u$, so that there is a distinguished triangle:
\[\Con'(u) \To A \rTo{u} B \To \Con'(u)[1]. 
\]
In other words, $\Con'(u)$ is the total complex of the double complex $[A\tto{-u} B]$  where $A$ is put in the $0$-th column (that is, ${\rm Fibre}(-u)$ in the notation of \cite{sai.wsil}). Explicitly, $\Con'(u)^n = A^n \oplus B^{n-1}$, $d(a, b) = (da, -u(a)-db)$, $\Con'(u)\to A$ maps $(a, b)$ to $a$ and $B\to \Con'(u)[1]$ maps $b$ to $(0, b)$. 

Let $X$ be a topological space and $\I$ a group.  We identify the (abelian) category of sheaves of $\I$-modules on $X$ with the category of sheaves of $R$-modules on $X$, where $R$ is the group ring $\bz[\I]$.  The functor $\und \Gamma_\I$ which takes an object  to its sheaf of $\I$-invariants identifies with the functor $\Hhom(\bz, - )$, where  $\bz \cong R/J$ and $J$ is the augmentation ideal of $R$.

Now suppose that $\I$ is free of rank one, with a chosen generator
$\gamma$.  Then $\lambda := e^\gamma -1$ (see\   \S 3.1) is a generator of the ideal $J$, and we have an exact sequence of sheaves of $R$-modules:
\[
  0 \To  R \rTo{\lambda}  R  \To \bz \To 0,
\]
which defines a quasi-isomorphism $C_\cx \isommap \bz$, where $C_\cx$ is the complex $[R \tto{\lambda} R]$ in degrees $-1$ and $0$.  The functor $\Hhom(R, -)$ is exact, and hence the functor $\und\Gamma_\I$ can be identified with the functor $\Hhom(C_\cx, -)$.   The $R$-linear dual of $C_\cx$  is the complex
\[
  C^\cx := \left[ R \rTo{-\lambda} R \right], 
\]
(cf. \cite[0.3.3.2]{bbm.tdc} for the sign change) in degrees $0$ and $1$, and for any complex $K^\cx$ of sheaves of $\I$-modules, 
\begin{equation} \label{cxdef.e}
  C(K^\cx):=  \Hhom_R(C_\cx, K^\cx) \cong C^\cx \ot_R K^\cx
\end{equation}
is a representative for $ \R\und \Gamma_\I (K^\cx)$.   Note that
\[
  C^q(K^\cx) = K^q \oplus K^{q-1}, \quad d(x,y) = (dx, -\lambda x -dy),
\]
and thus $C(K^\cx)$ is the cocone of the morphism $\lambda \colon K^\cx \to K^\cx$.  

In particular, $C^\cx = C(R^\cx)$, where, $R^\cx$ is the complex consisting of $R$ placed in degree zero, and we have a quasi-isomorphism
\[
  \ep \colon C^\cx \rTo\sim \bz[-1] \quad\mbox{ given by the
  augmentation $R  \to \bz$ in degree one}.
\]

\begin{proposition}[{Compare with \cite[\S 1]{rz.lzs}}] \label{rgami.p}
  Let $\I$ be a free abelian group of rank one, with generator $\gamma$, let $R:= \bz[\I]$, and  let $C^\cx$ be the complex~(\ref{cxdef.e}) above. For an object $K^\cx$  of the derived category $D_\I(X)$ of sheaves of $\I$-modules on a topological space $X$, let $C(K^\cx) :=  C^\cx \ot_R K^\cx$.
  \begin{enumerate}
    \item There   are natural isomorphisms:
      \[
        C(K^\cx)  \cong R\Hhom_\I(\bz, K^\cx)  \cong \R\und \Gamma_\I(K^\cx)
      \]
      and a distinguished triangle:
      \begin{equation} \label{cocone.e}
        C(K^\cx) \rTo{a} K^\cx \rTo{\lambda}  K^\cx \rTo{b} C(K^\cx)[1] .
      \end{equation}
    \item Let $\partial \colon \bz \to \bz[1]$ denote the morphism defined by the exact sequence~(\ref{groupal.e}) :
      \[
        0 \To \bz \To R/J^2 \To \bz \To 0
      \]
      (the first map sends $1$ to the class of $\lambda$). 
      Then $b \circ a = \partial  \ot \id \colon C(K^\cx) \to C(K^\cx)[1] $.  
    \item There are natural exact sequences:
      \[
        \cdots \To  R^q\und\Gamma_\I (K^\cx) \rTo{a} \Hh^q(K^\cx) \rTo{\lambda} \Hh^q(K^\cx) \rTo{b} R^{q+1}\und \Gamma_\I(K^\cx) \To \cdots
      \]
      and
      \[  
        0 \To R^1\und\Gamma_\I(\Hh^{q-1}(K^\cx))\rTo{b} R^q\und\Gamma_\I(K^\cx) \rTo{a} \und\Gamma_\I(\Hh^q(K^\cx)) \To 0. 
      \]
    \item If the action of $\I$ on $\Hh^q(K^\cx)$ is trivial, $a$ and $b$ induce canonical isomorphisms:
      \begin{align*}
        \Gamma_\I(\Hh^q(K^\cx))& \cong  \Hh^q(K^\cx) \\ 
        \Hh^q(K^\cx) &\cong R^1\Gamma_\I(\Hh^q(K^\cx))
      \end{align*}
  \end{enumerate}
\end{proposition}


\begin{proof}
We have already explained statement (1) in the preceding   paragraphs (the distinguished triangle expresses the fact that $C(K^\cx)$ is the cocone of $\lambda\colon K^\cx\to K^\cx$).  Since  $C(K^\cx) \cong C(R) \ot K^\cx$ and the distinguished triangle in (1) for $K^\cx$ is obtained by tensoring the triangle for $R$ with $K^\cx$, it will suffice to prove (2) when $K^\cx = R$.  In  this case, $a \colon C^\cx \to R$ is  given by the identity map in degree $0$, and $b \colon R \to C^\cx[1]$  is given by the identity map in degree $0$. Thus $b \circ a \colon C^\cx  \to C^\cx[1]$ is the map
\[
  \xymatrix{
    0\ar[r] \ar[d] & R\ar[r]^{-\lambda} \ar[d]^{\id} & R \ar[d] \\
    R \ar[r]^\lambda & R\ar[r] & 0
  }
\]
Composing with the quasi-isomorphism $\ep[1]$, we  find that $\ep[1]\circ  b \circ a$ is given by
\[ 
  \xymatrix{
    R \ar[r]^{-\lambda} \ar[d]_{{\rm aug}} & R \\
    \bz
  }
\]
The pushout of the exact sequence $ 0 \to R \to R \to \bz \to 0$ along $R  \to \bz$ is the sequence~(\ref{groupal.e}).  It follows that the morphism $b\circ a \colon C^\cx \to C^\cx[1]$ is the same as the morphism $\partial \colon \bz \to \bz[1]$ defined by that sequence.   This  proves statement (2).

Since $C(K^\cx)  = \R\und \Gamma_\I(K^\cx)$, the first sequence of  statement (3) is just the cohomology sequence associated with the distinguished triangle in (1); the second sequence follows from the first and the fact that for any $\I$ module $E$, $\Gamma_\I(E) \cong \Ker (\lambda)$ and $\R^1\Gamma_\I(E) \cong  \Cok(\lambda)$. 
Statement (4) follows, since in this case $\lambda = 0$.  
\end{proof}

\subsection{Proof of the monodromy formula}
\label{ss:monpf}

We now turn to the proof of the integral version of  statement  (2)  of Theorem~\ref{monthm.t}. Recall that $\Psi_{X/S} = R\ttau_{X*} \bz$ (cf.\ \eqref{taudiag.e}); let us also set $\Psi_X = R\tau_{X*} \bz$ and $\Psi_S = R\tau_{S*} \bz$.  We begin with the following observation, which is a consequence of the functoriality of the maps $\sigma$ as defined in Theorem~\ref{logsymb.t}. 

\begin{lemma} \label{lemma:ses-comp}
  The following diagram with exact rows commutes:
  \begin{equation} \label{eqn:twoext}
    \xymatrix{
        0 \ar[r] & f^* \ocM^\g_S(-1) \ar[r] \ar[d]^{f^* \sigma_{S}} & \ocM^\g_X(-1)  \ar[r] \ar[d]^{\sigma_X} & \cM^\g_{X/S}(-1) \ar[r] \ar[d]^{\sigma_{X/S}} & 0 \\
        0 \ar[r] & f^* \Psi^1_S \ar[r] & \Psi^1_X \ar[r]  & \Psi^1_{X/S} \ar[r] & 0.
    }
  \end{equation}
  Consequently one has a commutative diagram in the derived category
  \[ 
    \xymatrix{
      \cM^\g_{X/S}(-1) \ar[r]^-E\ar[d] & f^* \ocM^\g_S(-1)[1]=P^\g(-1)[1]\ar[d]\\
      \Psi_{X/S} \ar[r]_-F & f^* \Psi^1_S [1]. 
    } 
  \]
\end{lemma}

\vspace{-0.9cm} \hfill $\qed$ \vspace{0.3cm} 

We will achieve our goal by establishing the commutativity of the
following diagram: 
\[ 
  \xymatrix@R=3em@C=1.5em{
    {\bigwedge^q} \cM \ar[r]^-{E_q} \ar[d]_{\wedge^q
      \sigma} \ar@{}[dr]|{(1)} & 
 P^\g(-1)[1] \ot{\displaystyle \bigwedge^{q-1}} \cM 
\ar[r]^-{\gamma \otimes 1} \ar[d] 
\ar@{}[dr]|{(2)}  & {\displaystyle \bigwedge^{q-1}} \cM[1] \ar[d]^{\wedge^{q-1}\sigma} \\
    {\displaystyle \bigwedge^q} \Psi^1_{X/S} \ar[r]_-{F_q}  \ar[dd]_{\rm  mult.} \ar@{}[ddrr]_{(3)} \ar@{}[drr]|{(4)}  \ar@/_2.7pc/[rr]_-{G_q(\gamma)} & {\displaystyle
      f^* \Psi_S[1]\ot      \bigwedge^{q-1}} \Psi^1_{X/S} 
    \ar[r]_-{ \gamma\otimes 1} & {\displaystyle \bigwedge^{q-1}} \Psi^1_{X/S}[1] \ar[dd]^{\rm mult.} \\
    & & \\
    \Psi^q_{X/S} \ar[rr]_{(-1)^{q-1}L^q_\lambda} & &  \Psi^{q-1}_{X/S}  [1].
  }
\]

Here we have written $\cM$ as a shorthand for $\cM^\g_{X/S}(-1)$ and
$\gamma:{ f^*\Psi^1_S  \to\bz}$ for  the pullback by $f$ of
$\Psi^1_S=P^\g(-1){\tto{\gamma}}\bz$. The maps $E_q$, $F_q$, and
$G_q(\gamma)$ are defined by applying the $q$-th exterior power
construction $\xi\mapsto\xi_q$ of \S\ref{ss:koszul} to the extensions
$E$, $F$, and $G(\gamma)$, respectively. Here the extension
$G(\gamma):\Psi^1_{X/S}\to \bz[1]$ is defined {by the exact sequence} \eqref{eqn:def-G} below.  Thus the commutativity of the larger outer rectangle in this diagram is the desired formula (2) of Theorem~\ref{monthm.t}. We prove this commutativity by checking the interior cells (1) through (4). 

{\bf (1):} This square commutes by functoriality of the maps $\xi_q$ defined in \S\ref{ss:koszul} and Lemma~\ref{lemma:ses-comp}.


{\bf (2):}  It suffices to check the commutativity when $q=1$, in which case it follows from the the definition of the map $\gamma:\Psi^1_S\to \bz$

{\bf (3):} We let the $\I$ be the subgroup of $\li P$ generated by $\gamma$ and work in the category of $\I$-modules.  Applying (1) of Proposition~\ref{rgami.p}, we find a distinguished triangle:
\begin{equation} \label{eqn:group-coh-triangle}
  R\Gamma_\I(\Psi_{X/S})\rTo{a} \Psi_{X/S}\rTo{\gamma-1} \Psi_{X/S} \rTo{b} R\Gamma_\I(\Psi_{X/S})[1].
\end{equation}
Since $\gamma$ acts trivially on the $\Psi^q_{X/S}$, the long cohomology exact sequence of the above triangle yields a short exact sequence
\begin{equation} \label{rqtauseq.e}
  0\To \Psi^{q-1}_{X/S} \rTo{b^q}  R^q\Gamma_\I(\Psi_{X/S}) \rTo{a^q} \Psi^q_{X/S}  \To 0.
\end{equation}

When $q = 1$, the exact sequence~(\ref{rqtauseq.e}) reduces to 
\begin{equation} \label{eqn:def-G}
  0 \To \bz  \rTo{\beta := b^1}   R^1\Gamma_\I(\Psi_{X/S}) \rTo{\alpha := a^1} \Psi^1_{X/S} \To 0,
\end{equation}
where $\beta(1)$ is the image of the class $\theta\in R^1\Gamma_\I(\bz)$ in $R^1\Gamma_\I(\Psi_{X/S})$. Applying the exterior power construction of \S\ref{ss:koszul}, one obtains for each  $q \ge 1$ an exact sequence:
\[
  0 \To  \bigwedge^{q-1}\Psi^1_{X/S} \rTo{\beta^q}  \bigwedge^{q} R^1\Gamma_\I(\Psi_{X/S}) \rTo{\alpha^q} \bigwedge^q \Psi^1_{X/S} \To 0,
\]
where $\beta^q$ is deduced from cup product with $\theta$ on the left. We assemble the arrows $\alpha^{q-1}$ and $\beta^q$ to form the top row of the following diagram, and the arrows $a^{q-1}$ and $b^q$ to form the bottom row.
\[ 
  \xymatrix{
    \bigwedge^{q-1}
    R^1\Gamma_\I(\Psi_{X/S})\ar[r]^-{\alpha^{q-1}}\ar[d]_{\rm mult.} & 
\bigwedge^{q-1} \Psi^1_{X/S} \ar[r]^{\beta^q}\ar[d]_{\rm mult.}&  \bigwedge^{q} R^1\Gamma_\I(\Psi_{X/S}) \ar[d]_{\rm mult.} \\
    R^{q-1} \Gamma_\I(\Psi_{X/S}) \ar[r]_{a^{q-1}} & \Psi^{q-1}_{X/S}\ar[r]_{b^q} & R^q \Gamma_\I(\Psi_{X/S}). 
  }
\]

The maps $a$ and $\alpha$  are the restriction maps on group cohomology from $\I$ to the zero group, and hence commute with cup product, so that the left square commutes.  By (2) of Proposition~\ref{rgami.p}, the composition  $b^q \circ a^{q-1}$ is given by cup product on the left with the morphism $\theta \colon \bz \to \bz[1]$ defined by the fundamental extension~(\ref{groupal.e}). By the above discussion, the same is true for $\beta^q\circ \alpha^{q-1}$. Since the vertical maps are also defined by cup product, we see that the outer rectangle commutes. As the map $\alpha^{q-1}$ is surjective, we deduce that the right square also commutes.

Putting these squares alongside each other in the opposite order, we get a commutative diagram with exact rows:
\[ 
  \xymatrix{
      0 \ar[r] & \bigwedge^{q-1} \Psi^1_{X/S} \ar[r]^-{\beta^q}
      \ar[d]_{\rm mult.} &  \bigwedge^q R^1 \Gamma_\I(\Psi_{X/S})
      \ar[r]^-{\alpha^q} \ar[d]_{\rm mult.} 
   & \bigwedge^q \Psi^1_{X/S} \ar[r] \ar[d]_{\rm mult.} & 0 \\
      0 \ar[r] & \Psi^{q-1}_{X/S} \ar[r]^-{b^q} & R^q \Gamma_\I(\Psi_{X/S}) \ar[r]^-{a^q} & \Psi^q_{X/S} \ar[r] & 0.
  }  
\]
Taking the maps in the derived category corresponding to these extensions gives a~commutative square:
\[ \xymatrix{
\bigwedge^q \Psi^1_{X/S} \ar[r]\ar[d]_{\rm mult.}& \ar[d]^{\rm mult.} \bigwedge^{q-1} \Psi^1_{X/S} [1]\cr
 \Psi^q_{X/S} \ar[r]\ar[r] &  \Psi^{q-1}_{X/S} [1]
}
\]
Proposition~\ref{hom.p} applied to the triangle \eqref{eqn:group-coh-triangle} implies that the bottom arrow is $\kappa^q = (-1)^{q-1}L^q_\lambda[q]$, while the top arrow is $G_q(\gamma)$ by definition. It follows that cell (3) commutes. 

{\bf (4):} Once again we can reduce to the case  $q=1$ by functoriality of the construction of \S\ref{ss:koszul}. Consider the action of $\I$ on $\Psi_{X/S}$ via $\gamma$.  It is enough to establish the commutativity of the diagram
\begin{equation} \label{eqn:two-extensions}
  \xymatrix{
    0 \ar[r] & f^* \Psi^1_S \ar[d]_\gamma \ar[r] & \Psi^1_X \ar[d]_\phi \ar[r] & \Psi^1_{X/S} \ar[r]\ar@{=}[d] & 0 \\
    0 \ar[r] & \bz \ar[r] & R^1\Gamma_\I(\Psi_{X/S}) \ar[r] & \Psi^1_{X/S} \ar[r] & 0.
  }
\end{equation}
Here $\phi$ is the restriction map $\Psi^1_X = R^1\Gamma_{\li P} (\Psi_{X/S}) \to R^1\Gamma_{\I} (\Psi_{X/S})$ along $\gamma:\bz\to I$. Indeed, the top extension being $F:\Psi^1_{X/S}\to f^* \Psi^1_S [1]$, the bottom extension (which is the pushout of the top extension by $\gamma$) is $\gamma\circ F : \Psi^1_{X/S} \to \bz[1]$. On the other hand,  as we  saw in Proposition~\ref{hom.p},  the bottom extension corresponds to $L^1_\lambda \colon \Psi^1_{X/S}\to \bz[1]$.

The right square of \eqref{eqn:two-extensions} commutes by functoriality of restriction maps $R\Gamma_I(-)\to R\Gamma_{\bz}(-) \to R\Gamma_{0}(-) = (-)$. The left square is isomorphic to
\[ 
  \xymatrix{
    H^1(\li P, \Hh^0(\Psi_{X/S})) \ar[r] \ar[d] & H^1(\li P, \Psi_{X/S}) \ar[d] \\
    H^1(\I, \Hh^0(\Psi_{X/S})) \ar[r]  & H^1(\I, \Psi_{X/S}), \\
  }
\]
which commutes by functoriality of the maps $H^1(G, \Hh^0(-))\to H^1(G, (-))$ with respect to $G$.

\subsection{\'Etale cohomology}
\label{etale.s}

The results of Sections~\ref{lognearby.s} and \ref{monpf.s} have natural algebraic analogues, due to Fujiwara, Kato, and Nakayama \cite{cnak.lec}, obtained by replacing the space $X_\lag$ with the Kummer-\'etale topos $X_\ket$, and the (logarithmic) exponential sequence \eqref{logexp.e} with the (logarithmic) Kummer sequence(s). We refer the reader to \cite{ill.ofknlec} for a survey of the Kummer \'etale cohomology.

The algebraic version of our setup is as follows: we fix an algebraically closed field $k$ and work in the category of fine and saturated log schemes locally of finite type over $k$. We fix an integer $N>1$ invertible in $k$ and use $\Lambda=\bz/N\bz$ as a coefficient ring. We define $\Lambda(1) = \mu_N(k)$, $\Lambda(n) = \Lambda(1)^{\otimes n}$ for $n\geq 0$, $\Lambda(n) = \Lambda(-n)^\vee$ for $n\leq 0$; for a~$\Lambda$-module $M$, $M(n)$ denotes $M\otimes \Lambda(n)$.

We start by considering a single fs log scheme $X$. We denote by $\varepsilon: X_\ket \to X_\et$ the projection morphism (here $X_\et$ is the \'etale site of the underlying scheme). The sheaf of monoids $\cM_X$ on $X_\et$ extends naturally to a sheaf $\cM^\ket_X$ on $X_\ket$ associating $\Gamma(Y_\et, \cM_Y)$ to a Kummer \'etale $Y\to X$; we have a natural homomorphism $\varepsilon^* \cM_X\to \cM^\ket_X$. The \emph{logarithmic Kummer sequence} is the exact sequence of sheaves on $X_\ket$
\begin{equation} \label{logkummer.e}
  0 \To \Lambda(1) \To \cM^{\ket, \g}_X\rTo{N} \cM^{\ket, \g}_X\To 0.
\end{equation} 
Applying the projection $\varepsilon_*$ yields a homomorphism
\[ 
    \sigma_0:\cM_X^\g \To \varepsilon_* \varepsilon^* \cM_X^\g \To \varepsilon_* \cM_X^{\ket, \g} \To R^1 \varepsilon_* \Lambda(1).
\]

\begin{theorem}[{\cite[Theorem 2.4]{kkcn.lblelsc}, \cite[Theorem 5.2]{ikn.qlrhc}}]\label{rqeps.t}
  The map $\sigma_0$ factors through $\ov\cM_X^\g$, inducing an isomorphism $\sigma:\ov\cM_X^\g\otimes\Lambda(-1)\to R^1\varepsilon_* \Lambda$ and, by cup product, isomorphisms
  \[ 
    \sigma^q : \bigwedge^q \ov\cM_X^\g\otimes\Lambda(-q) \isomlong R^q\varepsilon_* \Lambda.
  \]
\end{theorem}


We now turn to the relative situation. The base $S$ is a fine and saturated split log point associated to a fine sharp monoid $P$ (that is, $S=\spec(P\to k)$). Consider the inductive system $\tilde P$ of all injective maps $\phi:P\to Q$ into a sharp fs monoid $Q$ such that the cokernel of $\phi^\g$ is torsion of order invertible in $k$, and let $\tilde S = \spec(\tilde P\to k)$. Let $\li P$ be the automorphism group of $\tilde S$ over $S$ (the logarithmic inertia group of $S$); we have a natural identification $\li P\cong \Hom(P^\g, \hat\bz'(1))$ where $\hat\bz'(1) = \varprojlim_N \mu_N(k)$ \cite[4.7(a)]{ill.ofknlec}. We can identify the topos $S_\ket$ with the classifying topos of $\li P$. 

We consider an fs log scheme $X$ locally of finite type over $k$ and a saturated morphism $f:X\to S$. We define $\tilde X = X\times_S \tilde S$ (fiber product in the category of systems of fs log schemes). We denote the projections $\varepsilon: X_\ket \to X_\et$ and $\tilde\varepsilon: \tilde X_\ket \to \tilde X_\et = X_\et$. 

\begin{lemma} \label{exactmodn.l} 
The sequence of \'etale sheaves on $X$
\[ 0\To \ocM_S^\g\ot\Lambda \To \ocM_X^\g\ot\Lambda \To \ocM_{\tilde X}^\g\ot\Lambda \To 0  \]
is exact, yielding an identification $\ocM_{\tilde X}^\g \otimes \Lambda \cong \cM_{X/S}^\g \otimes \Lambda$. 
\end{lemma}

\begin{proof}
Note first that since $X\to S$ is saturated, the square
\[ 
  \xymatrix{
    \tilde X \ar[r] \ar[d] & X \ar[d] \\
    \tilde S \ar[r] & S
  }
\]
is cartesian in the category of (systems of) log schemes, and in particular the corresponding diagram of underlying schemes is cartesian, i.e.\ $\tilde X\cong X$ as schemes. Let $\bar x$ be a geometric point of $X$. We have pushout squares
\[ 
  \xymatrix{
    P\ar[r] \ar[d] & \ocM_{X, \bar x} \ar[d]  \ar@{}[rd]|-{\textstyle \text{and}} & P^\g \ar[r]\ar[d] & \ocM^\g_{X, \bar x} \ar[d] \\ 
    \tilde P\ar[r]  & \ocM_{\tilde X, \bar x} & \tilde P^\g \ar[r] & \ocM_{\tilde X, \bar x}^\g, 
  }
\]
and therefore also a pushout square
\[ 
  \xymatrix{
    P^\g \ot\Lambda \ar[r]\ar[d] & \ocM^\g_{X, \bar x} \ot\Lambda \ar[d] \\
    \tilde P^\g\ot\Lambda \ar[r] & \ocM_{\tilde X, \bar x}^\g \ot\Lambda.
  }
\]
But $\tilde P^\g$ is $N$-divisible for all $N$ invertible in $k$, so $\tilde P^\g\ot\Lambda = 0$, yielding the desired exactness.
\end{proof}

The complex of nearby cycles is the complex $\Psi_{X/S} := R\tilde\varepsilon_* \Lambda$ of discrete $\li P$-modules on $X_\et$. Its cohomology is described by the following analog of Theorem~\ref{logsymb.t}

\begin{theorem} \label{logsymbet.t}
  There are canonical isomorphisms
  \[
    \sigma^q : \bigwedge^q \cM^\g_{X/S}\otimes\Lambda(-q)\isomlong \Psi^q_{X/S}
  \]
  for all $q$. In particular, the logarithmic inertia group $\li P$ acts trivially on $\Psi^q_{X/S}$.
\end{theorem}

\begin{proof}
This follows from Theorem~\ref{rqeps.t} for $\tilde X$, using the identifications $X_\et = \tilde X_\et$ and $\cM^\g_{X/S}\otimes\Lambda \cong \bar\cM^\g_{\tilde X}\otimes\Lambda$.
\end{proof}

As before, we denote by $\lambda^q_\gamma : \Psi^q_{X/S}\to \Psi^{q-1}_{X/S}[1]$ the map induced by $\gamma-1 : \Psi_{X/S}\to \Psi_{X/S}$. The usual Kummer sequence on $X_\et$ yields a map $\oh X^\g\to \Lambda(1)[1]$, which composed with the map $\cM_{X/S}^\g\to \oh X^*[1]$ coming from the extension \eqref{mptompb.e} yields a map $\ch_{X/S} : \cM_{X/S}^\g \to \Lambda(1)[2]$.

With these in place, we can state the \'etale analog of Theorem~\ref{monthm.t}

\begin{theorem} \label{monthmetale.t}
  Let $f\colon X\to S$ be as above. Then:
  \begin{enumerate}
    \item  For each $q \ge 0$, the following diagram commutes:
      \[ 
        \xymatrix{
          \bigwedge^q \cM_{X/S}^\g\otimes\Lambda(-q)  \ar[r]^-{{\rm ch}_{X/S}^q} \ar[d]_{q!\sigma^q} & \bigwedge^{q-1}\cM_{X/S}^\g\otimes\Lambda(1-q)[2] \ar[d]^{q!\sigma^{q-1} } \\
          \Psi^q_{X/S} \ar[r]_{\delta^q} &  \Psi^{q-1}_{X/S}[2].
        }
      \]
    \item  For each $q\ge 0$  and each $\gamma \in \li P$, the following diagram commutes:
      \[ 
        \xymatrix{
          \bigwedge^q \cM_{X/S}^\g\otimes\Lambda(-q) \ar[r]^-{\kappa^q_\gamma} \ar[d]_{\sigma^q} & \bigwedge^{q-1} \cM_{X/S}^\g\otimes\Lambda(1-q) [1] \ar[d]^{\sigma^{q-1}} \\
          \Psi^q_{X/S}  \ar[r]_-{(-1)^{q-1}\lambda^q_\gamma}  & \Psi^{q-1}_{X/S} [1] ,
        }
      \]
{where $\kappa^q_\gamma$ is as in Proposition~\ref{hom.p}.}
  \end{enumerate}
\end{theorem}

\begin{proof}
The proof of (1) relies on the following analogue of the isomorphism \eqref{steenfor.e}. The exact sequence \eqref{logkummer.e} provides a quasi-isomorphism on $\tilde X_\ket$
\[ 
  \Lambda(1) \isomlong \left[ \cM^{\ket,\g}_{\tilde X}\to \cM^{\ket,\g}_{\tilde X}   \right]
\]
and the morphism of complexes:
\[ 
  \left[ \tilde\varepsilon^* \cM^\g_{\tilde X}  \tto{N} \tilde\varepsilon^* \cM^\g_{\tilde X} \right]
  \To
  \left[ \cM^{\ket,\g}_{\tilde X}\tto{N} \cM^{\ket,\g}_{\tilde X}   \right]
\]
induces by adjunction a morphism
\[ 
  \left[ \tilde\varepsilon^* \cM^\g_{\tilde X}  \tto{N} \tilde\varepsilon^* \cM^\g_{\tilde X} \right]
  \To \Psi_{X/S}(1).
\]
This morphism induces an isomorphism $[\tilde\varepsilon^* \cM^\g_{\tilde X}  \tto{N} \tilde\varepsilon^* \cM^\g_{\tilde X}] \cong \tau_{\leq 1} \Psi_{X/S}$, our analogue of \eqref{steenfor.e}. Then assertion (1) follows exactly as before. The proof of (2) follows the lines of our second proof of the analogous assertion in \S\ref{ss:monpf}. We omit the details. 
\end{proof}
\section{Curves}
\label{curves.s}

The goal of the present section is to illustrate
Theorems~\ref{monthm.t} and \ref{idealcut.t} for curves.  We shall attempt to convince our
readers that the combinatorics arising from the log structures
are  essentially equivalent to the data usually expressed in terms
of the ``dual graph'' of a degenerate curve,
for example in~\cite[Exp. IX 12.3.7]{g.sga7I}.
In particular, we show how the classical Picard--Lefschetz formula
for curves can be derived from our monodromy formula.
In this section we work over the field $\bc$ of complex numbers.

\subsection{Log curves and their normalizations}\label{lcn.ss}

Our exposition is based on  F.~Kato's  study of 
the moduli of log curves and their relation to the classical
theories~\cite{kaf.logcurve}.  Let us recall his basic notions.
 
 \begin{definition}\label{logcurve.d}
Let $S$ be a fine saturated  and locally noetherian log scheme.  
A~\textit{log curve} over $S$  is a smooth, finite type,   and saturated  morphism $f  \colon X \to S$  of fine saturated log schemes such that every geometric fiber of $\und f \colon \und X \to \und S$
has pure dimension one. 
 \end{definition}

Kato requires that $\und X$ be connected, a condition we have dropped from our definition. 
If $\und X/\bc$ is a smooth curve and $\und Y$
is a finite set of closed points  of $ \und X$, then the compactifying
log structure  associated with the open subset
$\und X \setminus \und Y$ of $\und X$ is fine and saturated, and
the resulting log scheme is a log curve over $\bc$.  In fact,
every log curve over $\bc$ arises in this way, so that to give a 
log curve over $\bc$ is equivalent to giving 
a~smooth curve with a set (not a sequence)  of marked points.

For simplicity, we we concentrate  on the case 
of vertical log curves over the standard log point $S := \spec (\bn \to \bc)$. 
Then a morphism of fine saturated log schemes $X \to S$ 
is automatically integral~\cite[4.4]{kato.lsfi}, and if it is  smooth, it is saturated
if and only if its fibers are
reduced~\cite[{II.4.2}]{ts.smls},~{\cite[IV.4.3.6]{o.llogg}}.
Since $X/S$ is vertical,
the  sheaf   $\cM_{X/S} := \cM_X/f^*\cM_{S}$ is in fact  a sheaf of groups.
Corollary~\ref{vertons.c} says that the set $Y := \{x \in X :\cM_{X/S} \neq 0\}$ is
closed in $X$, that its complement $U$ is open and dense, and that
the underlying scheme $\und U$ of $U$ is smooth.  In fact 
Kato's analysis of log curves gives the following detailed local 
description of $X/S$.

\begin{theorem}[F. Kato] \label{fkato.t}
  Let $f \colon X \to S$ be a  vertical log curve over the standard log point $S$ and let $x$ be a closed point of $\und X$. 
  \begin{enumerate}
  \item 
The underlying scheme $\und X$ is smooth  at  $x$ if and only 
if there is an isomorphism $\ocM_{X,x} \cong \bn$.  If this is the case, 
there exist an \'etale neighborhood $V$ of $x$ and a commutative diagram
  \[
    \xymatrix{ 
      V \ar[r]^-g\ar[rd] & S  \times {\ls  \bz} \ar[d] \\
       & S,
    }
  \]
  where $g$ is strict and  \'etale.  
\item The underlying scheme $\und X$ is singular at $x$ if and only
if there  exist an integer $n$ and an 
isomorphism $\ocM_{X,x} \cong Q_n$, where 
$Q_n$  is the monoid given by  generators $q_1,q_2, q$ satisfying the  relation $q_1+q_2 = nq$.   In this case there exist
an \'etale neighborhood $V$ of $x$  and a~commutative diagram
  \[
    \xymatrix{ 
      V \ar[r]^-g\ar[rd] &\ls{Q_n,J}\ar[d]^{\ls \theta} \\
       & S,}
  \]
  where  $g$ is strict and \'etale,   where
$J$ is the ideal  of $Q_n$ generated by $q$, and 
where  $\theta \colon \bn \to Q_n$ is the homomorphism sending $1$ to $q$.  \qed
  \end{enumerate}
\end{theorem}
\begin{proof}
  
For the convenience of the reader we give an indication of the proofs, using the point of view developed in Corollary~\ref{vertons.c}.  We saw there that the set $U := \{x \in X : \ocM_{X,x} = \bn\}$ is open in $X$. Morover  $\und U$ is smooth over $\bc$,  so it can be covered by open sets $V$ each of which
admits   an \'etale map $U \to {\bf G_m} = \ls \bz$.
Since the morphism $U \to \und U \times S$ is an isomorphism,
we find  a diagram as in case (1) of the theorem.

Suppose on the other hand that $x \in Y := X\setminus U$.  Since 
the sheaf of groups $\cM_{X/S}^\g$ is torsion free,
one sees from  \cite[IV, 3.3.1]{o.llogg} that in a neighborhood $V$ of 
$x$, there exists a chart for $f$  which is neat and smooth at $x$.  That is, there 
exist a fine saturated monoid $Q$, an injective homomorphism $\theta \colon \bn \to Q$,
and a map $V \to \ls Q$ such that induced map $V \to S \times_{\ls \bn} \ls Q$
is smooth and such that the homomorphism $Q^\g/\bz \to \cM_{X/S,x}$
is an isomorphism.   By \cite[III, 2.4.5]{o.llogg}, the chart $Q \to \cM_{X}$
is also neat at $x$.  Let $J$ be the ideal of $Q$ generated by $q:= \theta(1)$. 
Then $S\times_{\ls \bn} \ls Q = \ls {(Q,J)}$.  Since $\theta$ is vertical,
$J$ is the interior ideal of $Q$, and the set of minimal primes containing
it is the set of height one primes of $Q$.  Thus  the dimension of  $\ms {(Q,J)}$ 
is the dimension of $\ms F$, where $F$ is any facet of $Q$. 
This dimension is the rank of $F^\g$; if it is zero, then $Q^\g$ has rank at most one,
hence $\bn \to Q$ is an isomorphism, contradicting our assumption that $x \in Y$. 
Thus $ Q^\g$ has rank at least two. 
Since $\und V$ has dimension one and is smooth over $\ms {(Q,J)}$, it follows that
in fact $F^\g$ has rank one, that $Q^\g$  has rank two, 
 and that $\und V$ is \'etale over $\ms {(Q, J)}$.  
Since $Q^\g/\bz \cong \cM_{X/S}^\g$  and $Q^*$ maps to zero in $\cM_{X/S}^\g$
it follows that $Q$ is sharp
of dimension two, and hence has exactly two faces $F_1$ and $F_2$. 
Each of these is canonically  isomorphic to $\bn$; let $q_i$ be 
the unique generator of $F_i$.  Since $\theta$ is saturated,
every element of $Q$ can be written uniquely in the form
$nq + m_i q_i$ where $n, m_i \in \bn$ and $i \in \{1,2\}$~\cite[I, 4.8.14]{o.llogg}.
Writing  $q_1 +q_2$ in this form,  we see that necessarily $m_i = 0$
(otherwise $q$ would belong to a proper face).  Thus $Q$
is generated by $q_1$ and $q_2$,  subject to the relation
$nq = q_1 + q_2$ for some $n > 0$.  Since $Q$ has dimension two,
it is necessarily isomorphic to $Q_n$, and since the chart $Q \to \cM_{X}$
is neat at $x$, in fact $Q_n \cong \ocM_{X,x}$. 
Finally, we note that $\und X$ is singular at $x$,
so $U$ is precisely the smooth locus of $\und X$. 
\end{proof}
Let us remark that the isomorphism $\ocM_{X,x} \cong \bn$
  in (1) is unique, since the monoid
$\bn$ has no nontrivial automorphisms.  In case  (2),
the integer $n$ is unique, and   
 there are exactly two isomorphisms $\ocM_{X,x} \cong Q_n$, since $Q_n$ has a unique
nontrivial automorphism, which exchanges $q_1$ and $q_2$.  

Thanks to Kato's result,
 we can give the following more  explicit
version of Corollary~\ref{vertons.c} in the cases of log curves.
{Since we are working over the standard (split) log point $S$,
  we have a map $\bn \to \cM_S \to \cM_X$, and we let
  $\cM_{X/\bn} := \cM_X/\bn$.}

\begin{proposition}\label{mxptomxp.p}
Let $X/S$ be a vertical log curve over the standard log  point, let $\ep \colon \und
X' \to\und X$ be its normalization, and let $X'/\bc$  (resp. $X''$) be
the log curve obtained by endowing $\und X'$ with the compactifying log structure
associated to the open embedding $U' \to U$ (resp. with the induced
log structure from $X$.)
  \begin{enumerate}
\item There is a unique morphism of log schemes
$h  \colon X'' \to X'$ which is the identity on the  underlying schemes. 
\item  The maps
  $\ocM^\g_{X'} \to \ep^{-1}(\cM_{X/S})$ and
  $\cM_{X'}^{\g} \to \ep^*(\cM_{X/\bn})$  
induced by $h$  are isomorphisms,
  where $\ep^*(\cM_{X/\bn}):= \oh {X'}^* \times_{\ep^{-1}(\oh X^*)}
  \ep^{-1}(\cM_{X/\bn})$.  
\item  Let $X''' := X'\times_{\und S} S$,   and let $g \colon X'' \to
  X'''$ be the map induced by $f \circ \ep$ and $h$.  
Then the morphism $g $
identifies $X''$ with a strict log transform of $X'''$, 
 \ie, the closure of $U'$ in the 
log blowup of $X'''$ along  a coherent sheaf of ideals  of $\cM_{X'''}$,
(made explicit below). 
  \end{enumerate}
\end{proposition}
\begin{proof}
Statement (3) of Corollary~\ref{vertons.c} implies statement (1) of
Proposition~\ref{mxptomxp.p},  statement (4) implies that
$h$ induces an isomorphism
$ \theta \colon \ocM_{X'}^\g \to \ep^{-1}(\cM_{X/S})$, and
it follows that $\cM_{X'}^\g \to \ep^*(\cM_{X/P})$ 
is an isomorphism, since this map is morphism of
$\oh {X'}^*$ torsors  over $\theta$.  This proves statements (1) and (2);
we should remark that they are quite simple to prove directly
in the case of curves, because the normalization
$\und X'$ of $\und X$ is smooth.

Our proof of (3) will include an explicit description
of a sheaf of ideals defining the blowup.  For each point $y'$ of $Y'$,
let $n$ be the integer such that $\ocM_{X,\ep(y')} \cong Q_n$, let  
$\cK_{y'}$ be the ideal of $\cM_{X''',y'}$ generated by
$\cM^+_{X',y'}$ and $n\cM^+_S$, and let 
$\cK := \cap\{  \cK_{y'}:  y' \in Y'\}$,
a coherent sheaf of ideals in $\cM_{X'''}$.
Observe   that the ideal of $\cM_{X''}$ generated by $\cK$
is invertible.  This is clear at points $x' $ of $ U'$.  If $y' \in Y'$, the ideal
$\cK_{y'}$ is generated by the images of $q_2$ and $n q$, and
in  $Q_n$  the ideal $(q_2, nq)$ is generated by $q_2$, since $nq = q_1 + q_2$.    
Thus the map $X''  \to X'''$ factors through the log blowup~\cite[4.2]{niz.tslbgr}.
A chart for $X'''$ near $y'$ is given by $\bn \oplus \bn$ mapping
$(1,0) $ to $q_2$ and $(0,1) $ to $q$, and the log blowup of the ideal
$(q_2, nq)$ has a standard affine cover consisting of two open sets.
The first is obtained by adjoining $nq-q_2$, and the corresponding
monoid is $Q_n$; and the second by adjoining $q_2-nq$. 
The closure of $U'$ is contained in the first affine piece, so we
can ignore the second. Thus
the induced map
is indeed an isomorphism as described.
\end{proof}

Proposition~\ref{mxptomxp.p}  shows that one can recover the log curve $X''/S$ 
directly from the log curve $X'$ together with the function $\nu \colon Y'  \to \bz^+$
taking a  point $y'$ to the number $n$ such that $\ocM_{\ep(y')} \cong Q_n$. 
In fact there are additional data at our disposal, 
as the following proposition shows. 

\begin{proposition}\label{exdata.p}
  Let $X/S$ be a vertical log curve over the standard log point and let
  $X'/\bc$ be the corresponding log curve over $\bc$ as described in
  Proposition~\ref{mxptomxp.p}. Then $X'/\bc$ is naturally equipped  with the following additional data.
\begin{enumerate}
\item A fixed point free involution $\iota$ of $\und Y'$.
\item A mapping  $ \nu \colon \und Y' \to \bn$ such that
$\nu(y') = \nu(\iota(y'))$ for every $y' \in Y'$.
\item A trivialization of the invertible sheaf
$N_{\und Y'/\und X'} \ot \iota^*(N_{\und Y'/\und X'})$ on $\und  Y'$. 
\end{enumerate}
\end{proposition}
\begin{proof}
  These data arise as follows.
Each  fiber of the map $\und \ep \colon  \und Y' \to \und Y$ 
has cardinality two, and hence there is a unique
involution $\iota$ of $\und Y'$ which interchanges
the points in each fiber.  
The function $\nu$ is defined as above: 
$\nu(y')$ is the integer $n$
such that $\ocM_{X,\ep(y')}  \cong Q_n$. 
To obtain the trivialization in (3), let $y'$ be a point of $Y'$ and 
 let $y:= \ep(y')$.
Recall from Remark~\ref{xlogbund.r} that  if $X$
is a fine  log space and $\ov m \in \Gamma(X, \ocM_X)$,
there is an associated invertible sheaf $\cLL_{\ov m}$ whose
local generators are the sections of $\cM_X$ mapping to $\ov m$. 
Observe that,  since the log point $S$ is equipped with
a splitting $\ocM_S \to \cM_S$, there is a 
a canonical generator $m_S$  of the invertible sheaf
$\cLL_{1,S}$ on $S$.
 Let us use the notation of the proof of Proposition~\ref{mxptomxp.p}.
Endow $Y$ with the log structure from $X$ and 
choose a point $y$ of $Y$.  The choice of a chart at $y$
defines sections $m_1$ and $m_2$ of $\cM_{Y,y}$, whose
images $\ov m_1$ and $\ov m_2$ in $\Gamma(y, \ocM_{Y,y})$  are independent
of the choice of the chart and define
one-dimensional vector spaces $\cLL_{\ov m_i}$.
The equality $\ov m_1 + \ov m_2 = n\ov f^\flat(1)$ induces
an isomorphism $\cLL_{\ov  m_1} \ot \cLL_{\ov m_2} \cong
\cLL_{1,s}^n \cong \bc$.  As we have seen, the element $m_2$
corresponds to a generator of the ideal of the point $y'_1$ in $\und X'_1$, 
so   there is a canonical isomorphism $\ep^*(\cLL_{\ov m_2}) \cong N_{\und Y'/\und   X',y'_1}^{-1} $;  similarly   $\ep^*(\cLL_{\ov m_1}) \cong N_{\und Y'/\und X',y_2}^{-1}$.
 Thus we find isomorphisms:
\[
N_{\und Y'/\und X', y'_1} \ot (\iota^*N_{\und Y'/\und X'})_{y'_1} \cong
N_{\und Y'/\und X', y'_1} \ot N_{\und Y'/\und X,y'_2} \cong
\ep^*(\cLL_{\ov m_2})^{-1} \ot \ep^*(\cLL_{\ov m_1})^{-1} \cong \bc. \qedhere 
\]
\end{proof}

In fact the data in Proposition~\ref{exdata.p} are enough to reconstruct
the original log curve $X/S$.  (For an analogous result 
in the context of semistable reduction, see \cite[III, Prop. 1.8.8]{o.llogg}.)
Rather that write out the proof, 
let us explain how one can construct the fibration
$X_\lag \to \sone$ directly from $X'$ together with these data.

 It will be notationally convenient
for us to extend $\iota$ to a set-theoretic involution
on all of $\und X'$, acting as the identity on $\und U'$.
If $y' \in Y'$ and $v $ is a nonzero element of $N_{y'/\und X'}$,
let $\iota(v)$ be the element of $N_{\iota(y')/\und X'}$ which is
dual to $v$ with respect to the  pairing defined by (3) above.  Then
$\iota(\lambda v) = \lambda^{-1}\iota(v)$ for all $v$. 
Note that since $\ocM_{X',y'} = \bn$ for every $y' \in Y'$
and vanishes otherwise, we have a natural  set-theoretic action of
$\sone$ on $X'_\lag$ covering the identity of $X'$.    
Thus the following sets of data are equivalent:
\begin{enumerate}
\item a trivialization of $N_{\und Y'/\und X'} \ot \iota^*(N_{\und Y'/\und X'})$;
\item  an involution $\iota$ of $N_{\und Y'/\und X'}$, covering
the involution $\iota$ of $Y'$, 
such that $\iota (\lambda v) = \lambda^{-1}\iota (v)$ for $\lambda \in \bc^*$
and $v \in N_{Y'/X}$;
\item an involution  of $\iota$ of $\sone(N_{\und Y'/\und X'})$ (the circle
bundle of $N_{\und Y'/\und X'}$), covering the involution $\iota$  of $Y'$,
such that $\iota(\lambda(v)) = \lambda^{-1}(\iota(v))$
for   $\lambda \in \sone$ and  $v \in \sone(N_{\und Y'/\und X'})$;
\item an involution  $\iota$ of   $X'_\lag$  such that 
$\tau_{X'}(\iota(x'_\lag)) = \iota(\tau_{X'}(x'_\lag))$
and $ \iota(\zeta x'_\lag) = \zeta^{-1}\iota (x'_\lag)$ for 
 $\zeta \in \sone$ and  $x'_\lag \in X'_\lag$.
\end{enumerate}
The data in  (3) and (4) are equivalent thanks to Remark~\ref{xlogbund.r}.
We should also point out that
these data are unique up to (non-unique) isomorphism.

\begin{proposition}\label{gluexlag.p}
  Let $X/S$ be a log curve and let $X'$  and $\iota$ be as
  above.  Let $\nu(x'_\lag) := \nu(\ep(\tau_{X'}(x'_\lag)))$ and
define $\iota$ on $X'_\lag \times \sone$ by
\[\iota(x'_\lag,\zeta) := (\zeta^{\nu(x'_\lag)}\iota(x'_\lag),\zeta).\]
Then $X_\lag$ is the quotient of $X'_\lag \times \sone$
by the equivalence relation generated by $\iota$. 
\end{proposition}

\begin{proof}
Let $y$ be a point of $Y$ and let $\ep^{-1}(y) := \{ y'_1, y'_2\}$.  
We can check the formula with the aid of charts, using again the notation
of the proof of  Proposition~\ref{mxptomxp.p}. 
Then 
$\ocM^\g_{X'', y'_1} \cong \ocM_{X',y'_1}^\g \oplus \ocM_S^\g$ is free with basis $\ov m_2, \ov m$
and $\ocM_{X'', y'_2} \cong \ocM_{X',y'_2}^\g \oplus \ocM_S^\g$ is free with basis $\ov m_1, \ov m$.
We have isomorphisms
\[\xymatrix{
\ocM^\g_{X'',y''_1} &\ar[l]  \ocM^\g_{X,y} \ar[r] &\ocM^\g_{X'',y''_2}}\]
sending $\ov m $ to $\ov m $ and $\ov m_2$ to $\ov m_1 = n\ov m  -\ov m_2$. 
  Then the formula follows immediately.
\end{proof}

This gluing map is compatible with the map $X_\lag \to  \sone$. 
Figure~\ref{glue.f} illustrates the restriction of this gluing to
$\tau^{-1}(y) \to \sone$, when pulled back to
via the exponential map  $[0, 2\pi i] \to \sone$. 
\begin{figure}[h!]
  \centering
  \includegraphics[scale=.25]{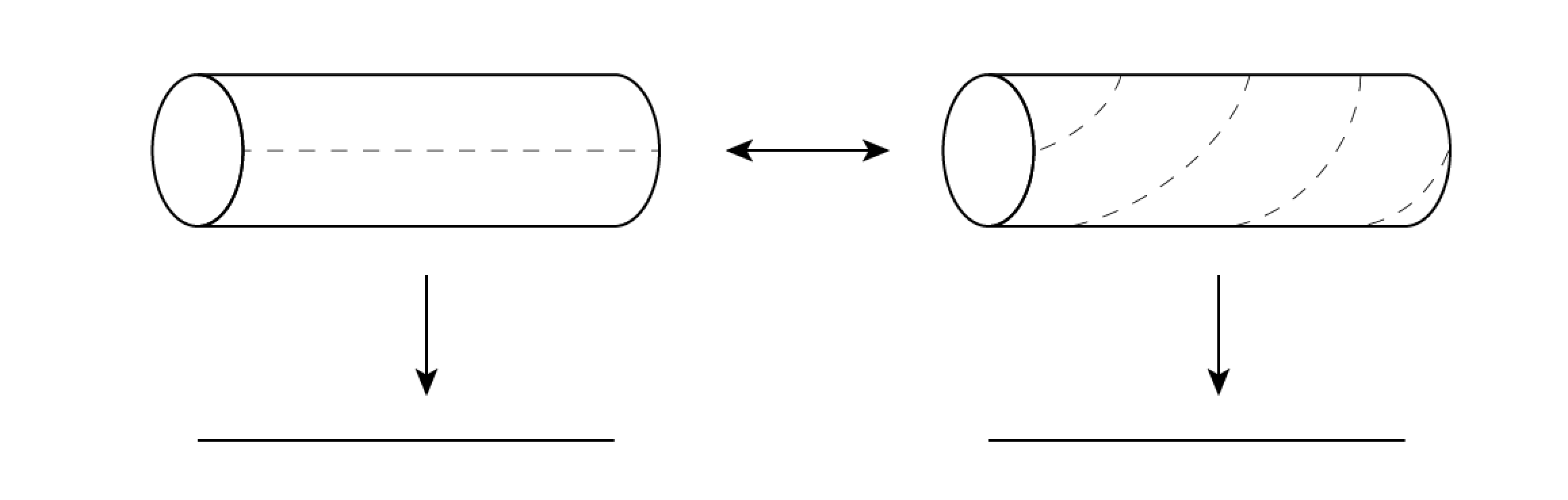}
  \caption{Gluing log fibers}
  \label{glue.f}
\end{figure}
The reader may recognize the gluing   map   as a Dehn twist.
It appears here as  gluing data, not monodromy.
It is of course also possible to see the monodromy from this
point of view as well, using a chart for $X$ at a point $y$ of $Y$.
Since this description is well-known but not functorial,
we shall not develop it here.  

\subsection{Log combinatorics and the dual graph}

Proposition~\ref{mxptomxp.p} and the data of Proposition~\ref{exdata.p}
will enable us to give a combinatorial description of 
the sheaf $\cM_{X/S}$ on $X$.  In fact there are
two ways to do this, each playing its own role and each
related to the ``dual graph'' associated to the underlying
nodal curve of $X/S$.  

We begin with the following elementary construction.
Let $B$ be a finite set with an  involution $\iota$ and let 
$\ep \colon B \to E$ be its  orbit space. 
There are two
natural exact sequences of $\bz[\iota]$-modules:
\begin{eqnarray}\label{eb.e}
0 \rTo{} \bz_{E/B} \rTo{i} \bz^B \rTo{s} \bz^E \rTo{} 0   \cr
0 \rTo{} \bz^E \rTo{j} \bz^B \rTo{p} \bz_{B/E} \rTo{} 0  
\end{eqnarray}
  
The map $s$ in the first sequence sends a basis vector $b$
of $\bz^B$ to the basis vector $ \ep(b) $  of $\bz^E$, 
and $i$ is the  kernel of $s$.  The map $j$ in the second sequence
sends a basis vector $e $ to $\sum \{ b \in \ep^{-1}(e) \}$,
and $p$ is the cokernel of $j$.  
These two sequences  are naturally dual to each other, and in particular
$\bz_{E/B}$ and $\bz_{B/E}$ are naturally dual. 
For each $b \in B$, let  $d_b :=  b - \iota(b) \in \bz_{E/B}$
and $p_b := p(b) \in \bz_{B/E}$.  Then $\pm d_b$ (resp. $\pm p_b$)
depends only on $\ep(b)$.
There is a well-defined  isomorphism of $\bz[\iota]$-modules defined by
\begin{equation}\label{betoeb.e}
t\colon \bz_{B/E} \to \bz_{E/B} : p_b := p(b) \mapsto  d_b := b- \iota(b).
  \end{equation}
 The resulting duality 
$\bz_{B/E} \times \bz_{B/E} \to \bz$ is  positive definite,
and the set of classes of elements $\{ p(b) : b \in B\}$
forms an orthonormal basis.

We apply these constructions to the involution $\iota $ of $\und Y'$
and regard $\ep \colon  \und Y' \to \und Y$ as the orbit space of this action.
The construction of $\bz_{Y/Y'}$ and $\bz_{Y'/Y}$ is compatible
with localization on $Y$ and hence these form sheaves of groups on $Y$. 
Since we are assuming that $X/S$ is vertical, $\iota$ is fixed point
free.  As we shall see, there are natural identifications of the  sheaf $\cM_{X/S}$ 
both with $\bz_{Y/Y'}$ and with $\bz_{Y'/Y}$.  We begin with the
former.

Because $\alpha_{X'}$ is the compactifying log structure associated
to   the set of marked points $Y'$, there are natural isomorphisms
of sheaves of monoids on $\und X'$:
\begin{equation}\label{cmy.e}
\ocM_{X'} \cong \und \Gamma_{\und Y'}(\Div^+_{\und X'})  \cong
\bn_{Y'}.
  \end{equation}
 Combining this identification with the isomorphism $\ep^{-1}(\cM_{X/S}) \cong \ocM_{X'}^\g$
of  statement (2) of Proposition~\ref{mxptomxp.p}, we find an isomorphism
$\ep^{-1} (\cM_{X/S}) \cong \bz_{\und Y'}$, and hence an
injection
$$\psi \colon \cM_{X/S} \to \ep_*\ep^{-1}(\cM_{X/S}) \cong \ep_*(\bz_{\und   Y'}).$$ 

\begin{proposition}\label{cmup.p}
  The inclusion $\psi$  defined above fits into an exact sequence:
\[    0 \rTo{} \cM_{X/S} \rTo{\psi} \ep_* \bz_{\und Y'} \rTo{s} \bz_{\und Y}
  \rTo{} 0,\]
and hence induces an isomorphism
\[ \psi_{X/S} \colon  \cM_{X/S} \isomlong \bz_{Y/Y'}.\]
 If $\ell \in
\Gamma(X, \cM_{X/S})$ and $ \cLL_{\ell}$ is the corresponding
invertible sheaf on $X$ coming from the exact sequence~(\ref{xsmp.e}), 
 then $\ep^*(\cLL_{\ell}) \cong \oh{\und   X'}(-\psi(\ell))$.
\end{proposition}
\begin{proof}
  Since the maps $\psi $ and $s$ are already defined globally, it is
  enough to check that the sequence is exact at each point $y$ of $Y$.  
We work in  charted neighborhood of  a point $y \in Y$ as in the proof of 
Proposition~\ref{mxptomxp.p},
 using the notation there.   Then $\cM_{X/S,y}$  is the free
abelian group  generated by the image $\ell_2$ of $q_2$, and $\ell_1
= - \ell_2$.  The pullback $t'_2$ of $t_2$ to $\und X'$ is a local
coordinate near $y'_1$ and defines a generator for $\ocM_{X',y'_1}$
mapping   to $1_{y'_1} \in \bz_{Y'}$.
 The analogous formulas
hold near $y'_2$, and hence $\psi(\ell_2) = 1_{y'_1} - 1_{y'_2}$.  
This implies that $s \circ \psi = 0$ and that the sequence is exact.
Furthermore, it follows from Proposition~\ref{mxptomxp.p} that 
$\ep^*(\cLL_{\ell}) = \cLL_{\ov m'}$, where $\ov m' \in \Gamma(X',\ocM_{X'}^\g)$
corresponds to $\ell \in \Gamma(X,\cM_{X/S})$ via the isomorphism
$\ocM_{X'}^\g \to \cM_{X/S}^\g$ in statement (2) of that proposition. 
The sheaf $\cLL_{\ov m'}$ is the ideal sheaf of the divisor $D$
corresponding to $\ov m'$, \ie,  $\oh {X'}(-D)$, and $D = \psi(\ell)$.
\end{proof}

The relationship between $\cM_{X/S}$ and $\bz_{Y'/Y}$ is more subtle
and involves the integers $\nu(y)$.  First
consider the commutative diagram of exact sequences:

\begin{equation} \label{eq:q-vs-qprime}
  \xymatrix{
    0 \ar[r] & \bz_X \ar[r]\ar[d]^{\rm res} & \ep_* (\bz_{X'})
    \ar[d]^{\rm res}\ar[r] & \bz_{X'/X}\ar[d]^{\rm res} \ar[r] & 0 \\
    0 \ar[r] & \bz_Y \ar[r] & \ep_* (\bz_{Y'}) \ar[r] & \bz_{Y'/Y} \ar[r] & 0,
  }
\end{equation}
where $\bz_{X'/X}$ is  by definition the cokernel in the top row.
As $\ep$ is an isomorphism over  $U$,
the right vertical map 
is an isomorphism and we will allow ourselves to identify
its source and target without further comment. 
Note that since $\und X'/\bc$ is smooth, the natural
map $\bz_{X'} \to j'_*(\bz_{U})$ is an isomorphism, 
hence $\ep_*(\bz_{X'}) \cong j_*(\bz_U)$, 
and the top row of the 
above diagram can be viewed as an  exact sequence:
\begin{equation}\label{bzux.e}
  0 \to \bz_X \to j_*(\bz_{U}) \to \bz_{X'/X} \to 0 .
\end{equation}
Since $X/S$ is vertical, it follows from Theorem~\ref{fkato.t} that  
$Y$ is precisely the support of  $\cM_{X/S}$  and that the map $\bz \to \ov \cM_X^\g$ is an isomorphism on $U$.  Thus there is a natural map:
 \begin{equation}
   \label{phix.e}
  \phi_X \colon \ocM_X^\g \to j_*j^*(\ov \cM_X^\g) \cong j_*(\bz_U)
  \cong  \ep_*(\bz_{X'}).   
 \end{equation}
In fact, the map $\phi_X$ is the adjoint of the homomorphism
${\ov \rho}^\g \colon \ocM_{X''}^\g \cong \ep^{-1}(\ocM_X^\g) \to \bz_{X'}$ deduced
from the homomorphism $\rho \colon \cM_{X''} \to P_{X'} = \bn_{X'}$ defined
in (4) of Corollary~\ref{vertons.c}.

 \begin{proposition}\label{norm.p}
Let $X/S$ be a vertical log curve. 
The homomorphisms $\psi_{X/S}$ of Proposition~\ref{cmup.p}
and $\phi_X$   defined above fit into a commutative diagram with exact rows:
\begin{equation} \label{mtoz.e}
  \xymatrix{ 
   & 0 \ar[r] & f^{-1}\ocM_S^\g  \ar[r]\ar[ld]_\id& \ocM_X^\g \ar[r]^\pi \ar[ld]_{\id} & \cM_{X/S} \ar[r] \ar[ld]_{\psi_{X/S}}\ar@{.>}[ldd]^-{\phi_{X/S}}& 0 \\
    0 \ar[r] & f^{-1}\ocM_S^\g  \ar[r]\ar[d]_\cong& \ocM_X^\g \ar@{.>}[r] \ar[d]^{\phi_X}& \bz_{Y/Y'} \ar[r] \ar[d]_{c_{X/S}}& 0 \\
    0 \ar[r] & \bz_X  \ar[r]&  \ep_* (\bz_{X'}) \ar[r]^p & \bz_{Y'/Y} \ar[r] & 0,\\
  }\end{equation} 
where $c_{X/S}$ is the map sending $d_{y'}$ to $-\nu(y') p_{y'}$ for every $y' \in Y'$.  
 \end{proposition}
\begin{proof}
  We compute the stalk of  the map $\phi_X$ at a point $x$ of $X$. If $x$ belongs to $U$,  the  maps $\bz  \to \ocM_{X,x}^\g$ and
  $\ocM^\g_{X,x} \to \ep_*(\bz_X')_x$ are isomorphisms, and hence so is $\phi_X$.  If $x$  belongs  to $Y$,  we call it $y$ and
work in a neighborhood  as in the proof of Proposition~\ref{mxptomxp.p}.
Then $\und X$ is the analytic space associated to $\spec\left (\bc[x_1, x_2]/(x_1x_2)\right)$, endowed with the log structure associated to the homomorphism $\beta \colon Q_n \to \bc[x_1,x_2]/(x_1x_2)$ sending  $q_i$ to $x_i$ and $q$ to $0$. 
   The point $y:=x$ is defined by $x_1=x_2= 0$, and has a basis of neighborhoods $W$ defined by $|x_i| < \ep$.  
On  the connected component $W_1 \cap U$ of $W \cap U$,  the coordinate $x_1 $ vanishes and $x_2$ becomes a unit.
 Let $\ov m_i$ (resp. $\ov m$) be the image of $q_i$ (resp. of $q$) in $\ocM_X$.    
The stalk of $\ep_*(\bz_{X'})\cong j_* (U)$ at $x$  is free with basis
$(b_1, b_2)$, where $b_i$ is the germ of the characteristic function of $W_i \cap U$ at $x$.
The isomorphism $res \colon \ep_*(\bz_{X'})_y \isomlong  \ep_*(\bz_{Y'})_y$ takes $b_i$ to the basis
element $y'_i$.
The  restriction of the sheaf $\ocM_X$  to $W_1\cap U$ is constant and freely generated by 
$\ov m_{|_{W_1\cap U}}$, while $\ov m_{1|_{W_1\cap U}} = \nu(y) \ov m_{|_{W_1 \cap U}}$ and $\ov m_{2|_{W_1\cap U}} = 0$.
  Thus  $\phi_X(\ov m_i) = \nu(y)  b_i$ and $\phi_X(\ov m) = b_1 + b_2$. 
In particular, $p(\phi(\ov m_i)) = \nu(y) p(y'_i)$.  
On the other hand, we saw in the proof of Proposition~\ref{cmup.p} that
$\psi_{X/S} (\pi(\ov m_1)) = y'_2 - y'_1 = -d_{y'_1} \in \bz_{Y/Y'}$.  Thus
\[c_{X/S} (\psi_{X/S}(\pi(\ov m_1)))  = c_{X/S}(-d_{y'_1} )= \nu(x) p_{y'_1} = p (\phi_X(\ov m_1)). \qedhere \]
\end{proof}

Since $\psi_{X/S}$ is an isomorphism, the middle row of the diagram
(\ref{mtoz.e}) above contains the same information
as the top row, a.k.a. the log Kodaira-Spencer sequence.  Furthermore, the bottom row identifies
with the  exact sequence (\ref{bzux.e}). 
The following corollary relates the 
corresponding derived morphisms of these sequences.

\begin{corollary}\label{kapkod.c}
  Let $\kappa_{X/S}  \colon \cM_{X/S}^\g \to \bz[1]$ be morphism associated to the log Kodaira--Spencer
  sequence~(\ref{logks.e}) and
   let $\kappa_{A/S} \colon \bz_{X'/X} \to \bz[1]$ be the morphism associated to the exact sequence~(\ref{bzux.e}). Then 
$\kappa_{X/S}=   \kappa_{A/S} \circ  c_{X/S}\circ \psi_{X/S} $
\end{corollary}
\begin{proof}
The diagram   (\ref{mtoz.e}) of exact sequences yields a diagram of distinguished triangles:
\[\xymatrix{ 
   &  f^{-1}\ocM_S^\g  \ar[r]\ar[ld]_\id& \ocM_X^\g \ar[r]^\pi \ar[ld]_{\id}& \cM_{X/S} \ar[r]^{\kappa_{X/S}} \ar[ld]_{\psi_{X/S}}\ar@{.>}[ldd]& f^{-1}\ocM_S^\g[1]\ar[ld]\ar@{.>}[ldd] \\
 f^{-1}\ocM_S^\g  \ar[r]\ar[d]_\cong& \ocM_X^\g \ar@{.>}[r] \ar[d]&
 \bz_{Y/Y'} \ar[r] \ar[d]_{c_{X/S}} & f^{-1}\ocM_S^\g[1] \ar[d] \\
 \bz_X  \ar[r]&  \ep_* (\bz_{X'}) \ar[r]^p & \bz_{Y'/Y} \ar[r]^{\kappa_{A/S}} & \bz_X[1]\\
  }\]
The arrows on the right are all identifications, and the formula in
the corollary follows.
\end{proof}

\begin{remark} \label{signissue.r}
The sheaf $\bz_{X'/X}$ can be naturally identified with $\Hh^1_Y(\bz)$. 
In fact there are two such natural identifications differing by sign. The first identification is the boundary map $\delta:\cQ = \Hh^0_Y(\bz_{X'/X}) \to \Hh^1_Y(\bz)$ in the long exact sequence obtained by applying the 
cohomological $\delta$-functor
 $\Hh^*_Y(-)$ to the short exact sequence~\ref{bzux.e}.
It is an isomorphism because $\Hh_Y^i(j_*\bz_U)=0$ for $i=0,1$. To define the second, recall that, by the construction of local cohomology, there is a canonical 
exact sequence:
\[
 0 \to H^0_Y(X,\bz) \to H^0(X,\bz) \to H^0(U,\bz) \to H^1_Y(X,\bz) \to H^1(X,\bz) \to \cdots,
\]
compatible with restriction to open subsets $V \subseteq X$.  In our situation, $H^0_Y(V,\bz) = 0$ for all $V$ and
$H^1(V) = 0$ for a neighborhood basis of every point of $X$. Replacing $X$ by $V$ and $U$ by $V\cap U$ for varying open
$V$ and sheafifying yields a map $j_* (\bz_U )\to \Hh^1_Y(\bz)$ which factors through an isomorphism $\delta':\bz_{X'/X}\isommap
\Hh^1_Y(\bz)$.  It follows from \cite[Cycle 1.1.5]{sga4.5}) that $\delta = -\delta'$. 
\end{remark}

 We shall see that there is a very natural connection
between the  log structures associated to a log curve
over the standard log point and the 
 the  ``dual graph'' of the underlying marked nodal curve. 
The precise meaning of this graph seems to vary from author to author;
the original and most precise definition we have found is   due to
Grothendieck~\cite[Exp. IX 12.3.7]{g.sga7I}. 
 We use the following variant,
corresponding to what some authors call an ``unoriented multi-graph.''

\begin{definition} \label{graph.d}
  A \textit{graph}  $\Gamma$ consists of two mappings between finite sets: $ \ep  \colon B \to E$ and $\zeta \colon B \to V$, where for each $e  \in E$, the cardinality of $\ep^{-1}(e)$ is either one  or two. A morphism
of graphs $\Gamma_1 \to \Gamma_2$ consists of morphisms $f_B \colon B_1 \to B_2$, $f_E \colon E_1 \to E_2$
and $f_V \colon V_1 \to V_2$ compatible with $\ep_i$ and  $ \zeta_i$ in the evident sense. 
\end{definition}

The set $V$ is  the set of ``vertices''  of $\Gamma$, the set $E$ is
the set of ``edges'' of $\Gamma$, and the set $B$ is the set of
``endpoints'' of the edges of $\Gamma$.  For each edge $e$, the set
$\ep^{-1}(e) $ is the set of endpoints of the edge $e$,
and for each $b  \in B$, $\zeta(b) $ is the vertex of $\Gamma$
corresponding to the endpoint $b$.   There is a natural involution $b
\mapsto \iota (b)$ of $B$, defined so that for each $b \in B$,
$\ep^{-1}(\ep(b)) = \{b, \iota(b)\}$.  
 The notion of a graph could equivalently be defined as a~map $\zeta
 \colon B \to V$ together with an involution of $B$;
 the  map $\ep \colon B \to E$ is then just
the projection to the orbit space of the involution. 

\begin{definition}\label{dualcurve.d}
 Let $\und X$ be a nodal  curve.  The 
 \textit{dual graph} $\Gamma(\und X)$ of $\und X$ consists
of the following data:
\begin{enumerate}
\item $V$ is the set of irreducible components of $\und X$, or equivalently, the set of connected
components of the normalization $\und X'$ of $\und X$. 
\item $E$ is  the set $Y$ of nodes of $\und X$. 
\item $B := \ep^{-1}(E)$, the inverse image of $E$ in the
  normalization $\und X'$ of $\und X$.
\item $\zeta \colon B \to V$ is the map taking a point $x'$ in $\und X'$ to  the connected component
of $\und X'$ containing it. 
\end{enumerate}
\end{definition}

The involution of the graph of a nodal curve is fixed point free,
since each $\ep^{-1}(y)$ has exactly two elements. 
A morphism of nodal curves $f \colon \und X_1 \to \und X_2$
induces a morphism of graphs provided that $f$ takes
each node of $\und X_1$ to a node of  $\und X_2$. 

\begin{definition}\label{grcomp.d}
  Let $\Gamma$ be a graph in the sense of Definition~\ref{graph.d}.
  Suppose that $\iota$ is fixed-point free, so that each $\ep^{-1}(e)$ has   cardinality  two.
  \begin{enumerate}
    \item $C_\cx(\Gamma)$ is the chain complex $C_1(\Gamma) \to C_0(\Gamma)$: 
\[ \bz_{E/B} \rTo{d_1} \bz^V, \]
    where $d_1$ is the composition $\bz_{E/B}  \rTo{i} \bz^B \rTo{\zeta_*} \bz^V$,
where $i $ is as shown in (\ref{eb.e}), and where
$\zeta_* $ sends $b$ to $\zeta(b)$. 
       \item   
$C^\cx(\Gamma)$ is the cochain complex $C^0(\Gamma) \to C^1(\Gamma)$:
      \[ \bz^V \rTo{d_0} \bz_{B/E}, \]
where $d_0$ is the composition $\bz^V \rTo{\zeta^*} \bz^B \rTo{p} \bz_{B/E}$,
where $p  $ is as shown in (\ref{eb.e}), and  where $\zeta^*(v) = \sum \{ b : \zeta(b) = v \}$. 
\item 
  \begin{trivlist}
  \item 
$\angles \ \  \colon C_i(\Gamma) \times C^i(\Gamma) \to \bz$ is the  (perfect) pairing
induced by the evident bases for $\bz^B$ and $\bz^V$, 
\item $\b \ \  \colon C^1 (\Gamma) \times C^1(\Gamma)  \to \bz$ is the (perfect)
pairing defined by $\angles \ \ $ and the isomorphism $t \colon C^1 (\Gamma) \to C_1(\Gamma)$~(\ref{betoeb.e}). 
\end{trivlist}
 \end{enumerate}
\end{definition}
It is clear from the definitions that the complexes $C_\cx(\Gamma)$ that
a morphism of graphs $f \colon \Gamma_1 \to \Gamma_2$ induces morphisms of complexes. 
\[
  C_\cx(f) \colon C_\cx(\Gamma_1) \to C_\cx(\Gamma_2) \quad \text{and} \quad  
 C^\cx(f) \colon C^\cx(\Gamma_2)  \to C^\cx(\Gamma_1),
\]
compatible with $d_1$ and $d^0$.

The proposition below is of course well-known.  We explain it here
because our constructions are somewhat nonstandard.  Statement
(3) explains the relationship between the pairings we have defined
and intersection multiplicities.

\begin{proposition} \label{graph.p}
  Let $\Gamma$ be a finite graph such that $\ep^{-1}(e)$ has cardinality 
two for every $e \in E$.  Let $C_\cx(\Gamma)$ and $C^\cx(\Gamma)$ be the  complexes defined in Definition~\ref{grcomp.d},
 and let $H_*(\Gamma)$ and $H^*(\Gamma)$ the corresponding (co)homology groups. 
For each pair of elements $(v,w)$ in $V$, let 
$$E(v,w) := \ep(\zeta^{-1}(v)) \cap \ep(\zeta^{-1}(w)) \subseteq E $$
and let $ e(v,w)$ be the cardinality of $E(v,w)$.
  \begin{enumerate}
\item The homomorphisms $d_1 \colon C_1(\Gamma) \to C_0(\Gamma)$ 
and $d^0 \colon C^0(\Gamma) \to C^1(\Gamma)$ are adjoints, with respect to the
pairings defined above.
\item 
      The groups $H_*(\Gamma)$ and $H^*(\Gamma)$ are torsion free, and the inner product on $C_1(\Gamma)$ (resp. on $C_0(\Gamma)$) defines a perfect pairing $\angles \ \ $ between  $H^1(\Gamma)$ and $H_1(\Gamma)$  (resp., between $H_0(\Gamma)$ and $H^0(\Gamma)$).  In fact, $H_0(\Gamma)$ identifies with the free abelian group on  $V/\sim$, where $\sim $ is the equivalence relation generated by the set of pairs $(v, v')$ such that $E(v,v') \neq \emptyset$. 
    \item For each $v \in V$,
      \[
        d_1( t (d^0(v))) = \sum_{ v' \neq v}  e(v, v') (v -  v'),\mbox{      and}
      \]
      \[
         \b {d^0(v)}{ d^0({w})}  =
          \begin{cases}
            - e(v, w)& \text{if } v \neq  w \\
 \sum_{ v'\neq v} e(v, v')& \text{if } v =  w.
          \end{cases}
      \]
\item  Let $h^i(\Gamma)$ denote the rank of $H^i(\Gamma)$ and let $\chi(\Gamma) := h^0(\Gamma) - h^1(\Gamma)$.  Then 
\[\chi(\Gamma) = |V|-|E|\].  
  \end{enumerate}
\end{proposition}
\begin{proof}
Statement (1) is clear from the construction, since $d_0$ is dual to
$d^1$ and $\zeta_*$ is dual to $\zeta^*$.  

To prove (2), observe that each equivalence
class of $E$ defines a subgraph of $\Gamma$, that
$\Gamma$ is the disjoint union of these
subgraphs, and that the complex $C_\cx(\Gamma)$
is the direct sum of the corresponding complexes.
Thus we are reduced to proving (2) when there
is only one such equivalence class.   There
is a natural augmentation
$\alpha  \colon \bz^V \to \bz$ sending each
basis vector $v$ to $1$, and if $b \in B$, 
$\alpha(d_1(b -\iota(b))) = \alpha(\zeta(b) - \zeta(\iota(b))) = 0$,
so $d_1$ factors through $\Ker (\alpha)$.  
Thus it will suffice to prove that $d_1$ maps surjectively
to this kernel.  Choose some $v_0 \in V$; then
$\{ v - v_0 : v \in V, v \neq v_0\}$ is a basis for $\Ker(\alpha)$.  
Say  $(v, v')$ is a~pair of  distinct elements of $V$ and $E(v,v') \neq \emptyset$.
Choose $e \in \ep(\zeta^{-1}(v)) \cap \ep (\zeta^{-1}(v'))$ and 
$b \in \ep^{-1}(e) \cap \zeta^{-1}(v)$.  Then necessarily
$\zeta(\iota(b)) = v'$, so $d_1(b - \iota(b)) = v - v'$.  
Since any two elements of $E$ are equivalent, given any $v  \in V$,
there is a sequence $(v_0, v_1, \ldots, v_n)$ with each $v_{i-1} \sim v_i$,
 and  for each such pair choose $b_i$
with $d_1(b_i - \iota(b_i)) = v_i - v_{i-1}$.  
Then $d_1(b_1 + \cdots b_n) = v_n - v_0$.  

It follows that $H_0(\Gamma)$ is torsion free.  Then the duality 
 statement follows from the fact that $d^0$
is dual to $d_1$. 

The formulas for $d_1$ and $d^0$  imply that for $b \in B$ and $v \in V$,
\begin{eqnarray*}
  d_1(d_b) & = & \zeta(b) - \zeta(\iota(b)) \cr
d^0(v) & = & \sum_{b \in \zeta^{-1}(v)} p_b.
\end{eqnarray*}
Hence if $v$ in $V$, 
\begin{eqnarray*}
  d_1 ( t(d^0(v))) &=& d_1\left( \sum_{b \in \zeta^{-1}(v)} d_b\right ) \cr
            & = &  \sum_{b \in \zeta^{-1}(v)} \zeta(b) - \zeta(\iota(b)).
\end{eqnarray*}
But if $b \in \zeta^{-1}(v)$,
\[
  \zeta(b) - \zeta(\iota(b)) = 
    \begin{cases}
      v -\zeta(\iota(b)) &\text{if  $\zeta(b) \neq \zeta(\iota(b))$} \\
      0 &\text{otherwise}.
    \end{cases}
\]
For each $v' \in V \setminus \{v\}$, the map $\ep$ induces a bijection
from $\{ b \in \zeta^{-1}(v) : \zeta(\iota(b)) = v'\}$ to $E(v,v')$.
Thus
\[ \sum_{b \in \zeta^{-1}(v)} \zeta(b) - \zeta(\iota(b)) = \sum_{v' \in V} e(v,v') v - e(v,v')v'  
\]
and the first formula of (3) follows. 
Then 
\[
  \b{d^0v} {d^0w}  = \b{d_1 (t(d^0(v)))} {w} = \sum_{v' \neq v} e(v,v') \b v w -  \sum_{v'  \neq v} e(v,v')\b {v'} w ,\]
and the second formula follows.  Statement (4) is immediate. 
\end{proof}


 The geometric meaning
of the cochain complex of a nodal curve is
straightforward and  well-known.

\begin{proposition}\label{cohgraph.p}
Let  $\und X/\bc$ be a nodal curve and let $\Gamma(\und X)$ be
its dual graph.  Then there is a commutative diagram
\[\xymatrix{
H^0(\und X',\bz) \ar[r]^-A\ar[d]_\cong & H^0(\und X,\bz_{X'/X})\ar[d]^\cong  \cr
C^0(\Gamma(\und X)) \ar[r]_{d^0} & C^1(\Gamma(\und X)),
}\]
where the homomorphism $A$ comes from the map also denoted by $A$ in the short exact sequence:
\begin{equation}\label{zxx.e}
 0 \To \bz_X \To \ep_*(\bz_{X'} ) \rTo A  \bz_{X'/X} \To 0.  
  \end{equation}
Consequently there is an exact sequence:
\begin{equation}\label{hx.e}
    0 \To H^1(\Gamma(\und X)) \To H^1(\und X,\bz) \To H^1(\und X',\bz) \To 0.  
  \end{equation}
\end{proposition}
\begin{proof}
The  commutative diagram is an immediate consequence of the definitions. 
The cohomology sequence
attached to the exact sequence (\ref{zxx.e}) reads:
\[
  0 \To H^0(\und X,\bz) \To H^0(\und X',\bz) \rTo A H^0(\und X,\bz_{X'/X})\To H^1(\und X, \bz) \To H^1(\und X', \bz) \To  0,
\]
and the sequence (\ref{hx.e}) follows immediately.
\end{proof}

Note that $H^1(\und X',\bz)$ vanishes if and only if each irreducible component of $\und X$ is rational, a typical situation.

\subsection{The nearby cycles spectral sequence}

We now consider the associated morphism $f_\lag \colon X_\lag \to S_\lag$. Our goal is to use  the nearby cycles diagram~(\ref{taudiag.e}) and Theorem~\ref{monthm.t} to describe the general fiber $X_\eta$ of $f_\lag$, together with its monodromy action.

\begin{theorem} \label{curvethm1.t}
Let  $f  \colon X \to S$ be a vertical  log curve
over the standard log point $S$.   The morphism 
\[
  f_\lag \colon X_\lag \to  S_\lag = \sone
\]
is a topological submersion whose fibers are
topological manifolds of real dimension $2$.  If $f$ is proper and $\und X$ is connected, then  the morphism $f_\lag$ is a locally trivial fibration, its general
 fiber $X_\eta$ is  compact,  connected, and orientable,  and 
its genus is $1+g( \und X') +    h^0(Y) - h^0(\und X')$.
\end{theorem}
\begin{proof}
The first statement is proved in  \cite{no.rr}, although
it is much more elementary over a log point
as here.   Suppose  $f$ is proper.
Then so is $f_\lag$, and it follows that $X_\eta$ is compact.
Its  orientability is proved in \cite{no.rr}.

To compute the  cohomology of $X_\eta$, 
observe that since the fibration $\tX_\lag \to \br(1)$ is necessarily trivial, 
$X_\eta$ and $\tX_\lag$ have the same homotopy type, 
and in particular their homology groups are isomorphic. 
The  spectral sequence of nearby cycles  for the sheaf $\bz(1)$ on  $\tX_\lag$ reads:
\[
  E_2^{p,q}(1)= H^p(X,\Psi^q_{X/S}(1))   \quad \Rightarrow \quad H^{p+q}(\tilde X_\lag, \bz(1)).
\]
Theorem~\ref{logsymb.t} defines an isomorphism
$\sigma \colon \cM_{X/S} \isomlong \Psi^1_{X/S}(1)$,
and Proposition~\ref{cmup.p}  an isomorphism
$\psi_{X/S}  \colon \cM_{X/S} \isomlong \bz_{Y/Y'}$.     These sheavs are supported on the zero dimensional space $Y$, and $\Psi^q_{X/S}(1)$ vanishes for $q > 1$   Since $X$ has (real) dimension $2$, the only possible nonzero terms and arrows in the spectral sequence are:
\[
  \xymatrix{ 
    \bullet \ar[drr]^{d_2^{0,1}} \\
    \bullet & \bullet & \bullet
  }
\]
Hence $ E^{1,0}_\infty(1) = E_2^{1,0}(1)  = H^1(\und X, \bz(1))$, and there is an exact sequence:
\begin{equation}\label{nearex.e}
 0 \To E^{0,1}_\infty(1) \To H^0(\und X,\bz_{Y/Y'}) \rTo {d_2^{0,1}} H^2(\und X,\bz(1)) \To H^2(\tilde X_\lag, \bz(1)) \To 0.
\end{equation}

Since the normalization map $\ep$ is proper and an isomorphism outside $Y'$, it 
 induces an isomorphism:
\[
  H^2(\und X,\bz(1)) \isomlong  H^2(\und X', \bz(1)).
\]
Since $\und X'$ is a smooth compact complex analytic manifold of dimension $1$,  the trace map induces a canonical isomorphism: $H^2(\und X', \bz(1)) \cong H^0(\und X',\bz)$.  Combining this isomorphism
with the one above, we obtain an isomorphism:
\begin{equation} \label{x'x.e}
 \rm{ tr'} \colon H^2(\und X,\bz(1)) \isomlong H^0(\und X', \bz).
\end{equation}

\begin{lemma}\label{bettitilde.l}
Let $X/S$ be a proper, connected, and vertical log curve over the standard log point,
and let $\und X$ be its underlying nodal  curve.  Then the Betti numbers  of $\Gamma(\und X)$, of $X$,
 and of  the general
fiber $X_\eta$ of the fibration $X_\lag \to \sone$, are given by the following formulas.
\begin{align*}
h^1(\Gamma(\und X)) & =  1 - h^0(\und X') + h^0(Y) \\
  h^1(\und X) & =  h^1(\Gamma{(\und X)}) + h^1(\und X') \\
  h^1(X_\eta) & =  h_1(\Gamma{(\und X)}) + h^1(\und X).
\end{align*}
\end{lemma}

\begin{proof}
  The first formula follows from (4) of Proposition~\ref{graph.p} and the definition
of $\Gamma(\und X)$.  The second formula follows from the exact sequence (\ref{hx.e}).  
For the third formula, observe that $H^2(\tilde X_\lag, \bz(1))$ 
has rank one,  since $\tilde X_\lag$ has the same homotopy type
as $X_\eta$, which is a compact two-dimensional orientable manifold. 
It then  follows from   the exact seqeunce (\ref{nearex.e}) that
the rank $e_\infty^{0,1}(1)$ of $E_\infty^{0,1}(1)$ is given by
\begin{eqnarray*}
  e_\infty^{0,1}(1) &= &h^0(\und X, \bz_{Y/Y'}) - h^2(\und X, \bz(1)) + 1 \cr
    & = & h^0(\und X, \bz_{Y/Y'}) - h^0(\und X',\bz) + 1 \cr
    & = & h^0(\und X, \bz_{Y/Y'}) - h^0(\und X',\bz) + 1 \cr
    & = & |E(\Gamma(\und X)| - |V(\Gamma(\und X)| +1 \cr
     & = & 1- \chi(\Gamma(\und X)) \cr
     & = & h_1(\Gamma(\und X))
\end{eqnarray*} 
Then $h^1(X_\eta) = e_\infty^{0,1}(1)  + e_\infty^{1,0}(1) = h_1(\Gamma(\und X)) + h^1(\und X)$. 
\end{proof}
Combining the formulas of the lemma, we find
\[ h^1(X_\eta) = h_1(\Gamma) + h^1(\Gamma) + h^1(\und X') = 2 - 2h^0(\und X') + 2 h^0(Y) +2g (\und X'),\]
and hence $g(X_\eta) = 1 - h^0(\und X') + h^0(Y) + g(\und X')$. 
\end{proof}

The following more precise result shows that the differential in the nearby spectral
sequence can be identified with the differential in the chain complex $C_\cx$ attached
to the dual graph of $\und X$.  

\begin{proposition} \label{nearbygraph.p}
Let $X/S$ be  a proper and vertical 
 log curve {over the standard log point} and let $\und X$ be its underlying nodal curve.
 Then the following diagram commutes.
\[
  \xymatrix{ 
    &H^0(X, \Psi_{X/S}^1(1)) \ar[r]^-{-d_2^{0,1}} &H^2(\und X,\bz(1))\ar[d]^{\rm tr'}_\cong\cr
    H^0(X,\bz_{Y/Y'})\ar[ru]^\cong\ar[dr]_\cong& \ar[l]_{\psi_{X/S}}^\cong H^0(\und X,\cM_{X/S})\ar[d]^\cong
 \ar[u]^\sigma_\cong\ar[r] &H^0(\und X',\bz)\ar[d]^\cong\cr
    &C_1(\Gamma(\und X)) \ar[r]^{d_1} & C_0(\Gamma(\und X))
  }
\]
Consequently there is a canonical isomorphism $E_\infty^{0,1}(1) \cong H_1(\Gamma(\und X))$ and hence an exact sequence
\begin{equation}
  \label{hgamma.e}
  0 \to H^1(\und X,\bz(1)) \to H^1(\tX,\bz(1)) \to H_1(\Gamma(\und X)) \to 0
\end{equation}
\end{proposition}
\begin{proof}
The commutativity of this diagram follows from Proposition~\ref{cmup.p}
and statement (1) of Theorem~\ref{monthm.t}.
To write  out  the proof in detail, 
we use the notation of the proof of that result. It suffices to check
what happens to each basis element of the free abelian group 
$H^0(X,\cM_{X/S})$. 
Let  $y$ be a point of $Y$ and let $m_1$ and $m_2$ be the elements of $\cM_{X,y}$ as 
in the proof of Proposition~\ref{exdata.p},
with images $\ell_1$ and $\ell_2$ in $\Gamma(X, \cM_{X/S})$. 
Then $\ell_1 = -\ell_2$ is a typical basis element of $H^0(X, \cM_{X/S})$.
Theorem~\ref{monthm.t} says that 
$d_2^{0,1}(\ell_1) $ is the Chern class $c_1(\cLL_{\ell_1})$  of 
$\cLL_{\ell_1}$, where
$\cLL_{\ell_1}$ is the invertible sheaf on $X$ coming from the exact sequence (\ref{xsmp.e}).
Then 
$\ep^*(c_1(\cLL_{\ell_1})) = c_1(\ep^*(\cLL_{\ell_1})) = c_1( \oh {X'}(-\psi(\ell_1))$, 
by Proposition~\ref{cmup.p}.   But 
 if $p$ is a point  of the (smooth) curve $\und X'$, then $\tr (c_1( \oh {X'}(D))$ is the basis  element
of $H^0(\und X',\bz)$ corresponding the connected component of $\und X'$ containing $p$.
 The corresponding  generator  of $C_0(\Gamma)$ is precisely $\zeta(p)$.  This proves that the diagram commutes.
\end{proof}

\subsection{Monodromy and the Picard--Lefschetz formula}

We can now compute the monodromy action on $H^1(\tilde X, \bz)$. 
 \begin{theorem} \label{moncurve.t}
  Let $X/S$ be a log curve {over the standard log point}. Choose $\gamma \in \li \bn  = \bz(1)$,  let $\rho_\gamma$ be the corresponding automorphism of $H^1(\tilde X,\bz)$,
{ and let 
 $N_\gamma \colon E_\infty^{0,1} \to E_\infty^{1,0}$
 be the map induced by $\rho_\gamma-\id$~(see~\ref{Ngamma.e}).}
Let 
\[\kappa'_{X/S} := \kappa_{X/S} \circ \psi_{X/S}^{-1} \colon \bz_{Y/Y'}  \To \cM_{X/S} 
\To \bz_X[1] \]
    Then there is a commutative diagram:
  \[
    \xymatrix{ 
      &H^1(\tX,\bz)\ar[ddl]\ar[d]^\gamma\ar[r]^{\rho_\gamma-\id} & H^1(\tX, \bz) \cr
      &H^1(\tX, \bz(1))\ar[d]\ar[dl]_b &  \cr
      H_1(\Gamma(\und X))\ar[rdd]_i &E^{0,1}_\infty(1) \ar[l]_\cong \ar[dd]^\cong\ar[r]^{N_\gamma}& E^{1,0}_\infty \ar[uu]&  \cr
      & &H^1(\und X,\bz)\ar@/_1.5pc/[uuu]_{c}\ar[u]^\cong & \ar[l]H^1(\Gamma(\und X))\ar[luuu]_a\cr
      & H^0(X,\bz_{Y/Y'}) \ar[r]_{c_{X/S}} \ar[ur]^{\kappa'_{X/S}} & H^0(X,\bz_{Y'/Y})\ar[u]_{\kappa_{A/S}}\ar[ru]_p
    }
  \]
\end{theorem}

\begin{proof}
Applying $H^1$ to the commutative diagram defining $\lambda^1_\gamma$
\[ 
  \xymatrix{
    \Psi_{X/S}\ar[d]_{\rho_\gamma - \id} \ar[r] & \Psi^1_{X/S}[-1] \ar[d]^{\lambda^1_\gamma[-1]} \\
    \Psi_{X/S} & \Psi^0_{X/S} \ar[l]
  }
\]
yields a commutative diagram 
\[
  \xymatrix{
    H^1(\tX, \bz) \ar[d]_{\rho_\gamma - \id} \ar[r] & H^0(\und X, R^1 \ttau_* \bz) \ar[d] \\
    H^1(\tX, \bz) & H^1(\und X, R^0 \ttau_* \bz). \ar[l]
  }
\]
Thanks to  the identifications 
\[H^0(X, R^1 \ttau_* \bz(1)) \cong E^{0, 1}_2(1) \cong H^0(X, \bz_{Y/Y'})\]
 \[H^1(X, R^0\ttau_* \bz) = H^1(\und X, \bz) = E^{1, 0}_\infty,\] 
our monodromy formula from Theorem~\ref{monthm.t}(2) shows that the following diagram commutes
\[ 
  \xymatrix{
    H^1(\tX, \bz) \ar[d]_{\rho_\gamma - \id} \ar[r]^-\gamma & H^1(\tX, \bz(1)) \ar[r] &  H^0(\und X, \bz_{Y/Y'}) \ar[d]^{\kappa'_{X/S}} \\
    H^1(\tX, \bz) & & H^1(\und X, \bz). \ar[ll]^c 
  } 
\]
The rest of the big diagram commutes by the preceding discussion of the dual graph.
\end{proof}

\begin{remark}\label{vcycle.r}
  The dual of the exact sequence (\ref{hgamma.e})
can be written :
\[
  0  \to H_1(\Gamma(\und X))^\vee\to  H_1(\tX, \bz(-1)) \to H_1(\und X,\bz(-1)) \to 0,
\]
so that the  elements of  
$H_1(\Gamma(\und X))^\vee \cong H^1(\Gamma(\und X))$  can be interpreted   as \textit{vanishing cycles} on $\tX$.  The exact sequence~(\ref{hx.e}) shows
that they can also be interpreted as \textit{vanishing cocycles} on $\und X$.
\end{remark}

The  monodromy formula expressed by Theorem~\ref{moncurve.t}
can be made more explicit in terms of vanishing cycles.  For each
node $y \in Y$, choose a branch
$y' \in \ep^{-1}(y)$ and note that $\pm p_{y'}  \in \Gamma(X, \bz_{Y'/Y})$ 
depends only on $y$.  Write
$\angles \ \  $  for the pairing $\bz_{Y/Y'} \times \bz_{Y'/Y} \to \bz$ and let
$h_y \colon \bz_{Y/Y'} \to \bz_{Y'/Y}$ be the map
$\angles \ {p_{y'}} p_{y'}$. Then $h_y$ 
 depends only on $y$ and not on $y'$, 
and, by Proposition~\ref{norm.p}, we can write
\[
  c_{X/S} = \sum_y -\nu(y) h_y = \sum_y -\nu(y)   \angles {\ } {p_{y'}} {p_{y'}}. 
\]
(N.B. The map $c_{X/S}$ above encodes the `monodromy pairing' of Grothendieck, cf. \cite[Exp.\ IX, \S 9 and 12.3]{g.sga7I}). 
Then the composition
\[
  H^1(\tX,\bz) \rTo{b\circ\gamma} H_1(\Gamma) \rTo{p\circ c_{X/S}\circ i}
  H^1(\Gamma) \rTo{a} H^1(\tX,\bz)
\]
is the map sending an element $x$ to
$\sum_y - \nu(y)\angles {b\circ \gamma (x)} {p_y} a(p_y)$. The following  formula
is then  immediate.

\begin{corollary} \label{ourpl.c}
  If $\gamma \in\li P$ and $  x \in H^1(\tilde X,\bz)$,
  \[ 
    \rho_\gamma(x) = x -\sum_y \nu(y)\angles  {b\circ\gamma(x) }  {p_y}a(p_y). \eqno \qed
  \]
\end{corollary} 

When  all $\nu(y) = 1$, the formula of Corollary~\ref{ourpl.c} is the standard Picard--Lefschetz formula~\cite[Exp. XV]{dk.sga7II}.  To verify this, we must check the compatibility of the pairing $\angles \ \ $ used above with the standard pairing on cohomology.  As usual the determination
of signs is delicate; we give a (somewhat heurstic) argument below.

 Recall that we have a proper fibration $\tilde X \to \br(1)$, and hence for all $i$, $H^i(\tilde X, \bz) \cong H^i(\tilde X_0)$ where $\tX_0$ is the fiber of $\tX \to \br_\ge$  over zero (equivalently, the fiber of $X_\lag \to \sone$ over $1$).  Thus we can replace $\tilde X$ by $\tilde X_0$ in the diagrams above.  Since  $\tilde X_0$ is a compact manifold, whose orientation sheaf identifies with $\bz(1)$~\cite{no.rr},  we have a perfect pairing 
\[ 
  \b \ \  \colon H^1(\tilde X_0, \bz(1) ) \times H^1(\tilde X_0, \bz) \to  H^2(\tilde X_0,\bz(1)) \rTo{\tr} \bz,
\]
defined by cup-product and trace map.  For each $y$, let $v_y := a(\delta_y ) \in H^1(\tX_0,\bz)$.  Then the  usual Picard--Lefschetz formula \cite[Exp.~XV, Th\'eor\`eme~3.4]{dk.sga7II} reads:
\begin{equation} \label{usualpl.e}
  \rho_\gamma(x) =         x - \sum_y \nu(y)\b  {\gamma(x) }  {v_y}v_y.
\end{equation}
As we shall see from Proposition~\ref{cupcomp.p} below,  for $x \in H^1(\tilde X_0,\bz(1))$ and $y \in Y$, 
\[ 
  \angles {b(x)}{\delta_y} = \b x {a(\delta_y}).
\]
Thus Corollary~\ref{ourpl.c}  implies the Picard--Lefschetz formula~(\ref{usualpl.e}).

\begin{proposition} \label{cupcomp.p}
  The maps 
  \[
    a \colon  H^1_Y(X,  \bz) \To H^1(\tilde X_0, \bz) \quad \mbox{and} \quad
    b  \colon H^1(\tilde X_0, \bz(1)) \To H^1_Y(X,\bz)
  \]
  of the diagram in Theorem~\ref{moncurve.t} are mutually dual, where we use the standard cup-product and trace map:
  \[
    H^1(\tilde X_0, \bz(1)) \ot H^1(\tilde X_0, \bz) \To H^2(\tilde X_0, \bz(1)) \rTo{\rm tr} \bz 
  \]
  pairing and the form $\b \ \ $ \eqref{grcomp.d}  on  $H^1_Y(X,\bz) \cong C^1(\Gamma)$.  
\end{proposition}

\begin{proof}
We start by reducing to the local case. Since we will have to deal with non-proper $X$, we need to modify the map $a$ slightly,
letting:
\[ 
  a: H^1_Y(X, \bz) \To H^1_c(X, \bz) \To H^1_c(\tilde X_0, \bz),
\]
where the first map is induced by the natural transformation $\Gamma_Y\to \Gamma_c$ (defined because $Y$ is proper), and the other map is pull-back by $\ttau_0:\tilde X_0\to X$ (defined because $\ttau_0$ is proper). Note that $a$ is well-defined in the situation when $X$ is not proper, and that it coincides with $a$ defined previously in case $X$ is proper. Moreover, the map $b$ makes sense for non-proper $X$, and both maps are functorial with respect to (exact) open immersions in the following sense: if $j:U\to X$ is an open immersion, then the following squares commute:
\[ 
  \xymatrix{
    H^1_Y(X, \bz) \ar[r]^a & H^1_c(\tilde X_0, \bz) & & H^1(\tilde X_0, \bz(1)) \ar[r]^b \ar[d]^{j^*} & H^1_Y(X, \bz) \ar[d]^{j^*} \\
    H^1_{Y\cap U}(U, \bz) \ar[r]_a \ar[u]_{j_*} & H^1_c(\tilde U_0, \bz) \ar[u]_{j_*} & & H^1(\tilde U_0, \bz(1)) \ar[r]_b & H^1_{Y\cap U}(U, \bz). 
  }
\]
The two pairings in question are similarly functorial. Recall that 
\[ 
  H^1_Y(X, \bz) = \bigoplus_{y\in Y} H^1_{\{y\}}(X, \bz)
\]
is an orthogonal decomposition; let $a_y$ (resp. $b_y$) be the composition 
\begin{align*}
  a_y&\colon H^1_{\{y\}}(X, \bz) \To H^1_Y(X, \bz)\rTo{a} H^1_c(\tilde X_0, \bz) \\
  (\text{resp. }b_y&\colon H^1(\tilde X_0, \bz(1)) \To H^1_Y(X, \bz)\To H^1_{\{y\}}(X, \bz)).
\end{align*}
To check that $a$ and $b$ are mutually dual, it suffices to check that $a_y$ and $b_y$ are mutually dual for all $y\in Y$. Fix $y\in Y$, and let $U$ be a standard neighborhood of $y$. The functoriality of $a$ and $b$ discussed above implies that it suffices to prove the proposition for $X=U$.

We henceforth assume that $X=\{ (x_1, x_2)\,:\, x_1x_2=0\}$. So $Y=\{y\}$, $y=(0,0)$, and $X=X_1\cup X_2$ where $X_i= \{ x_i = 0 \}$. The choice of ordering of the branches at $y$ yields generators of the three groups in question as follows. First, the class of $X_1$ (treated as a section of $j_*\bz_U$, where $U=X\setminus Y$) gives a generator $u$ of $H^1_Y(X, \bz)$. Second, the loop in the one-point compactification of $\tilde X_0$ going from the point at infinity through $X_2$ and then $X_1$ gives a basis of its fundamental group, and hence a basis element $v$ of $H^1_c(\tilde X_0, \bz)$. Finally, identifying the circle $\tilde Y_0 = \ttau^{-1}_0(y) = \{ (\phi_1, \phi_2)\in\sone\,:\, \phi_1\phi_2 = 1\}$ with the unit circle in $X_1$ via the map $(\phi_1, \phi_2)\mapsto \phi_1$ yields a generator $w$ of $H^1(\tilde X_0, \bz(1)) \cong H^1(\tilde Y_0, \bz(1))$.

The assertion of the proposition will now follow from the three claims below:
\begin{enumerate} 
  \item $a(u) = v$,
  \item $b(w) = -u$,
  \item $\langle v,w \rangle = 1$. 
\end{enumerate}

To check the first claim, note that we have a similarly defined basis element $v'$ of $H^1_c(X, \bz)$ which pulls back to $v$. Let $\gamma:\br\cup\{\infty\}\to X\cup \{\infty\}$ be a loop representing $v'$, sending $0$ to $y$. Pull-back via $\gamma$ reduces the question to Lemma~\ref{lemma:real-axis} below.

For the second claim, recall first that $c'(u)=c'([X_1])=[q_1]$. Second, the isomorphism $\sigma:\cM_{X/S,y}^\g\to H^1(\tilde Y_0, \bz(1))$ sends $q_i$ to the pullback by $\phi_i$ of the canonical class $\theta \in H^1(\sone, \bz(1))$. On the other hand, since $x_2$ is the coordinate on $X_1$, $v = \phi_2^*\theta$. Since $\phi_1\phi_2 = 1$ on $\tilde X_0$, $\phi_1^* + \phi_2^* = 0$, and hence $b^{-1}(u) = \sigma(c'(u)) = \phi_1^* \theta = -\phi_2^* \theta = -w$. 

For the last claim, we note that the map
\[ 
    (r_1, \phi_1, r_2, \phi_2)\mapsto (r_1 - r_2, \phi_1): \tilde X_0 \to \br \times \sone
\]
is an orientation-preserving homeomorphism (where the orientation sheaves of both source and target are identified with $\bz(1)$). Under this identification, $w$ corresponds to the loop $0\times \sone$ (positively oriented), and $v$ correspond to the `loop' $\br \times \{1\}$ oriented in the positive direction. These meet transversely at one point $(0, 1)$, and their tangent vectors form a negatively oriented basis at that point, thus $\langle w, v\rangle = -1$.
\end{proof}

\begin{lemma} \label{lemma:real-axis}
Let $S=\br\cup \{\infty\}$ be the compactified real line, $Y=\{0\}$, $Z=\{\infty\}$, $X=\br = S\setminus Z$, $U=X\setminus Y$, $j:U\hookrightarrow X$. Let $e\in H^0(U, \bz)$ equal $1$ on $U_+ = (0, \infty)$ and $0$ on $U_- = (-\infty, 0)$. As before, we have a short exact sequence
\[ 
  0 \To \bz_X \To j_* \bz_U \To \mathcal{H}^1_Y(\bz_X) \To 0
\]
and hence an identification $H^1_Y(X, \bz) \cong H^0(X, \mathcal{H}^1_Y(\bz_X)) \cong H^0(U, \bz)/j^* H^0(X, \bz)$. The element $e$ thus gives a basis element $u$ of $H^1_Y(X, \bz)$. The orientation of the real axis gives a basis element of $\pi_1(S, \infty)$, and hence a basis element $v$ of $\Hom(\pi_1(S, \infty), \bz) \cong H^1(S, \bz) \cong H^1_c(X, \bz)$. Then the natural map $H^1_Y(X, \bz)\to H^1_c(X, \bz)$ sends $u$ to $v$.
\end{lemma}

\begin{proof}
By  \cite [Cycle 1.1.5, p. 132]{sga4.5}, $u$ corresponds to the partially trivialized $\bz_X$-torsor $(\bz_X, -e)$ (cf. Remark~\ref{signissue.r} and \cite[Cycle 1.1.4--5]{sga4.5}). Let $(\mathcal{F}, f)$ be a $\bz_S$-torsor with a section $f\in H^0(\mathcal{F}, S\setminus Y)$ such that there exists an isomorphism $\iota:\mathcal{F}|_X \cong \bz_X$ identifying $f|_{X\setminus Y}$ with $-e$. Then the class $[\mathcal{F}]$ of $\mathcal{F}$ in $H^1(S, \bz)=H^1_c(X, \bz)$ is the image of $u$. The image of $0$ under the isomorphism $\iota$ yields a trivializing section $g$ of $\mathcal{F}|_{X}$, and $f$ is a trivializing section of $\mathcal{F}|_{S\setminus Y}$. On the intersection $X\cap (S\setminus Y) = U$, we have $f-g = 0 - e$; thus $f$ is identified with $g$ on $U_-$, and $g$ is identified with  $f + 1$ on $U_+$. So the positively oriented loop has monodromy $+1$ on $\mathcal{F}$, i.e., $[\mathcal{F}]=v$ as desired.
\end{proof}
  
\bibliographystyle{plain} 
\bibliography{all,log,ogus}

\end{document}